\documentclass[reqno,11pt]{amsart} 
\usepackage{amssymb,amsthm,amsmath,mathrsfs,graphics,amsfonts,bbm}

\usepackage{bookmark}

\usepackage[left=1.2in, right=1.2in, bottom=1.5in]{geometry}
\newtheorem{theorem}{Theorem}[section]
\newtheorem{lemma}[theorem]{Lemma}

\newtheorem{corollary}[theorem]{Corollary}
\theoremstyle{definition}
\newtheorem{definition}[theorem]{Definition}

\theoremstyle{remark}
\newtheorem{remark}[theorem]{Remark}

\numberwithin{equation}{section}

\newcommand{\norm}[1]{\lVert#1\rVert}
\newcommand{\cL}{\mathcal{L}}
\newcommand{\R}{\mathbb{R}}

\newcommand{\e}{\varepsilon}
\newcommand{\tdiv}{\text{div}}
\newcommand{\rdiv}{\text{\rm div}}
\newcommand{\rtxt}[1]{\text{\rm #1}}
\newcommand{\osc}[2]{\underset{#1}{\rtxt{osc}}\, [#2]}

\def\Xint#1{\mathchoice
  {\XXint\displaystyle\textstyle{#1}}%
  {\XXint\textstyle\scriptstyle{#1}}%
  {\XXint\scriptstyle\scriptscriptstyle{#1}}%
  {\XXint\scriptscriptstyle\scriptscriptstyle{#1}}%
  \!\int}
\def\XXint#1#2#3{{\setbox0=\hbox{$#1{#2#3}{\int}$}
  \vcenter{\hbox{$#2#3$}}\kern-.5\wd0}}

\def\average{\Xint-}
\begin{document}

\title{Periodic homogenization of Green's functions \\for Stokes systems}



\author{Shu Gu\qquad\qquad Jinping Zhuge}






\begin{abstract}

This paper is devoted to establishing the uniform estimates and asymptotic behaviors of the Green's functions $(G_\e,\Pi_\e)$ (and fundamental solutions $(\Gamma_\e, Q_\e)$) for the Stokes system with periodically oscillating coefficients (including a system of linear incompressible elasticity). Particular emphasis will be placed on the new oscillation estimates for the pressure component $\Pi_\e$.  Also, for the first time we prove the \textit{adjustable} uniform estimates (i.e., Lipschitz estimate for velocity and oscillation estimate for pressure) by making full use of the Green's functions. Via these estimates, we establish the asymptotic expansions of $G_\e,\nabla G_\e, \Pi_\e$ and more, with a tiny loss on the errors. Some estimates obtained in this paper are new even for Stokes system with constant coefficients, and possess potential applications in homogenization of Stokes or elasticity system.

\end{abstract}
\maketitle

\pagestyle{plain}
\tableofcontents

\section{Introduction and Main Results}

\setcounter{equation}{0}
\quad\  The primary purpose of this paper is to study the asymptotic behavior of the Green's functions and their derivatives for Stokes systems with rapidly oscillating periodic coefficients. Precisely, we consider the following Dirichlet problem for Stokes systems in a bounded domain $\Omega\subset\mathbb{R}^d,\, d\ge 3$,
\begin{equation}\label{def.Dirichlet.Stokes}
\left\{
\begin{aligned}
\mathcal{L}_\varepsilon(u_\varepsilon)+\nabla p_\varepsilon &= F + \text{div}(h) &\qquad \text{ in }\Omega,\\
\text{div}(u_\varepsilon) &=g &\qquad\text{ in }\Omega,\\
u_\varepsilon& =f &\qquad\text{ on }\partial\Omega,
\end{aligned}
\right.	
\end{equation}
with the compatibility condition
\begin{equation}\label{cond.compatibility}
\int_\Omega g -\int_{\partial \Omega}f\cdot n=0,
\end{equation}
where $n$ denotes the outward unit normal to $\partial \Omega$. The elliptic operator $\mathcal{L}_\varepsilon$ is defined by
\begin{equation}\label{def.operator}
\mathcal{L}_\varepsilon= -\text{div}(A(x/\varepsilon)\nabla)=-\frac{\partial}{\partial x_i}\bigg[a_{ij}^{\alpha\beta}\Big(\frac{x}{\varepsilon}\Big)\frac{\partial}{\partial x_j}\bigg],
\end{equation} where $1\le i,j,\alpha,\beta \le d$ (the Einstein's summation convention is used throughout) and $\e>0$ is assumed to be a small parameter. We assume that the coefficient matrix  $A(y)=(a_{ij}^{\alpha\beta}(y))$ is real and satisfies the following conditions:
\begin{itemize}
	\item Strong ellipticity: there exists some $\mu>0$ such that
	\begin{equation}\label{cond.ellipticity}
	\mu |\xi|^2 \le a_{ij}^{\alpha\beta}(y)\xi_i^\alpha\xi_j^\beta \le \frac{1}{\mu} |\xi|^2 \qquad \text{for }y\in \mathbb{R}^d \text{ and }\xi=(\xi_i^\alpha) \in \mathbb{R}^{d\times d}.
	\end{equation}
	
	\item Periodicity:
	\begin{equation}\label{cond.periodicity}
	A(y+z)=A(y) \qquad \text{ for }y\in \mathbb{R}^d\text{ and }z\in \mathbb{Z}^d.
	\end{equation}
	
	\item H\"{o}lder continuity: there exist $\lambda\in(0,1] $ and $\tau\ge 0$ such that
	\begin{equation}\label{cond.holder}
	|A(x)-A(y)|\le \tau|x-y|^\lambda.
	\end{equation}
\end{itemize}

The homogenization of system (\ref{def.Dirichlet.Stokes}) has been introduced and studied in  the remarkable monographs \cite{Lions} and \cite{JikovKozlovOleinik94}. The mechanics in the model (\ref{def.Dirichlet.Stokes}) may be interpreted as an approximation of the stationary Newtonian flow in a porous medium (e.g., sponge), with the understanding that viscosity varies spatially at $\e$-scale as the fluid contacts and infiltrates the medium.\footnote{An alternative model for such medium is the flow in perforated domains.} In this case, the vector function $u_\e = (u_\e^1,u_\e^2,\cdots, u_\e^d)$ is the velocity field in a fixed material body $\Omega$ and the scalar function $p_\e$ is the pressure. 

Another more practical problem related to (\ref{def.Dirichlet.Stokes}) is the linearized incompressible elasticity for composite materials (e.g., rubber), which can be described as (see \cite{AllaireGhoshVanninathan17,Briane03,Gurtin})
\begin{equation}\label{def.elasticity}
	\left\{
	\begin{aligned}
		-\text{div}\big( B(x/\e) E(v_\e) \big)+\nabla p_\varepsilon &= F &\qquad \text{ in }\Omega,\\
		\text{div}(v_\varepsilon) &=0 &\qquad\text{ in }\Omega,\\
		v_\varepsilon& =f &\qquad\text{ on }\partial\Omega,
	\end{aligned}
	\right.	
\end{equation}
where $v_\e$ is the displacement, $E(v_\e):=\frac{1}{2} (\nabla v_\e + (\nabla v_\e)^T )$ is the infinitesimal strain tensor and $p_\e$ is the effective pressure. 
Instead of the strong ellipticity condition (\ref{cond.ellipticity}), the coefficient tensor $B(y) = (b_{ij}^{\alpha\beta}(y))$ satisfies the elasticity condition
\begin{equation}\label{cond.elasticity}
	\left\{
	\begin{aligned}
		& b_{ij}^{\alpha\beta} = b_{ji}^{\beta\alpha} = b_{\alpha j}^{i\beta} \\
		& \mu |\xi|^2 \le b_{ij}^{\alpha\beta}(y)\xi_i^\alpha\xi_j^\beta \le \frac{1}{\mu} |\xi|^2 \qquad \text{for }y\in \mathbb{R}^d \text{ and symmetric }\xi=(\xi_i^\alpha) \in \mathbb{R}^{d\times d}.
	\end{aligned}
	\right.	
\end{equation}
For Dirichlet problem, by a well-known trick (see, e.g., \cite{ShenLec17}), system (\ref{def.elasticity}) can be reduced to (\ref{def.Dirichlet.Stokes}) without changing the solution. Actually, if we define $\widetilde{B} = (\tilde{b}_{ij}^{\alpha\beta})$ and
\begin{equation}
	\tilde{b}_{ij}^{\alpha\beta} = b_{ij}^{\alpha\beta} + \frac{\mu}{2}\delta_{i\alpha} \delta_{j\beta} - \frac{\mu}{2} \delta_{i\beta}\delta_{j\alpha},
\end{equation}
where $\delta_{i\alpha}$ is the Kronecker delta, then the solution of (\ref{def.elasticity}), $(v_\e, p_\e)$, satisfies
\begin{equation}
	-\text{div}\big( \widetilde{B}(x/\e) \nabla v_\e(x) \big)+\nabla p_\varepsilon(x) = F(x),
\end{equation}
and $\widetilde{B}$ satisfies the strong ellipticity condition (\ref{cond.ellipticity}) (possibly with a different ellipticity constant). Note that the above reduction will not change the (linearized) incompressibility condition and Dirichlet boundary value. Therefore, system (\ref{def.elasticity}) is reduced to (\ref{def.Dirichlet.Stokes}) and we only need to focus on the latter.

Recently, notable progress has been made towards the theory of convergence rates and uniform regularity in homogenization of Stokes system (\ref{def.Dirichlet.Stokes}); see \cite{GuShen15,Gu16,GuXu17,Xu17,AllaireGhoshVanninathan17,Briane03}. In the present paper, among others, we are particularly interested in the asymptotic behavior of the Green's functions and their derivatives for the Stokes systems. It is well-known that the estimates of Green's functions (or fundamental solutions) is a central problem in partial differential equations, as many properties of the solutions can be derived essentially from the Green's functions. Historically, the asymptotic expansions of the Green's functions (and fundamental solutions) for elliptic systems with $\cL_\e$ have been studied comprehensively. In \cite{Sevostjanova82} and \cite{terElstRobinsonSikora01}, the method of Bloch waves was used to study the asymptotic expansions of the fundamental solutions and heat kernels, respectively. In \cite{AL91}, the asymptotic expansions of the fundamental solutions for elliptic operator $\cL_\e$ were obtained via the uniform regularity theory established in \cite{AL8701,AL89}; and most recently, the results were extended to the higher order elliptic systems in \cite{NiuShenXu17} and to parabolic equations in \cite{GengShen1702}. For elliptic systems in a bounded domain, the asymptotic expansion for the Poisson kernel was obtained in \cite{AL8902}, using the Dirichlet correctors. Recently, C. E. Kenig, F. Lin and Z. Shen carried out a comprehensive study in \cite{KenigLinShen14} on the asymptotic expansions for both the Green's functions and Neumann functions. These motivate us to investigate the Green's functions for the Stokes systems.

To describe our main theorem, we introduce the definition of the Green's functions, which was proposed, for example, in \cite{ChoiLee15} for Stokes system with variable coefficients (adapted in our situation with $\e$).

\begin{definition}\label{def.GreenFunc}
	We call a pair $(G_\e(x,y), \Pi_\e(x,y)) = (G_\e^{\beta}(x,y), \Pi_\e^\beta(x,y))_{1\le \beta\le d}$ the Green's functions for Stokes system (\ref{def.Dirichlet.Stokes}), if it satisfies the following properties:
	
	\begin{enumerate}
	\item[(i)] For each $y \in \Omega$, $G_\e(\cdot,y) \in W^{1,1}_0(\Omega;\R^{d\times d}) \cap H^2_{\rtxt{loc}}(\Omega\setminus\{y\}; \R^{d\times d})$ and $\Pi_\e(\cdot,y) \in L_0^1(\Omega;\R^d)\cap L^2_{\rtxt{loc}}(\Omega\setminus\{y\};\R^d)$.
	
	\item[(ii)] For each $y$, $(G_\e(\cdot,y), \Pi_\e(\cdot,y))$ satisfies the following system in the sense of distribution,
	\begin{equation}\label{def.Green}
	\left\{
	\begin{aligned}
	\cL_\e(G_\e(\cdot,y))+\nabla \Pi_\e(\cdot,y) &= \delta_y I \qquad &\text{ in }& \Omega,\\
	\text{div}(G_\e(\cdot,y)) &= 0 \qquad & \text{ in }& \Omega,\\
	G_\e(\cdot,y)& = 0 \qquad & \text{ on }&\partial\Omega.
	\end{aligned}
	\right.	
	\end{equation}
	
	\item[(iii)] If $(u_\e,p_\e)\in H^1_0(\Omega;\R^d)\times L^2_0(\Omega)$ is the weak solution of
	\begin{equation}
	\left\{
	\begin{aligned}
	\cL_\e^*(u_\e)+\nabla p_\e &= F + \rtxt{div}(h) \qquad &\text{ in }& \Omega,\\
	\text{div}(u_\e) &= g \qquad & \text{ in }& \Omega,\\
	u_\e & = 0 \qquad & \text{ on }&\partial\Omega,
	\end{aligned}
	\right.	
	\end{equation}
	with smooth $F,h$ and $g$ satisfying $\int_{\Omega} g = 0$, then
	\begin{equation}\label{eq.ueLe*}
	u_\e(x) = \int_{\Omega} G_\e(y,x)^T F(y) dy - \int_{\Omega} \nabla G_\e(y,x)^T h(y) dy - \int_{\Omega} \Pi_\e(x,y) g(y) dy.
	\end{equation}
\end{enumerate}
\end{definition}
Similarly, we can define the adjoint Green's functions, denoted by $(G_\e^*(x,y),\Pi_\e^*(x,y))$, for the adjoint Stokes system (replacing $\cL_\e$ by $\cL_\e^*$ in (\ref{def.Dirichlet.Stokes})). We also use $(G_0(x,y), \Pi_0(x,y))$ to denote the Green's functions for the homogenized Stokes system; see (\ref{def.Dirichlet.Stokes0}).

Let $P_j^\beta(x)=x_j e^\beta$ for $1\le j,\beta\le d$ and $e^\beta$ be the $\beta$th unit coordinate vector. We define the Dirichlet correctors $(\Phi_{\varepsilon,j}^\beta,\Lambda_{\varepsilon,j}^\beta)$ as the solution of the following system
\begin{equation}\label{def.Dirichlet-corrector}
\left\{
\begin{aligned}
\mathcal{L}_\varepsilon(\Phi_{\varepsilon,j}^\beta)+\nabla \Lambda_{\varepsilon,j}^\beta &=0  &\qquad \text{ in }\Omega,\\
\text{div}(\Phi_{\varepsilon,j}^\beta) &=\text{div}(P_j^\beta) &\qquad\text{ in }\Omega,\\
\Phi_{\varepsilon,j}^\beta & =P_j^\beta &\qquad\text{ on }\partial\Omega.
\end{aligned}
\right.	
\end{equation}
Recall that the Dirichlet correctors were first introduced in \cite{AL8701} for elliptic systems as a replacement of the usual correctors (see (\ref{def.corrector})), in order to modify the boundary effect. Our Dirichlet correctors invented for Stokes systems serve the same role.

The following are the main results of the paper.

\begin{theorem}\label{thm.main}
	Let $A$ satisfy (\ref{cond.ellipticity}), (\ref{cond.periodicity}) and (\ref{cond.holder}). Then
	\begin{enumerate}
	\item[(i)] If $\Omega$ is a bounded $C^{1,1}$ domain, then for any $x,y\in \Omega$,
	\begin{equation}\label{est.mainG}
	|G_\e(x,y) - G_0(x,y)| \le \frac{C\e}{|x-y|^{d-1}}.
	\end{equation}
	
	\item[(ii)] If $\Omega$ is a bounded $C^{2,\eta}$ domain for some $\eta\in (0,1)$, then for any $x,y\in \Omega$,
	\begin{equation}\label{est.mainDG}
	\begin{aligned}
	\left|\frac{\partial}{\partial x_i}\{G_\varepsilon^{\alpha\beta}(x,y)\}-\frac{\partial}{\partial x_i}\{\Phi_{\varepsilon,j}^{\alpha\gamma}(x)\}\cdot\frac{\partial}{\partial x_j}\{G_0^{\gamma\beta}(x,y)\}\right| \qquad \qquad \quad \\
	\le \frac{C\varepsilon(\ln[\varepsilon^{-1}|x-y|+2])^2}{|x-y|^d},
	\end{aligned}
	\end{equation}
	and for any $x,z,y\in \Omega$,
	\begin{equation}\label{est.mainPi}
	\begin{aligned}
	&\bigg|\Big[\Pi^\beta_\varepsilon(x,y)-\Pi^\beta_0(x,y)-\Lambda_{\varepsilon,j}^\gamma(x)\frac{\partial}{\partial x_j}\{G^{\gamma\beta}_0(x,y)\}\Big]\\
	&\qquad\qquad-\Big[\Pi^\beta_\varepsilon(z,y)-\Pi^\beta_0(z,y)-\Lambda_{\varepsilon,j}^\gamma(z)\frac{\partial}{\partial z_j}\{G^{\gamma\beta}_0(z,y)\}\Big]\bigg|\\
	&\qquad\le \frac{C\varepsilon(\ln[\varepsilon^{-1}|x-y|+2])^2}{|x-y|^d}+\frac{C\varepsilon(\ln[\varepsilon^{-1}|z-y|+2])^2}{|z-y|^d},
	\end{aligned}
	\end{equation}
	where $\Lambda_{\varepsilon,j}^\gamma$ is uniquely specified by $\Lambda_{\varepsilon,j}^\gamma(x_0) = \chi_j^\gamma(x_0/\e)$ for some fixed $x_0\in \Omega$ (see Lemma \ref{lem.Dirichlet-corrector.estimate}). The constant $C$ depends only on $d,m,A$ and $\Omega$.
	\end{enumerate}

\end{theorem}

In the above theorem, estimates (\ref{est.mainG}) and (\ref{est.mainDG}) should be compared to the corresponding estimates for the Green's functions of elliptic system in \cite[Theorem 1.1]{KenigLinShen14}, whereas estimate (\ref{est.mainPi}) is completely new and exhibits distinct feature of the Stokes system. It is worth noting that the left-hand side of (\ref{est.mainPi}) adopts a form of difference in the first variable. In addition, by integrating (\ref{est.mainPi}) with respect to $z$ and using the normalized assumption $\int_{\Omega} \Pi_\e(z,y)\, dz = 0$, one can show an expansion in non-difference form with slightly worse error (see Corollary \ref{coro.PiExp}))
\begin{equation}\label{est.mainPi1}
\begin{aligned}
\bigg|\Pi_\varepsilon(x,y)-\Pi_0(x,y)-\Big[\Lambda_{\varepsilon,j}^\beta(x)\frac{\partial}{\partial x_j}\{G^{\beta}_0(x,y)\}&-\average_{\Omega}\Lambda_{\varepsilon,j}^\beta(z)\frac{\partial}{\partial z_j}\{G^{\beta}_0(z,y)\} dz\Big]\bigg|\\
&\le \frac{C\varepsilon(\ln[\varepsilon^{-1}|x-y|+2])^3}{|x-y|^d}.
\end{aligned}
\end{equation}

We point out in advance that the proof of (\ref{est.mainG}) follows the same idea in \cite{KenigLinShen14}; yet a very different and more difficult part for Stokes systems is the strategy to deal with (\ref{est.mainDG}) and (\ref{est.mainPi}), beacuse the pressure is inevitably involved. Following the proof of Theorem \ref{thm.main}, the asymptotic expansions for $\nabla_y \nabla_x G_\e(x,y)$ and $\nabla_y \Pi_\e(x,y)$ are established in Theorem \ref{thm.asymp-expansion.ddGe}, and the corresponding estimates and expansions for the fundamental solutions $(\Gamma_\e,Q_\e)$ of Stokes systems are presented in Theorem \ref{thm.FundSol} and Theorem \ref{thm.Fund.Exp}.

As aforementioned, the estimates of the Green's functions (or fundamental solutions) have various crucial applications. A typical application of Theorem \ref{thm.main} is the rates of convergence in homogenization theory. Let $(u_\e,p_\e)$ be the solution of (\ref{def.Dirichlet.Stokes}) with $h = 0, \, g = 0$ and $f = 0$ and $(u_0,p_0)$ be the corresponding homogenized solution, then it follows from (\ref{est.mainG}), (\ref{est.mainDG}),  (\ref{est.mainPi1}) and the solution representation that
\begin{equation*}
\norm{u_\e - u_0}_{L^q(\Omega)} \le C \e \norm{F}_{L^p(\Omega)},
\end{equation*}
for $1\le p<d$ and $1/q=1/p-1/d$, or $p>d$ and $q=\infty$, and
\begin{equation*}
\begin{aligned}
\norm{u_\e - u_0 - (\Phi_{\e,j}^\beta - P_j^\beta ) \frac{\partial u_0^\beta}{\partial x_j}}_{W^{1,p}_0(\Omega)} 
& + \norm{p_\e - p_0 - \Lambda_{\e,j}^\beta \frac{\partial u_0^\beta}{\partial x_j}}_{L^p_0(\Omega)} \\
& \qquad \le C\e (\ln(\e^{-1} R_0 + 2))^{8|\frac{1}{2}-\frac{1}{p}|} \norm{F}_{L^p(\Omega)},
\end{aligned}
\end{equation*}
for any $1\le p\le \infty$. See Theorem \ref{thm.convergence-rates.Lq} and Theorem \ref{thm.convergence-rates.W1p} for more details. 

Beyond the straightforward application in the rates of convergence, many other significant applications may be found in the literature, such as $L^p$ boundedness of Riesz transform \cite{AL91,terElstRobinsonSikora01}, asymptotic expansion of the Dirichlet-to-Neumann map \cite{KenigLinShen14}, layer potential methods \cite{KenigShen1101}, quantitative analysis of boundary layers \cite{Gerard-VaretMasmoudi12,ArmstrongKuusiMourratPrange17,ShenZhuge1702,Zhuge16}, etc. We plan to conduct research on some of these topics in other lines in the future.

We now describe the outline of the proof of Theorem \ref{thm.main} and explain some new ideas in the proof. As we know, the existence and estimates of the Green's functions essentially rely on the uniform regularity estimates for the solutions of (\ref{def.Dirichlet.Stokes}). We mention that the interior uniform Lipschitz estimate for velocity and $L^\infty$ estimate for pressure were established by the first author of this paper and Z. Shen in \cite{GuShen15}. More recently, the boundary uniform estimates were obtained by the first author of this paper and Q. Xu in \cite{GuXu17} with data in certain spaces. Precisely, they proved that for any $x\in \Omega$,
\begin{equation}\label{est.oldunif}
\begin{aligned}
&|\nabla u_\e(x)| + \bigg|p_\e(x) - \average_{D_1(x)} p_\e(y) dy \bigg| \\
&\qquad \le C\Big\{ \norm{u_\e}_{L^2(D_2(x))} + \norm{h}_{W^{1,p}(D_2(x))} + \norm{g}_{W^{1,p}(D_2(x))} + \norm{f}_{C^{1,\eta}(\Delta_2(x))} \Big\},
\end{aligned}
\end{equation}
where $D_r(x) = B_r(x) \cap \Omega, \Delta_r = B_r(x) \cap \partial\Omega$ and $\average_E$ denotes the average integral over $E$. In this paper, we first improve the estimate above by assuming the minimal regularity for the data (which is required even for the systems with constant coefficients), i.e., $F \in L^p(\Omega;\R^d)$ for some $p>d$, $h\in C^{0,\eta}(\Omega;\R^{d\times d}), g\in C^{0,\eta}(\Omega)$  and $f\in C^{1,\eta}(\partial\Omega;\R^d)$ for some $\eta \in (0,1)$. Actually, in Theorem \ref{thm.boundary-estimate}, we show that for any $0<r\le \text{diam}(\Omega)$,
\begin{align}\label{est.unif}
\begin{aligned}
&\norm{\nabla u_\e}_{L^\infty(D_r)} + \osc{D_r}{p_\e}\\
&\qquad \le C \bigg\{ \frac{1}{r} \bigg( \average_{D_{4r} } |u_\e|^2 \bigg)^{1/2} + r\bigg( \average_{D_{4r} } |F|^p \bigg)^{1/p}+  r^\eta [h]_{C^{0,\eta}(D_{4r})}+ \norm{g}_{L^\infty(D_{4r})}\\ 
&\qquad \qquad   +  r^\eta [g]_{C^{0,\eta}(D_{4r})}+ \frac{1}{r}\norm{f}_{L^\infty(\Delta_{4r})} + \norm{\nabla_{\rtxt{tan}} f}_{L^\infty(\Delta_{4r})} + r^\eta [\nabla_{\rtxt{tan}} f]_{C^{0,\eta}(\Delta_{4r})} \bigg\},
\end{aligned}
\end{align}
where $\text{osc}_E[v] = \text{esssup}_{x,y\in E}|v(x) - v(y)| $ denotes the maximal oscillation of $v$ on the set $E$. We point out that estimate (\ref{est.unif}) is an improved version of (\ref{est.oldunif}). Apart from the obvious improvement on the regularity of $h$ and $g$, the novelty of (\ref{est.unif}) is the uniform oscillation estimate of $p_\e$, which has apparent meaning in physics, i.e., the pressure difference. More importantly, it allows us to control the pressure difference for any two points in the material body $\Omega$ by connecting the points through a sequence of balls (avoiding the singular point when dealing with the Green's functions). This new idea of concerning the pressure difference via uniform oscillation estimate plays a crucial role in several places of this paper and goes a long way to explain why the right scheme of the asymptotic expansion (\ref{est.mainPi}) should be in form of difference. We mention that the estimate of velocity $u_\e$ in (\ref{est.unif}) follows from the line of \cite{Shen17}, while the oscillation estimate of $p_\e$ absorbs some useful ideas in \cite{GuXu17}.

The uniform estimate (\ref{est.unif}), together with the result of \cite{ChoiLee15}, guarantees the existence of the Green's functions in a bounded $C^{1,\eta}$ domain. Meanwhile, by the same method for elliptic systems, it is not hard to establish the uniform global pointwise estimates for $G_\e(x,y)$ and its derivatives, which are exactly the same as the Green's functions of elliptic systems; see Theorem \ref{thm.green-function.pointwise}, (i) - (iii). On the other hand, by the new idea mentioned before regarding the pressure difference, we are able to show the uniform oscillation estimate for $\Pi_\e(x,y)$, i.e., for any $y\in \Omega, r>0$
\begin{equation}\label{est.PieOsc}
\osc{\Omega \setminus B(y,r)}{\Pi_\e(\cdot,y)} \le C\min\bigg\{\frac{\delta(y)}{r^{d}},\, \frac{1}{r^{d-1}} \bigg\},
\end{equation}
where $\delta(y) = \text{dist}(y,\partial\Omega)$. Obviously, the above estimate can be written equivalently in a form of difference, which can be compared, as we expected, with the estimate of $\nabla_x G_\e(x,y)$. Similarly, we also prove that for any $y\in \Omega, r>0$
\begin{equation}\label{est.DPieOsc}
\osc{\Omega \setminus B(y,r)}{\nabla_y \Pi_\e(\cdot,y)} \le \frac{C}{r^d},
\end{equation}
which can be compared with the estimate of $\nabla_x \nabla_y G_\e(x,y)$. We point out that there is no similar estimate for $\nabla_x \Pi_\e(x,y)$, since it behaves more like $\nabla_x^2 G_\e(x,y)$, which does not possess any uniform estimate. As far as we know, (\ref{est.PieOsc}) and (\ref{est.DPieOsc}) are new and will play a significant role in the proof of our main theorem.

Now we would like to describe the core ideas of the proof of Theorem \ref{thm.main}. The estimate (\ref{est.mainG}) is obtained through the Miranda-Agmon maximum principle and a similar approach of \cite{KenigLinShen14} with substantial modifications according to the Stokes systems. The Miranda-Agmon maximum principle for Stokes system is established in Theorem \ref{thm.Lpqr} with the help of the uniform estimates for the Green's functions. However, the proof of (\ref{est.mainDG}) and (\ref{est.mainPi}) is much more involved. The key of the proof is the following adjustable uniform estimates, which seems new even for elliptic systems with constant coefficients.

\begin{theorem}\label{thm.modified-boundary-estimate}
	Assume $\Omega$ is a bounded $C^{1,\eta}$ domain. Let $(u_\varepsilon,p_\varepsilon)$ be a solution of 
	\begin{equation}\label{def.general-system-e-4r.f=g=0}
	\left\{
	\begin{aligned}
	\cL_\e(u_\e)+\nabla p_\e &= F+\rdiv{(h)} &\qquad \text{ in }D_{5r},\\
	\rdiv(u_\e) &=0 &\qquad\text{ in }D_{5r},\\
	u_\e& =0 &\qquad\text{ on }\Delta_{5r}.
	\end{aligned}
	\right.	
	\end{equation}
	Then for any $0<t\le r \ (\le \rtxt{diam}(\Omega))$,
	\begin{equation}\label{ineq.modified-boundary-estimate}
	\begin{aligned}
	\norm{\nabla u_\e}_{L^\infty(D_r)} +\osc{D_{r}}{p_\varepsilon} & \le C \bigg\{ \frac{1}{r} \average_{D_{5r} } |u_\e| +\ln[t^{-1}r+2]\|\mathcal{M}_{D_{5r},t}(F\delta)\|_{L^\infty(D_{5r})} \\  
	&\quad +t\bigg(\average_{D_{5r}}|F|^p\bigg)^{1/p}+\ln[t^{-1}r+2]\|h\|_{L^\infty(D_{5r})}+t^{\eta}[h]_{C^{0,\eta}(D_{5r})}\bigg\},
	\end{aligned}
	\end{equation}
	where $\delta(x)=\text{\rm dist}(x,\partial D_{5r})$, and $\mathcal{M}_{E,t}(\varphi)$ is the truncated Hardy-Littlewood maximal function defined by
	\begin{equation}\label{def.truncated-Hardy-Littlewood}
	\mathcal{M}_{E,t}(\varphi)(x)=\sup_{s>t}\average_{B(x,s)\cap E} |\varphi|,
	\end{equation}
	and the constant $C$ above depends only on  $d, \,\eta,\, A$ and $\Omega$.
\end{theorem}

We call (\ref{ineq.modified-boundary-estimate}) the adjustable uniform estimate, since it allows us to choose the adjustable parameter $t$ flexibly to minimize the right-hand side of (\ref{ineq.modified-boundary-estimate}), according to certain norms of $F$ and $h$. The special function $\mathcal{M}_{D_{5r},t}(F\delta)$ is introduced here due to a technical reason that may be seen in the proof Theorem \ref{thm.main} (ii). Observe that (\ref{ineq.modified-boundary-estimate}) recovers (\ref{est.unif}) if we simply set $t = r$; see Remark \ref{rmk.recover}. To see the usefulness of (\ref{ineq.modified-boundary-estimate}), we assume $F = 0$ and $h$ is highly oscillatory in a form of $h(x) = \tilde{h}(x/\e)$. Then the usual uniform estimate (\ref{est.unif}) gives $\norm{\nabla u_\e}_{L^\infty(D_r)} + \text{osc}_{D_r}[p_\varepsilon] \lesssim O(\e^{-\eta})$, while (\ref{ineq.modified-boundary-estimate}) with $t = \e$ leads to an obviously improved bound $O(|\ln \e|)$. The loss of $O(|\ln \e|)$, leading to the extra logarithm in (\ref{est.mainDG}), is expected and essentially cannot be avoided, since by the singular integral theory, the mapping $h \mapsto \nabla u_\e$ is not bounded in $L^\infty$, even for the equation with Laplace operator.

Notice that in Theorem \ref{thm.modified-boundary-estimate}, we have assumed $g = 0$ and $f = 0$. This is due to the lack of the integral representation for the pressure $p_\e$ in the most general setting. Actually, similar adjustable estimate is still valid for non-trivial data $g$ and $f$, if only $\norm{\nabla u_\e}_{L^\infty(D_r)}$ is concerned. The strategy to deal with the divergence data $g$ is contained in Theorem \ref{thm.divergence.modifiedC1} (the theorem itself is critical for proving the Theorem \ref{thm.main}). The strategy to handle the boundary data $f$ is contained in Theorem \ref{thm.elliptic.adj}, where, of independent interest, we obtain an analog for elliptic system. Precisely, we prove that if $u_\e$ is the solution of $\cL_\e(u_\e) = \text{div}(h)$ in $\Omega$ and $u_\e = f$ on $\partial\Omega$, then there exists a constant $C$ such that for any $0<t\le R_0 = \text{diam}(\Omega)$,
\begin{equation*}
\norm{\nabla u_\e}_{L^\infty(\Omega)} \le C\ln[t^{-1}R_0+2] \big(\norm{h}_{L^\infty(\Omega)} + \norm{\nabla f}_{L^\infty(\Omega)} \big) + Ct^\eta \big([h]_{C^\eta(\Omega)} + [\nabla f]_{C^\eta(\Omega)} \big).
\end{equation*}
The last estimate is new and can be viewed as a refined version of the uniform Lipschitz estimate proved in \cite{AL8701}.

Finally, we point out that the proof of the adjustable estimate in Theorem \ref{thm.modified-boundary-estimate} (also in Theorem \ref{thm.divergence.modifiedC1} and \ref{thm.elliptic.adj}) essentially relies on the uniform pointwise estimates of the Green's functions established in Theorem \ref{thm.green-function.pointwise} and Lemma \ref{lem.Pi-difference}.

The organization of the paper is the following. In section 2, we provide some preliminaries in periodic homogenization of Stokes systems. In section 3, we establish the uniform estimates, i.e., Lipschitz estimate for the velocity and oscillation for the pressure. In section 4, we obtain the uniform estimates for the Green's functions. Estimate (\ref{est.mainG}) is proved in Section 5, (\ref{est.mainDG}) and (\ref{est.mainPi}) are proved in Section 6, and the expansions for $\nabla_y \nabla_x G_\e(x,y), \nabla_y \Pi_\e(x,y)$ and the fundamental solutions are given in Section 7.

\section{Preliminaries}

In this section, we give a review of homogenization theory of the Stokes systems with periodic coefficients. We refer the reader to \cite[pp.76-81]{Lions} and \cite{GuShen15,Gu16} for more details. We begin with the solvability of Stokes system (\ref{def.Dirichlet.Stokes}) and the energy estimate.
\begin{theorem}[\cite{GuShen15}, Theorem 2.1]\label{thm.energy-estimate}
Let $\Omega$ be a bounded Lipschitz domain and $A$ satisfy the ellipticity condition (\ref{cond.ellipticity}). Assume $F\in H^{-1}(\Omega;\mathbb{R}^d), h\in L^2(\Omega;\R^{d\times d})$, $g\in L^2(\Omega)$ and $f\in H^{1/2}(\partial\Omega;\mathbb{R}^d)$ satisfy the compatibility condition (\ref{cond.compatibility}). Then there exists a unique (up to a constant) weak solution $(u_\varepsilon,p_\varepsilon) \in H^1(\Omega;\mathbb{R}^d)\times L^2(\Omega)$ of Stokes system (\ref{def.Dirichlet.Stokes}) such that
\begin{equation}\label{ineq.energy}
\|u_\varepsilon\|_{H^1(\Omega)}+\|p_\varepsilon-\average_\Omega p_\varepsilon\|_{L^2(\Omega)} 
\le C\Big\{\|F\|_{H^{-1}(\Omega)}+\norm{h}_{L^2(\Omega)}+\|g\|_{L^2(\Omega)}+\|f\|_{H^{1/2}(\partial\Omega)}\Big\},
\end{equation}
where $C$ depends only on $d$, $\mu$ and $\Omega$.
\end{theorem}
In view of (\ref{ineq.energy}), for a bounded measurable set $E$, we define the space $L_0^p(E)$ as
\begin{equation}\label{def.L20space}
L_0^p(E)=\bigg\{\varphi\in L^p(E) : \int_{E} \varphi=0\bigg\}.
\end{equation}
For any $\varphi\in L^p(E)$, the $L_0^p$-norm is defined by
$$
\|\varphi\|_{L^p_0(E)}=\|\varphi-\average_E \varphi\|_{L^p(E)}.
$$

Let $Y=(0,1]^d$. Suppose $A$ satisfies (\ref{cond.ellipticity}) and (\ref{cond.periodicity}). For each $1\le j,\beta \le d$, there exist 1-periodic functions $(\chi_j^\beta,\pi_j^\beta)\in H^1_{\text{loc}}(\mathbb{R}^d;\mathbb{R}^d)\times L^2_{\text{loc}}(\mathbb{R}^d)$, which are called the correctors of the Stokes system (\ref{def.Dirichlet.Stokes}),
such that the following cell problem holds
\begin{equation}\label{def.corrector}
\left\{
\begin{aligned}
\mathcal{L}_1(\chi_j^\beta+P_j^\beta)+\nabla \pi_j^\beta &= 0 \quad \text{in }\mathbb{R}^d, \\
\text{ div}(\chi_j^\beta) &=0 \quad \text{in }\mathbb{R}^d,\\
\int_Y \pi_j^\beta=0,\text{ and} \int_Y \chi_j^\beta &=0.\\
\end{aligned}
\right.
\end{equation}
Recall that the homogenized operator is given by $\mathcal{L}_0=-\text{div}(\widehat{A}\nabla)$, where $\widehat{A}=(\widehat{a}_{ij}^{\alpha\beta})$ is the homogenized (effective) matrix defined by
\begin{equation}\label{def.effective-matrix}
\widehat{a}_{ij}^{\alpha\beta}=\average_Y \left[a_{ij}^{\alpha\beta}(y)+a_{ik}^{\alpha\gamma}(y)\frac{\partial}{\partial y_k}(\chi_j^{\gamma\beta})(y)\right]dy.
\end{equation}
By the homogenization theory of Stokes systems (see \cite{Lions,GuShen15}), we know that, as $\varepsilon\rightarrow 0$, 
$$
u_\varepsilon\rightharpoonup u_0 \text{ weakly in }H^1(\Omega;\mathbb{R}^d)\quad\text{ and }\quad p_\varepsilon-\average_{\Omega} p_\varepsilon\rightharpoonup p_0-\average_{\Omega} p_0 \text{ weakly in }L^2(\Omega),
$$
where $(u_0,p_0)$ is a solution of the following homogenized (effective) Stokes system
\begin{equation}\label{def.Dirichlet.Stokes0}
\left\{
\begin{aligned}
\mathcal{L}_0(u_0)+\nabla p_0 &= F + \text{div}(h) &\qquad \text{ in }\Omega,\\
\text{div}(u_0) &=g &\qquad\text{ in }\Omega,\\
u_0& =f &\qquad\text{ on }\partial\Omega.
\end{aligned}
\right.	
\end{equation}

Then, we introduce the dual correctors $(\phi_{kij}^{\alpha\beta},q_{ij}^\beta)$ of Stokes system (\ref{def.Dirichlet.Stokes}); details may be found in \cite{Gu16}.
For $1\le i,j,\alpha,\beta \le d$, let
\begin{equation}\label{def.b}
b_{ij}^{\alpha\beta}(y)=a_{ij}^{\alpha\beta}(y)+a_{ik}^{\alpha\gamma}(y)\frac{\partial}{\partial y_k}(\chi_{j}^{\gamma\beta}(y))-\widehat{a}_{ij}^{\alpha\beta}.
\end{equation}
It is worth noting that $b_{ij}^{\alpha\beta}\in L^\infty(Y)$ is $1$-periodic with 
$\int_Y b_{ij}^{\alpha\beta}(y)dy=0$.
\begin{lemma}[\cite{Gu16}, Lemma 3.1 and Remark 3.2]\label{lem.dual-correctors}
For fixed $1\le i,j,\beta\le d$, there exists $(\phi_{kij}^{\alpha\beta},q_{ij}^\beta) \in H_{\text{per}}^1(Y)\times H_{\text{per}}^1(Y)$ such that
\begin{equation}\label{eq.dual-correctors}
b_{ij}^{\alpha\beta}=\frac{\partial}{\partial y_k}(\phi_{kij}^{\alpha\beta})+\frac{\partial}{\partial y_\alpha}(q_{ij}^\beta) \quad \text{\rm and }\quad \phi_{kij}^{\alpha\beta}=-\phi_{ikj}^{\alpha\beta}.
\end{equation}
Moreover,
\begin{equation}\label{ineq.dual-corrector-energy}
\|\phi_{kij}^{\alpha\beta}\|_{L^\infty(Y)}+\|q_{ij}^{\beta}\|_{L^\infty(Y)} \le C,
\end{equation}
and
\begin{equation}\label{eq.q-pi}
\pi_j^\beta=\frac{\partial q_{ij}^\beta}{\partial y_i},
\end{equation}
where $C$ depends only on $d$ and $\mu$.
\end{lemma}
For fixed $1\le j,\beta\le d$, let $(V_{\varepsilon,j}^\beta(x),T_{\varepsilon,j}^\beta(x))\in H^1(\Omega;\mathbb{R}^d)\times L^2(\Omega)$ satisfy
\begin{equation}\label{def.V}
\left\{
\begin{aligned}
\mathcal{L}_\varepsilon(V_{\varepsilon,j}^\beta)+\nabla T_{\varepsilon,j}^\beta&=0 &\quad\text{ in }\Omega,\\
\text{div}(V_{\varepsilon,j}^\beta)&=\text{ div}(P_j^\beta) &\quad\text{ in }\Omega.
\end{aligned}
\right.
\end{equation}
The following lemma plays a vital part in deriving asymptotic expansions of Green's functions and their derivatives. For future applications, we choose either $(\Phi^\beta_{\varepsilon,j}(x),\Lambda_j^\beta(x))$ or $(\varepsilon\chi_j^\beta(x/\varepsilon)+P_j^\beta(x), \pi_j^\beta(x/\varepsilon))$ in the place of $(V_{\varepsilon,j}^\beta(x),T_{\varepsilon,j}^\beta(x))$, whichever is convenient. 
\begin{lemma}\label{lem.w-system}
Suppose that $(u_\varepsilon,p_\varepsilon)\in H^1(\Omega;\mathbb{R}^d)\times L^2(\Omega)$ and $(u_0,p_0)\in H^2(\Omega;\mathbb{R}^d)\times H^1(\Omega)$ satisfy
$$
\left\{
\begin{aligned}
\mathcal{L}_\varepsilon(u_\varepsilon)+\nabla p_\varepsilon&= \mathcal{L}_0(u_0)+\nabla p_0 &\quad\text{ in }\Omega,\\
\text{\rm div}(u_\varepsilon)&=\text{\rm div}(u_0)  &\quad\text{ in }\Omega.
\end{aligned}
\right.
$$
Let
\begin{equation}\label{def.w}
w_\varepsilon(x)=u_\varepsilon(x)-u_0(x)-\big\{V_{\varepsilon,j}^\beta(x)-P_j^\beta(x)\big\}\cdot\frac{\partial u_0^\beta}{\partial x_j},
\end{equation}
where $(V_{\varepsilon,j}^\beta,T_{\varepsilon,j}^\beta)$ is defined in (\ref{def.V}) for each $1\le j,\beta \le d$. Then
\begin{equation}\label{eq.w-system}
\begin{aligned}
&(\mathcal{L}_\varepsilon(w_\varepsilon))^\alpha+\frac{\partial}{\partial x_\alpha}\bigg\{p_\varepsilon-p_0-T_{\varepsilon,j}^\beta\frac{\partial u_0^\beta}{\partial x_j}-\varepsilon q_{ij}^\beta(x/\varepsilon)\frac{\partial^2 u_0^\beta}{\partial x_i \partial x_j}\bigg\}\\
&=-\frac{\partial}{\partial x_i}\bigg\{ \Big[\varepsilon \phi_{kij}^{\alpha\beta}(x/\varepsilon)-a_{ik}^{\alpha\gamma}(x/\varepsilon)\big(V_{\varepsilon,j}^{\gamma\beta}-P_j^{\gamma\beta}\big)\Big]\frac{\partial^2 u_0^\beta}{\partial x_k\partial x_j}\bigg\}-\varepsilon\frac{\partial}{\partial x_i}\bigg\{ q_{ij}^\beta(x/\varepsilon)\frac{\partial^2 u_0^\beta}{\partial x_\alpha\partial x_j}\bigg\}\\
&\quad+a_{ik}^{\alpha\gamma}(x/\varepsilon)\frac{\partial}{\partial x_k}\Big[V_{\varepsilon,j}^{\gamma\beta}-P_j^{\gamma\beta}-\varepsilon\chi_j^{\gamma\beta}(x/\varepsilon)\Big]\frac{\partial^2 u_0^\beta}{\partial x_i \partial x_j}+\Big[\pi_j^\beta(x/\varepsilon)-T_{\varepsilon,j}^\beta\Big]\frac{\partial^2 u_0^\beta}{\partial x_\alpha\partial x_j}.
\end{aligned}
\end{equation}
\end{lemma}
\begin{proof}
Since $\mathcal{L}_\varepsilon(u_\varepsilon)+\nabla p_\varepsilon=\mathcal{L}_0(u_0)+\nabla p_0$, a direct calculation shows that
\begin{equation*}
\begin{aligned}
(\mathcal{L}_\varepsilon(w_\varepsilon))^\alpha &= \bigg\{[\mathcal{L}_0-\mathcal{L}_\varepsilon](u_0)\bigg\}^\alpha-\frac{\partial}{\partial x_\alpha}(p_\varepsilon-p_0)-\bigg\{\mathcal{L}_\varepsilon\Big(\big[V_{\varepsilon,j}^\beta-P_j^\beta\big]\frac{\partial u_0^\beta}{\partial x_j}\Big)\bigg\}^\alpha\\
&=-\frac{\partial}{\partial x_i}\bigg\{\Big[\widehat{a}_{ij}^{\alpha\beta}-a_{ij}^{\alpha\beta}(x/\varepsilon)\Big]\frac{\partial u_0^\beta}{\partial x_j}\bigg\}-\frac{\partial}{\partial x_\alpha}(p_\varepsilon-p_0)\\
&\quad+\frac{\partial}{\partial x_i}\bigg\{a_{ik}^{\alpha\gamma}(x/\varepsilon)\frac{\partial}{\partial x_k}\Big[V_{\varepsilon,j}^{\gamma\beta}-P_j^{\gamma\beta}\Big]\bigg\}\frac{\partial u_0^\beta}{\partial x_j}\\
&\quad+a_{ik}^{\alpha\gamma}(x/\varepsilon)\frac{\partial}{\partial x_k}\Big[V_{\varepsilon,j}^{\gamma\beta}-P_j^{\gamma\beta}\Big]\frac{\partial^2 u_0^\beta}{\partial x_i \partial x_j}\\
&\quad +\frac{\partial}{\partial x_i}\bigg\{a_{ik}^{\alpha\gamma}(x/\varepsilon)\Big[V_{\varepsilon,j}^{\gamma\beta}-P_j^{\gamma\beta}\Big]\frac{\partial^2 u_0^\beta}{\partial x_k\partial x_j}\bigg\}.
\end{aligned}
\end{equation*}
By definitions (\ref{def.V}) and (\ref{def.corrector}), we obtain that
\begin{equation*}
\mathcal{L}_\varepsilon(V_{\varepsilon,j}^\beta-P_j^\beta)=-\mathcal{L}_\varepsilon(P_j^\beta)-\nabla T_{\varepsilon,j}^\beta=\mathcal{L}_\varepsilon(\varepsilon\chi_j^\beta(x/\varepsilon))+\nabla (\pi_j^\beta(x/\varepsilon)-T_{\varepsilon,j}^\beta).
\end{equation*}
Hence,
\begin{equation}\label{eq.w-system-in-b}
\begin{aligned}
(\mathcal{L}_\varepsilon(w_\varepsilon))^\alpha & =\frac{\partial}{\partial x_i}\bigg\{b_{ij}^{\alpha\beta}(x/\varepsilon)\frac{\partial u_0^\beta}{\partial x_j}\bigg\}-\frac{\partial}{\partial x_\alpha}\bigg\{p_\varepsilon-p_0+\Big[\pi_j^\beta(x/\varepsilon)-T_{\varepsilon,j}^\beta\Big]\frac{\partial u_0^\beta}{\partial x_j}\bigg\}\\
&+a_{ik}^{\alpha\gamma}(x/\varepsilon)\frac{\partial}{\partial x_k}\Big[V_{\varepsilon,j}^{\gamma\beta}-P_j^{\gamma\beta}-\chi_j^{\gamma\beta}(x/\varepsilon)\Big]\frac{\partial^2 u_0^\beta}{\partial x_i \partial x_j}+\Big[\pi_j^\beta(x/\varepsilon)-T_{\varepsilon,j}^\beta\Big]\frac{\partial^2 u_0^\beta}{\partial x_\alpha\partial x_j}\\
& +\frac{\partial}{\partial x_i}\bigg\{a_{ik}^{\alpha\gamma}(x/\varepsilon)\Big[V_{\varepsilon,j}^{\gamma\beta}-P_j^{\gamma\beta}\Big]\frac{\partial^2 u_0^\beta}{\partial x_k\partial x_j}\bigg\},
\end{aligned}
\end{equation}
where $b_{ij}^{\alpha\beta}(y)$ is defined by (\ref{def.b}). By using Lemma \ref{lem.dual-correctors}, the first term on the right-hand side of (\ref{eq.w-system-in-b}) may be rewritten as
\begin{equation*}
\frac{\partial}{\partial x_i}\bigg\{\frac{\partial}{\partial x_k}\Big[\varepsilon \phi_{kij}^{\alpha\beta}(x/\varepsilon)\Big]\frac{\partial u_0^\beta}{\partial x_j}\bigg\}+\frac{\partial}{\partial x_i}\bigg\{\frac{\partial}{\partial x_\alpha}\Big[\varepsilon q_{ij}^{\beta}(x/\varepsilon)\Big]\frac{\partial u_0^\beta}{\partial x_j}\bigg\}:=I_1+I_2.
\end{equation*}
Thanks to the anti-symmetry condition in (\ref{eq.dual-correctors}), we have
\begin{equation*}
I_1=-\varepsilon\frac{\partial}{\partial x_i}\bigg\{ \phi_{kij}^{\alpha\beta}(x/\varepsilon)\frac{\partial^2 u_0^\beta}{\partial x_k\partial x_j}\bigg\}, 
\end{equation*}
and
\begin{equation*}
I_2=\frac{\partial}{\partial x_\alpha}\bigg\{\frac{\partial}{\partial x_i}\bigg[\varepsilon q_{ij}^\beta(x/\varepsilon)\frac{\partial u_0^\beta}{\partial x_j}\bigg]\bigg\}-\frac{\partial}{\partial x_i}\bigg\{\varepsilon q_{ij}^\beta(x/\varepsilon)\frac{\partial^2 u_0^\beta}{\partial x_\alpha\partial x_j}\bigg\}.
\end{equation*}
Combining these with (\ref{eq.q-pi}), we obtain (\ref{eq.w-system}) as desired. 
\end{proof}

Finally, we provide the Cacciopoli's inequality of Stokes systems. For any $x\in \overline{\Omega}$, we define $D_r = D_r(x) := \Omega \cap B_r(x)$, where $B_r(x)$ is the ball centered at $x$ with radius $r$. We also denote by $\Delta_r = \Delta_r(x): = \partial\Omega \cap B_r(x)$ the surface ball on $\partial\Omega$. These notations will be used throughout.

\begin{theorem}[\cite{GuShen15}, Theorem 6.2]\label{thm.Cacciopoli}
Suppose that $A$ satisfies the ellipticity condition (\ref{cond.ellipticity}). Let $(u_\varepsilon,p_\varepsilon) \in H^1(D_r;\mathbb{R}^d)\times L^2(D_r)$ be a weak solution of
$$
\left\{
\aligned
\mathcal{L}_\varepsilon (u_\varepsilon)+\nabla p_\varepsilon&=F +\text{\rm div}(h) &\ \text{ in } D_r,\\
\text{\rm div} (u_\varepsilon)&= g &\ \text{ in } D_r,\\
u_\varepsilon &=f &\ \text{ on } \Delta_r.
\endaligned
\right.
$$
Then
\begin{equation}\label{ineq.Cacciopoli}
\int_{D_{r/2}} |\nabla u_\varepsilon|^2  \le C\left\{\frac{1}{r^2} \int_{D_r} |u_\e|^2 +r^2 \int_{D_r} |F|^2+\int_{D_r} |h|^2 +\int_{D_r} |g|^2  +\|f\|^2_{H^{1/2}(\Delta_r)} \right\},
\end{equation}
where $C$ depends only on $d$, $\mu$ and $\Omega$.
\end{theorem}

\section{Uniform Regularity}
The goal of this section is to establish the uniform estimate for the solution $(u_\e,p_\e)$ of the general Stokes system (\ref{def.Dirichlet.Stokes}) with data in certain spaces. The following theorem is our main result of this section, which provides the local (both boundary and interior) Lipschtiz estimate for the velocity $u_\varepsilon$ and oscillation estimate for the pressure $p_\varepsilon$. To state the theorem, let $x_0\in \overline{\Omega}$ be fixed and set $D_r = D_r(x_0)$ and $\Delta_r=\Delta_r(x_0)$ for $0<r<\text{diam}(\Omega)$.

\begin{theorem}\label{thm.boundary-estimate}
	Suppose that $\Omega$ is a bounded $C^{1,\eta}$ domain and $A$ satisfies (\ref{cond.ellipticity}), (\ref{cond.periodicity}) and (\ref{cond.holder}). Assume $F \in L^p(\Omega;\R^d)$ for some $p>d$, $h\in C^{0,\eta}(\Omega;\R^{d\times d})$, $g\in C^{0,\eta}(\Omega)$ and $f\in C^{1,\eta}(\partial\Omega;\R^d)$ for some $\eta\in (0,1)$, satisfying compatibility condition (\ref{cond.compatibility}). Let $(u_\e,p_\e)$ be the weak solution of system (\ref{def.Dirichlet.Stokes}). Then for any $0<R<\rtxt{diam}(\Omega)$,
	\begin{align*}
	&\norm{\nabla u_\e}_{L^\infty(D_R)} + \osc{D_R}{p_\e}\\
	&\qquad \le C \bigg\{ \frac{1}{R} \bigg( \average_{D_{4R} } |u_\e|^2 \bigg)^{1/2} + R\bigg( \average_{D_{4R} } |F|^p \bigg)^{1/p}+  R^\eta [h]_{C^{0,\eta}(D_{4R})}+ \norm{g}_{L^\infty(D_{4R})}\\ 
	&\qquad \qquad   +  R^\eta [g]_{C^{0,\eta}(D_{4R})}+ \frac{1}{R}\norm{f}_{L^\infty(\Delta_{4R})} + \norm{\nabla_{\rtxt{tan}} f}_{L^\infty(\Delta_{4R})} + R^\eta [\nabla_{\rtxt{tan}} f]_{C^{0,\eta}(\Delta_{4R})} \bigg\},
	\end{align*}
	where $C$ depends only on $d,\, p, \, \eta, A$ and $\Omega$.
\end{theorem}

To prove the above theorem, we follow a line of \cite{Shen17} with some useful techniques from \cite{GuXu17}, which relies essentially on a (positive power) rate of convergence for system (\ref{def.Dirichlet.Stokes}). To focus on our main result, we will just state the necessary theorem below and postpone its technical proof to Appendix.

\begin{theorem}\label{thm.convergence-rates}
	Suppose $\Omega$ is a bounded $C^1$ domain and $A$ satisfies (\ref{cond.ellipticity}), (\ref{cond.periodicity}) and (\ref{cond.holder}). Assume $F \in L^p(\Omega;\R^d)$ for some $p>d$, $h\in C^{0,\eta}(\Omega;\R^{d\times d})$, $g\in C^{0,\eta}(\Omega)$ for some $\eta \in (0,1/2]$ and $f\in H^1(\partial\Omega;\R^d)$, satisfies the compatibility condition (\ref{cond.compatibility}). Let $\varepsilon\rightarrow 0$, the system (\ref{def.Dirichlet.Stokes}) homogenizes to the following effective system
	\begin{equation}\label{eq.generalL0}
	\left\{
	\begin{aligned}
	\cL_0(u_0)+\nabla p_0 &= F + \rdiv(h) &\qquad \text{ in }\Omega,\\
	\rdiv(u_0) &=g &\qquad\text{ in }\Omega,\\
	u_0& =f &\qquad\text{ on }\partial\Omega.
	\end{aligned}
	\right.	
	\end{equation}
	Moreover, there exists some $\gamma = \gamma(d,\eta)>0$ such that,
	\begin{equation}\label{ineq.convergence-rates}
	\begin{aligned}
	&\norm{u_\e - u_0}_{L^2(\Omega)} + \norm{p_\e - p_0 - \pi^\e \nabla u_0}_{L^2_0(\Omega)} \\
	&\qquad \le C\e^\gamma \Big(\norm{F}_{L^p(\Omega)} + \norm{h}_{C^{0,\eta}(\Omega)} + \norm{g}_{C^{0,\eta}(\Omega)} + \norm{f}_{H^1(\partial\Omega)} \Big),
	\end{aligned}
	\end{equation}
	where $C$ depends only on $d, p, \eta, \gamma, A$ and $\Omega$.
\end{theorem}
\begin{proof}
	The theorem is a corollary of \cite[Theorem 3.1]{GuXu17} and a general interpolation argument. See Appendix for details.
\end{proof}

Now we are ready to prove Theorem \ref{thm.boundary-estimate}.

\begin{proof}[Proof of Theorem \ref{thm.boundary-estimate}]

Part (i): Lipschitz estimate for the velocity $v_\varepsilon$.

Step 1: Set-up. Without loss of generality, we may assume $h(x_0)=0$ and $R> 10\e$. We define the following auxiliary quantities:
\begin{equation}\label{def.H}
\begin{aligned}
H(t;v)&=\frac{1}{t}\inf\limits_{\substack{M\in \mathbb{R}^{d\times d}\\ q\in\mathbb{R}^d}}\bigg\{\bigg( \average_{D_{t} } |v-Mx-q|^2 \bigg)^{1/2} + t^2\bigg( \average_{D_{t} } |F|^p \bigg)^{1/p}+  t^{1+\eta} [h]_{C^{0,\eta}(D_{t})}\\
&\quad   + t\norm{g-\text{div}(Mx)}_{L^\infty(D_{t})} +  t^{1+\eta} [g-\text{div}(Mx)]_{C^{0,\eta}(D_{t})} +\|f-Mx-q\|_{L^\infty(\Delta_t)}\\
&\quad+ t\norm{\nabla_{\rtxt{tan}} (f-Mx-q)}_{L^\infty(\Delta_{t})} + t^{1+\eta} [\nabla_{\rtxt{tan}} (f-Mx-q)]_{C^{0,\eta}(\Delta_{t})} \bigg\},
\end{aligned}
\end{equation}
and
\begin{equation}\label{def.Theta}
\begin{aligned}
\Theta(t)&=\frac{1}{t}\inf\limits_{q\in\mathbb{R}^d}\bigg\{\bigg( \average_{D_{t} } |u_\varepsilon-q|^2 \bigg)^{1/2} + t^2\bigg( \average_{D_{t} } |F|^p \bigg)^{1/p}+  t^{1+\eta} [h]_{C^{0,\eta}(D_{t})}+ t\norm{g}_{L^\infty(D_{t})}\\
&\quad\quad    +  t^{1+\eta} [g]_{C^{0,\eta}(D_{t})}+\|f-q\|_{L^\infty(\Delta_t)} + t\norm{\nabla_{\rtxt{tan}} (f)}_{L^\infty(\Delta_{t})} + t^{1+\eta} [\nabla_{\rtxt{tan}} (f)]_{C^{0,\eta}(\Delta_{t})} \bigg\}.
\end{aligned}
\end{equation}

Step 2: For any fixed $r\in [\e,R]$, we show that there exists a weak solution $(v,\tau)$ of 
\begin{equation}\label{def.v-system-Dr}
\left\{
\begin{aligned}
\mathcal{L}_0(v)+\nabla \tau &= F +\text{div}(h)&\qquad \text{ in }D_r,\\
\text{div}(v) &=g &\qquad\text{ in }D_r,\\
v& =f &\qquad\text{ on }\Delta_r,
\end{aligned}
\right.	
\end{equation}
such that 
\begin{equation}\label{ineq.step2}
\bigg(\average_{D_r} |u_\varepsilon-v|^2\bigg)^{1/2}\le C\varepsilon^\gamma r^{1-\gamma}\Theta(2r).
\end{equation}
Actually, by an argument of rescaling, it suffices to show (\ref{ineq.step2}) with $r = 1$. First, since $(u_\varepsilon,p_\varepsilon)$ satisfies (\ref{def.Dirichlet.Stokes}) in $D_2$, it follows from the Cacciopoli's inequality (\ref{ineq.Cacciopoli}) that
\begin{equation*}
\int_{D_{3/2}} |\nabla u_\varepsilon|^2 \le C\left\{\int_{D_2} |u_\varepsilon|^2+\int_{D_2} |F|^2+\int_{D_2} |h|^2+\int_{D_2} |g|^2+\|f\|^2_{H^{1/2}(\Delta_2)}\right\}.
\end{equation*}
By the co-area formula, there exists\footnote{The construction of $\widetilde{D}$ should preserve the $C^1$ character of $\Delta_2$, as do all the later constructions of the intermediate domains.} some $C^1$ domain $\widetilde{D}$ such that $D_{3/2}\subset \widetilde{D}\subset D_2$ and
\begin{equation*}
\int_{\partial{\widetilde{D}}\backslash \Delta_2} (|\nabla u_\varepsilon|^2+|u_\varepsilon|^2) \le C\left\{\int_{D_2} |u_\varepsilon|^2+\int_{D_2} |F|^2+\int_{D_2} |h|^2+\int_{D_2} |g|^2+\|f\|^2_{H^{1/2}(\Delta_2)}\right\}.
\end{equation*}
Let $(v,\tau)$ be a weak solution of 
\begin{equation*}
\left\{
\begin{aligned}
\mathcal{L}_0(v)+\nabla \tau &= F +\text{div}(h)&\qquad \text{ in }\widetilde{D},\\
\text{div}(v) &=g &\qquad\text{ in }\widetilde{D},\\
v& =u_\varepsilon &\qquad\text{ on }\partial \widetilde{D}.
\end{aligned}
\right.	
\end{equation*}
Now by applying Theorem \ref{thm.convergence-rates}, we obtain that
\begin{equation}\label{ineq.proof.Step2}
\begin{aligned}
\|u_\varepsilon-v\|_{L^2(D_1)}&\le \|u_\varepsilon-v\|_{L^2(\widetilde{D})} \\
&\le C\varepsilon^\gamma \big(\|u_\varepsilon\|_{H^1(\partial \widetilde{D})}+\|F\|_{L^p(\widetilde{D})}+[h]_{C^{0,\eta}(\widetilde{D}_t)}+\|g\|_{C^{0,\eta}(\widetilde{D})}\big)\\
&\le C\varepsilon^\gamma \big(\|u_\varepsilon\|_{L^2(D_2)}+\|F\|_{L^p(D_2)}+[h]_{C^{0,\eta}(D_2)}+\|g\|_{C^{0,\eta}(D_2)}+\|f\|_{C^{1,\eta}(\Delta_2)}\big).
\end{aligned}
\end{equation}
Then (\ref{ineq.step2}) holds true by applying (\ref{ineq.proof.Step2}) to $u_\varepsilon-q$ for $q\in \mathbb{R}^d$ (because $u_\e - q$ is also a solution) and taking minimum over all $q\in \R^d$.\\

Step 3: We prove that there exists $\theta\in (0,1/4)$, depending only on $\mu$, $d$, $\eta$ and $\Omega$, such that
\begin{equation}\label{ineq.step3}
H(\theta r;v)\le (1/2)H(r;v), \quad\text{ for any } r\in(0,R).
\end{equation}
Again, by rescaling, we may assume $r=1$. For any $\theta\in (0,1/4)$, by choosing $q=v(x_0)$ and $M=\nabla v(x_0)$, it is easy to see that 
\begin{equation*}
H(\theta;v)\le C\theta^{\eta_1}\bigg\{ \|v\|_{C^{1,\eta}(D_\theta)} +\bigg( \average_{D_{1} } |F|^p \bigg)^{1/p}+[h]_{C^{0,\eta}(D_1)}\bigg\},
\end{equation*}
where $\eta_1=\min\{\eta,1-d/p\}$. By the boundary $C^{1,\eta}$ estimate for Stokes system with constant coefficients \cite{GiaquintaModica82}, we have
\begin{equation*}
\begin{aligned}
\|v\|_{C^{1,\eta}(D_\theta)}&\le C\bigg\{\bigg( \average_{D_{1} } |v|^2 \bigg)^{1/2} + \bigg( \average_{D_{1} } |F|^p \bigg)^{1/p}+  [h]_{C^{0,\eta}(D_{1})}+ \norm{g}_{L^\infty(D_{1})} +  [g]_{C^{0,\eta}(D_{1})}\\
&\quad\qquad\qquad\qquad   +\|f\|_{L^\infty(\Delta_1)} + \norm{\nabla_{\rtxt{tan}} (f)}_{L^\infty(\Delta_{1})} + [\nabla_{\rtxt{tan}} (f)]_{C^{0,\eta}(\Delta_{1})} \bigg\}.
\end{aligned}
\end{equation*}
It follows that
\begin{equation}\label{ineq.H-growth-intermediate}
\begin{aligned}
H(\theta;v)&\le  C\theta^{\eta_1}\bigg\{\bigg( \average_{D_{1} } |v|^2 \bigg)^{1/2} + \bigg( \average_{D_{1} } |F|^p \bigg)^{1/p}+  [h]_{C^{0,\eta}(D_{1})}+ \norm{g}_{L^\infty(D_{1})} +  [g]_{C^{0,\eta}(D_{1})}\\
&\qquad\qquad\qquad\qquad  +\|f\|_{L^\infty(\Delta_1)} + \norm{\nabla_{\rtxt{tan}} (f)}_{L^\infty(\Delta_{1})} + [\nabla_{\rtxt{tan}} (f)]_{C^{0,\eta}(\Delta_{1})} \bigg\}.
\end{aligned}
\end{equation}
Now for any $q\in\mathbb{R}^d$ and $M\in \mathbb{R}^{d\times d}$, $(v-Mx-q,\tau)$ also satisfies the main system in (\ref{def.v-system-Dr}). Thus, applying (\ref{ineq.H-growth-intermediate}) to $v-Mx-q$, we obtain
\begin{equation*}
H(\theta;v)\le C\theta^{\eta_1}H(1;v),
\end{equation*}
which leads to our desired result (\ref{ineq.step3}) by fixing $\theta\in (0,1/4)$ so small that $C\theta^{\eta_1}\le 1/2$.\\

Step 4: Next, we show that for any $r \in [\e,R/2]$,
\begin{equation}\label{ineq.step4}
H(\theta r;u_\varepsilon) \le \frac{1}{2}H(r;u_\varepsilon)+C\Big(\frac{\varepsilon}{r}\Big)^{\gamma}\Theta(2r).
\end{equation}
To see this, we fix $r\in [\varepsilon, R/2]$, and let $(v,\tau)$ be the solution given in Step 1.
Then,
\begin{equation*}
\begin{aligned}
H(\theta r; u_\varepsilon) &\le \frac{1}{\theta r}\bigg(\average_{D_{\theta r}}|u_\varepsilon-v|^2\bigg)^{1/2}+H(\theta r;v)\\
&\le \frac{C}{r}\bigg(\average_{D_{r}}|u_\varepsilon-v|^2\bigg)^{1/2}+\frac{1}{2}H(r;v)\\
&\le \frac{1}{2}H(r;u_\varepsilon)+\frac{C}{r}\bigg(\average_{D_{r}}|u_\varepsilon-v|^2\bigg)^{1/2}\\
&\le \frac{1}{2}H(r;u_\varepsilon)+ C\Big(\frac{\varepsilon}{r}\Big)^\gamma\Theta(2r),
\end{aligned}
\end{equation*}
where we have used (\ref{ineq.step3}) and (\ref{ineq.step2}) in the second and last inequality, respectively.\\

Step 5: Now we proceed to prove the local Lipschitz estimate for the velocity $u_\e$. We claim that it suffices to show that for any $0<r \le R/2$,
\begin{equation}\label{ineq.step5}
	\Theta(r) \le C \Theta(R).
\end{equation}
Indeed, if (\ref{ineq.step5}) holds true, by the Cacciopoli's inequality (\ref{ineq.Cacciopoli}), we have
\begin{equation*}
\bigg(\average_{D_r}|\nabla u_\varepsilon|^2\bigg)^{1/2} \le C \Theta(R),
\end{equation*}
which, by letting $r\rightarrow 0$, implies $|\nabla u_\varepsilon(x_0)|\le C\Theta(R)$, for the given $x_0\in D_R$. Following by the same argument, we can further show that for any $x\in D_R$, $|\nabla u_\varepsilon(x)|\le C\Theta(4R)$. Hence, $\norm{\nabla u_\e}_{L^\infty(D_R)} \le C\Theta(4R)$ is true as desired. 

The proof of (\ref{ineq.step5}) is based on \cite[Lemma 8.5]{Shen17}, of which the conditions need to be verified. We may assume $0<\varepsilon<1/4$ and recall the definitions of $H(t;u_\e)$ and $\Theta(t)$. Let $M_{t,\varepsilon}$ be the $d\times d$ matrix such that,
\begin{equation*}
\begin{aligned}
H(t;u_\varepsilon)&=\frac{1}{t}\inf_{q\in\mathbb{R}^d}\bigg\{\bigg( \average_{D_{t} } |u_\varepsilon-M_{t,\varepsilon}x-q|^2 \bigg)^{1/2} + t^2\bigg( \average_{D_{t} } |F|^p \bigg)^{1/p}+  t^{1+\eta} [h]_{C^{0,\eta}(D_{t})}\\
&\qquad\qquad + t\norm{g-\text{div}(M_{t,\varepsilon}x)}_{L^\infty(D_{t})} +  t^{1+\eta} [g-\text{div}(M_{t,\varepsilon}x)]_{C^{0,\eta}(D_{t})} \\
&\qquad\qquad+\|f-M_{t,\varepsilon}x-q\|_{L^\infty(\Delta_t)}+ r\norm{\nabla_{\rtxt{tan}} (f-M_{t,\varepsilon}x-q)}_{L^\infty(\Delta_{t})}\\
&\qquad\qquad + t^{1+\eta} [\nabla_{\rtxt{tan}} (f-M_{t,\varepsilon}x-q)]_{C^{0,\eta}(\Delta_{t})} \bigg\}.
\end{aligned}
\end{equation*}
Observe that $\Theta(t)\le H(t;u_\varepsilon)+C|M_{t,\varepsilon}|$. Hence, (\ref{ineq.step4}) can be written as
\begin{equation}\label{ineq.step5-condition3}
H(\theta r;u_\varepsilon)\le \frac{1}{2}H(r;u_\varepsilon)+C\omega(\varepsilon/r)\left\{H(2r;u_\varepsilon)+|M_{2r,\varepsilon}|\right\},
\end{equation}
for any $r\in[\varepsilon,R/2]$, where $\omega(t)=t^\gamma$.

Now, to apply \cite[Lemma 8.5]{Shen17}, we need to verify
\begin{equation}\label{ineq.step5-condition}
\max\limits_{r\le t \le 2r} H(t;u_\varepsilon) + \max\limits_{r\le t,s \le 2r} \big||M_{t,\varepsilon}|-|M_{s,\varepsilon}|\big| \le C H(2r;u_\varepsilon).
\end{equation}
The first part of (\ref{ineq.step5-condition}) is obvious. For any $t,s\in [r,2r]$, the second part follows by
\begin{equation*}
\begin{aligned}
|M_{t,\varepsilon}-M_{s,\varepsilon}|& \le \inf_{q\in\mathbb{R}^d}\frac{C}{r}\bigg(\average_{D_r}|(M_{t,\varepsilon}-M_{s,\varepsilon})x-q|^2 \bigg)^{1/2}\\
&\le \frac{C}{t}\inf_{q\in\mathbb{R}^d}\bigg(\average_{D_t}|u_\varepsilon -M_{t,\varepsilon}x-q|^2 \bigg)^{1/2}+\frac{C}{s}\inf_{q\in\mathbb{R}^d}\bigg(\average_{D_s}|u_\varepsilon-M_{s,\varepsilon}x-q|^2 \bigg)^{1/2}\\
&\le C\{H(t;u_\varepsilon)+H(s;u_\varepsilon)\}\\
&\le C H(2r;u_\varepsilon).
\end{aligned}
\end{equation*}
Consequently, by applying \cite[Lemma 8.5]{Shen17}, we obtain that for $\varepsilon\le r\le R/2$, 
\begin{equation}\label{ineq.Theta.rR}
\begin{aligned}
\Theta(r) \le H(r;u_\varepsilon)+C|M_{r,\varepsilon}| \le C\{H(R;u_\varepsilon)+|M_{R,\varepsilon}|\}.
\end{aligned}
\end{equation}

Finally, observe that $H(R;u_\e) \le C\Theta(R)$ and
\begin{equation*}
\begin{aligned}
|M_{R,\e}| &\le \frac{C}{R} \inf_{q\in \R^d} \bigg( \average_{D_{R} } |M_{R,\e}x+q|^2 \bigg)^{1/2} \\
& \le \frac{C}{R} \bigg\{ \inf_{q\in \R^d} \bigg( \average_{D_{R} } |u_\e - M_{R,\e}x - q|^2 \bigg)^{1/2} + \inf_{q\in \R^d} \bigg( \average_{D_{R} } |u_\e-q|^2 \bigg)^{1/2} \bigg\} \\
& \le CH(R;u_\e) + C\Theta(R) \\
&\le C\Theta(R).
\end{aligned}
\end{equation*}
These, combined with (\ref{ineq.Theta.rR}), gives (\ref{ineq.step5}) for $\e \le r\le R/2$. The case $0<r<\varepsilon$ can be handled by a blow-up argument. The proof of (\ref{ineq.step5}) is complete.\\

Part (ii): Oscillation estimate for the pressure $p_\varepsilon$.

Step 1: Since $\Omega$ is a $C^1$ domain, we can construct a sequence of $C^1$ domains $\{\widetilde{D}_t\}$ such that $D_t\subset \widetilde{D}_t\subset D_{2t}$. For any $\varepsilon \le r \le R/4$, let $(v,\tau)$ be the solution of
$$
\left\{
\begin{aligned}
\mathcal{L}_0(v)+\nabla \tau &= F +\text{div}(h)&\qquad \text{ in }\widetilde{D}_s,\\
\text{div}(v) &=g &\qquad\text{ in }\widetilde{D}_s,\\
v& =u_\varepsilon &\qquad\text{ on }\partial \widetilde{D}_s,
\end{aligned}
\right.	
$$
for some $s=s(r,\varepsilon) \in (r,R/4)$ to be chosen later. Then, we show that
\begin{equation}\label{ineq.p.step1}
\|p_\varepsilon-\tau-\pi^\varepsilon\nabla v\|_{L^2_0(D_{r})}\le Cs^{d/2}\Big(\frac{\varepsilon}{s}\Big)^\gamma \Theta(R).
\end{equation}
To this end, note that by Part (i), we know $u_\e \in C^{0,1}(D_R)$ and hence $u_\varepsilon\in H^1(\partial \widetilde{D}_s)$. It follows from Theorem \ref{thm.convergence-rates} and a similar argument as in the proof of (\ref{ineq.step2}) that
\begin{equation*}
\begin{aligned}
\|p_\varepsilon-\tau-\pi^\varepsilon\nabla v\|_{L^2_0(D_{r})}&\le \|p_\varepsilon-\tau-\pi^\varepsilon\nabla v\|_{L^2_0(\widetilde{D}_s)}\\
&\le C\Big(\frac{\varepsilon}{s}\Big)^\gamma \big(s^{d/2}\Theta(2s)\big)\\
&\le Cs^{d/2}\Big(\frac{\varepsilon}{s}\Big)^\gamma \Theta(R),
\end{aligned}
\end{equation*}
where we have used (\ref{ineq.step5}) in the last inequality.\\

Step 2: We prove that, for any $\varepsilon \le r \le R/4$,
\begin{equation}\label{ineq.p.step2}
K(r;p_\varepsilon)\le 
\left\{
\begin{aligned}
&C\Big(\frac{\varepsilon}{r}\Big)^{\sigma}\Theta(R), &\text{ if } r<s_0,\\
&C\bigg\{\Big(\frac{R}{r}\Big)^{d/2}\Big(\frac{\varepsilon}{R}\Big)^{\gamma}+\Big(\frac{r}{R}\Big)^{\eta}+\Big(\frac{\varepsilon}{r}\Big)\bigg\}\Theta(R), &\text{ if } r\ge s_0,
\end{aligned}
\right.
\end{equation}
where $\sigma=2\gamma\eta/(d-2\gamma+2\eta)$, $s_0=R^{\frac{d-2\gamma+2\eta}{d+2\eta}}\varepsilon^{\frac{2\gamma}{d+2\eta}}$ and $K$ is defined by
\begin{equation}\label{def.K}
K(r;\varphi)=\sup_{r\le t\le 2r}\left|\average_{{D}_t} \varphi-\average_{{D}_{2r}} \varphi\right|.
\end{equation}
By (\ref{ineq.p.step1}) and the local $C^{0,\eta}$ estimate for the pressure of Stokes systems with constant coefficients (see \cite{GiaquintaModica82}), we have
\begin{equation}\label{ineq.triangle-K}
\begin{aligned}
K(r;p_\varepsilon) & \le K(r;p_\varepsilon-\tau-\pi^\varepsilon\nabla v)+K(r;\tau)+K(r;\pi^\varepsilon\nabla v)\\
& \le Cr^{-d/2}\|p_\varepsilon-\tau-\pi^\varepsilon\nabla v\|_{L^2_0(\widetilde{D}_{2r})} + K(r;\tau-\tau(x_0)) + K(r;\pi^\varepsilon\nabla v) \\
& \le C\Big(\frac{s}{r}\Big)^{d/2}\Big(\frac{\varepsilon}{s}\Big)^\gamma\Theta(R) + Cr^\eta[\tau]_{C^{0,\eta}({D}_{2r})} + K(r;\pi^\varepsilon\nabla v).
\end{aligned}
\end{equation}
To estimate $K(r;\pi^\varepsilon\nabla v)$, we let $ S_t=\{z\in\mathbb{Z}^d : \varepsilon\{Y+z\}\subset {D}_{t}\}$ and observe that
\begin{equation}\label{decompose.Dr}
{D}_{t}\backslash\bigcup_{z\in S_t} \varepsilon \{Y+ z\}\subset E_{t,\varepsilon}:=\big\{ x\in {D}_{t}| \text{ dist}(x,\partial {D}_{t})\le \sqrt{d}\varepsilon \big\}.
\end{equation}
Recall that $\int_Y \pi(y)\, dy=0$, then one has
\begin{equation*}
\begin{aligned}
\left|\average_{{D}_{t}}\pi^\varepsilon\nabla v\right|&\le \frac{1}{|{D}_t|}\bigg\{\sum_{z\in S_t}\int_{\varepsilon\{Y+z\}}\big|\pi^\varepsilon\big[\nabla v-\nabla v(z)\big]\big|+\int_{E_{t,\varepsilon}}|\pi^\varepsilon\nabla v|\bigg\}\\
&\le C\left\{\varepsilon^\eta[\nabla v]_{C^{0,\eta}({D}_{t})}+\frac{\varepsilon}{t}\|\nabla v\|_{L^\infty({D}_{t})}\right\}.
\end{aligned}
\end{equation*}
Therefore
\begin{equation}\label{ineq.K3}
\begin{aligned}
K(r;\pi^\varepsilon\nabla v)&\le \sup_{r\le t\le 2r} \left|\average_{{D}_{t}}\pi^\varepsilon\nabla v\right|+\left|\average_{{D}_{2r}}\pi^\varepsilon\nabla v\right|\\
&\le C\left\{\varepsilon^\eta[\nabla v]_{C^{0,\eta}({D}_{2r})}+\frac{\varepsilon}{r}\|\nabla v\|_{L^\infty({D}_{2r})}\right\}.
\end{aligned}
\end{equation}
Combining (\ref{ineq.triangle-K}) and (\ref{ineq.K3}), we obtain
\begin{equation}\label{ineq.K.rpe}
K(r;p_\varepsilon)\le C\bigg\{\Big(\frac{s}{r}\Big)^{d/2}\Big(\frac{\varepsilon}{s}\Big)^\gamma\Theta(R)+r^\eta[\tau]_{C^{0,\eta}({D}_{2r})}+\varepsilon^\eta[\nabla v]_{C^{0,\eta}({D}_{2r})}+\frac{\varepsilon}{r}\|\nabla v\|_{L^\infty({D}_{2r})}\bigg\}.
\end{equation}

Now, for any $q\in\mathbb{R}^d$, by the $C^{1,\eta}$ estimate for velocity and $C^{0,\eta}$ estimate for pressure of Stokes systems with constant coefficients, we have
\begin{equation}\label{ineq.v-full-C1eta}
\begin{aligned}
&\|\nabla v\|_{L^\infty({D}_{2r})}+s^{\eta}[\nabla v]_{C^{0,\eta}({D}_{2r})}+s^{\eta}[\tau]_{C^{0,\eta}({D}_{2r})}\\
&\le C \bigg\{ \frac{1}{s} \bigg( \average_{\widetilde{D}_{s} } |v-q|^2 \bigg)^{1/2} + s\bigg( \average_{\widetilde{D}_{s} } |F|^p \bigg)^{1/p}+  s^\eta [h]_{C^{0,\eta}(\widetilde{D}_{s})}+ \norm{g}_{L^\infty(\widetilde{D}_{s})}\\ 
&\quad\quad+  s^\eta [g]_{C^{0,\eta}(\widetilde{D}_{s})}+ \frac{1}{s}\norm{f-q}_{L^\infty(\widetilde{\Delta}_{s})} + \norm{\nabla_{\rtxt{tan}} f}_{L^\infty(\widetilde{\Delta}_{s})} + s^\eta [\nabla_{\rtxt{tan}} f]_{C^{0,\eta}(\widetilde{\Delta}_{s})} \bigg\}.
\end{aligned}
\end{equation}
By using (\ref{ineq.step2}) and (\ref{ineq.step5}) again, we see that
\begin{equation*}
\begin{aligned}
\frac{1}{s}\bigg( \average_{\widetilde{D}_{s} } |v-q|^2 \bigg)^{1/2}&\le \frac{1}{s}\bigg( \average_{\widetilde{D}_{s} } |u_\varepsilon-v|^2 \bigg)^{1/2}+\frac{1}{s}\bigg( \average_{\widetilde{D}_{s} } |u_\varepsilon-q|^2 \bigg)^{1/2}\\
&\le C\Big(\frac{\varepsilon}{s}\Big)^\gamma \Theta(R)+\frac{1}{s}\bigg( \average_{\widetilde{D}_{s} } |u_\varepsilon-q|^2 \bigg)^{1/2},
\end{aligned}
\end{equation*}
and by taking infimum over all $q\in \R^d$ for (\ref{ineq.v-full-C1eta}), we have
\begin{equation}\label{ineq.v-to-u}
\begin{aligned}
\|\nabla v\|_{L^\infty(D_{2r})}+s^{\eta}[\nabla v]_{C^{0,\eta}({D}_{2r})}+s^{\eta}[\tau]_{C^{0,\eta}({D}_{2r})}\le C\Big(\frac{\varepsilon}{s}\Big)^\gamma \Theta(R) +C\Theta(R).
\end{aligned}
\end{equation}
Then substituting (\ref{ineq.v-to-u}) into (\ref{ineq.K.rpe}), we arrive at
\begin{equation*}
K(r;p_\varepsilon)\le C\bigg\{\Big(\frac{s}{r}\Big)^{d/2}\Big(\frac{\varepsilon}{s}\Big)^\gamma+\Big(\frac{r}{s}\Big)^\eta+\Big(\frac{\varepsilon}{s}\Big)^\eta+\frac{\varepsilon}{r}\bigg\}\Theta(R),
\end{equation*}
which implies (\ref{ineq.p.step2}) if we choose 
\begin{equation*}
s=\left\{
\begin{aligned}
&r\bigg(\frac{r}{\varepsilon}\bigg)^{2\gamma/(d-2\gamma+2\eta)}, &\text{ if } r<R^{\frac{d-2\gamma+2\eta}{d+2\eta}}\varepsilon^{\frac{2\gamma}{d+2\eta}},\\
&R/4, &\text{ if } r\ge R^{\frac{d-2\gamma+2\eta}{d+2\eta}}\varepsilon^{\frac{2\gamma}{d+2\eta}}.
\end{aligned}
\right.
\end{equation*}
\\

Step 3: Now we are ready to prove the $L^\infty$ estimate for pressure, i.e.,
\begin{equation}\label{ineq.pe.infity}
\norm{p_\e - \average_{D_{2R}} p_\e}_{L^\infty(D_R)}
\le C\Theta(4R).
\end{equation}
Observe that this implies our desired oscillation estimate for $p_\e$, since
\begin{equation*}
\osc{D_R}{p_\e} \le 2\norm{p_\e - \average_{2R} p_\e}_{L^\infty(D_R)} \le C\Theta(4R).
\end{equation*}
Thus, it is sufficient to prove (\ref{ineq.pe.infity}). 

For any $\varepsilon<r<R/4$, there must exist some $k,\ell\in \mathbb{Z}$ such that $2^k r\le s_0\le 2^{k+1}r$ and $2^\ell r\le R \le 2^{\ell+1}r$. Using the triangle inequalities and (\ref{ineq.p.step2}), we obtain 
\begin{equation*}
\begin{aligned}
\left|\average_{D_r} p_\varepsilon-\average_{D_{R}} p_\varepsilon\right|&\le \sum_{i=0}^k K(2^ir;p_\varepsilon) +\sum_{i=k+1}^\ell K(2^ir;p_\varepsilon)\\
&\le C\Bigg\{\bigg(\frac{\varepsilon}{r}\bigg)^\sigma+\bigg(\frac{\varepsilon}{R}\bigg)^{\frac{2\gamma\eta}{d+2\eta}}+1+\bigg(\frac{\varepsilon}{R}\bigg)^{\frac{d+2\eta-2\gamma}{d+2\eta}}\Bigg\}\Theta(R)\\
&\le C\Theta(R).
\end{aligned}
\end{equation*}
Obviously, the same argument implies that for any $x\in D_R$,
\begin{equation}\label{ineq.pe2R}
\left|\average_{D_\varepsilon(x)} p_\varepsilon-\average_{D_{R}(x)} p_\varepsilon\right|\le C\Theta(2R).
\end{equation}
Moreover, by a blow-up argument, we have
\begin{equation}\label{ineq.p02e}
\left|p_\varepsilon(x)-\average_{D_\varepsilon(x)} p_\varepsilon\right|\le C\Theta(2\varepsilon)\le C\Theta(R),
\end{equation}
Therefore, by (\ref{ineq.pe2R}), (\ref{ineq.p02e}) and a similiar argument for the estimate of $K(R;p_\e)$, we conclude
\begin{equation*}
\begin{aligned}
\left|p_\varepsilon(x)-\average_{D_{2R}} p_\varepsilon\right|&\le \left|p_\varepsilon(x)-\average_{D_\varepsilon(x)} p_\varepsilon\right|+\left|\average_{D_\varepsilon(x)}p_\varepsilon-\average_{D_R(x)} p_\varepsilon\right|+\left|\average_{D_R(x)}p_\varepsilon-\average_{D_{2R}} p_\varepsilon\right|\\
&\le C\Theta(4R),
\end{aligned}
\end{equation*}
for any $x\in D_R$. This proves (\ref{ineq.pe.infity}) and hence the theorem.
\end{proof}

\begin{remark}
	Theorem \ref{thm.boundary-estimate} will be used mainly in Theorem \ref{thm.green-function.pointwise} for the estimates of Green's functions and in Lemma \ref{lem.Dirichlet-corrector.estimate} for the estimates of Dirichlet correctors. We mention that in order to establish the estimates of fundamental solutions in Theorem \ref{thm.FundSol}, only interior uniform estimates will be involved in an unbounded domain $\R^d\setminus \{y\}$. In this case, Theorem \ref{thm.boundary-estimate} is used simply by replacing $D_R$ with $B_R$.
\end{remark}
\section{Estimates of Green's Functions}
This section is devoted to establishing the existence of the Green's functions for the Stokes system (\ref{def.Dirichlet.Stokes}) and their corresponding pointwise estimates in a bounded $C^{1,\eta}$ domain. The estimates of Green's functions themselves will play essential roles in the future study of the Dirichlet problem of Stokes systems in periodic homogenization. 

To begin with, we mention the useful symmetry property for $G_\e$ (see \cite{ChoiLee15})
\begin{equation}\label{eq.symm}
G^*_\e(x,y) = G_\e(y,x)^T.
\end{equation}
However, $\Pi_\e(x,y)$ does not possess such symmetry property since the positions of $x$ and $y$ in $\Pi_\e(x,y)$ are not of equal level, even if $\cL_\e$ is self-adjoint. Roughly speaking, $\Pi_\e(x,y)$ behaves more like $\nabla_x G_\e(x,y)$, which can be seen from the following main theorem of this section.

\begin{theorem}\label{thm.green-function.pointwise}
	Let $\Omega$ be a bounded $C^{1,\eta}$ domain and $A$ satisfies (\ref{cond.ellipticity}), (\ref{cond.periodicity}) and (\ref{cond.holder}). Then the Green's functions $(G_\e,\Pi_\e)$ exist and are unique for Stokes system (\ref{def.Dirichlet.Stokes}). Moreover, we have
	\begin{enumerate}
		\item[(i)] Estimates for $G_\e$:
		\begin{equation}\label{est.ptG}
		|G_\e(x,y)| \le C\min \bigg\{ \frac{1}{|x-y|^{d-2}}, \frac{\delta(x)}{|x-y|^{d-1}}, \frac{\delta(y)}{|x-y|^{d-1}}, \frac{\delta(x)\delta(y)}{|x-y|^{d}} \bigg\}.
		\end{equation}
		
		\item[(ii)] Estimates for the first-order derivatives of $G_\e$:
		\begin{equation}\label{est.dxG}
		|\nabla_x G_\e(x,y)| \le C\min \bigg\{ \frac{1}{|x-y|^{d-1}}, \frac{\delta(y)}{|x-y|^{d}} \bigg\},
		\end{equation}
		and
		\begin{equation}\label{est.dyG}
		|\nabla_y G_\e(x,y)| \le C\min \bigg\{ \frac{1}{|x-y|^{d-1}}, \frac{\delta(x)}{|x-y|^{d}} \bigg\}.
		\end{equation}
		
		\item[(iii)] Estimate for the mixed derivatives of $G_\e$:
		\begin{equation}\label{est.dxdyG}
		|\nabla_x \nabla_y G_\e(x,y)| \le \frac{C}{|x-y|^{d}}.
		\end{equation}
		
		\item[(iv)] Estimates for $\Pi_\e$:
		\begin{equation}\label{est.Piy}
		|\Pi_\e(x,y) - \Pi_\e(z,y)| \le C\min \bigg\{ \frac{\delta(y)}{|x-y|^{d}} + \frac{\delta(y)}{|z-y|^d}, \frac{1}{|x-y|^{d-1}} + \frac{1}{|z-y|^{d-1}} \bigg\}.
		\end{equation}
		
		\item[(v)] Estimates for the derivatives of $\Pi_\e$:
		\begin{equation}\label{est.dxyPi}
		|\nabla_y \Pi_\e(x,y) - \nabla_y \Pi_\e (z,y)| \le C\bigg\{ \frac{1}{|x-y|^d} + \frac{1}{|z-y|^d} \bigg\}.
		\end{equation}
	\end{enumerate}
	The estimates {\em (i) - (v)} for  $(G_\e, \Pi_\e)$ are also valid for the adjoint Green's functions $(G_\e^*, \Pi_\e^*)$. The constant $C$ above depends only on $d,\eta, A$ and $\Omega$.
\end{theorem}

\begin{proof}
	The existence and uniqueness of the Green's functions as well as a global pointwise estimate for $G_\e$ (the first part of (\ref{est.ptG})) were proved in \cite{ChoiLee15} under conditions (A0) - (A2) therein, which are obviously guaranteed by the uniform Lipschitz estimate of the velocity in Theorem \ref{thm.boundary-estimate}. Then, all the remaining estimates in (i) and (ii) follow from a standard argument outlined below. 
	
	First, by the Lipschitz estimate in Theorem \ref{thm.boundary-estimate} for $u_\e$, the first part of (\ref{est.ptG}) implies the first parts of (\ref{est.dxG}) and (\ref{est.dyG}) (the symmetry property (\ref{eq.symm}) is used for (\ref{est.dyG})). These further imply the second and third parts of (\ref{est.ptG}) by the fact that $G_\e(x,y)$ vanishes on the boundary and the fundamental theorem of calculus. Now by employing the Lipschitz estimate again, the second and third parts of (\ref{est.ptG}) lead to the second parts of (\ref{est.dyG}) and (\ref{est.dxG}), respectively. Finally, the last part of (\ref{est.ptG}) follows from either the second part of (\ref{est.dxG}) or the second part of (\ref{est.dyG}), as well as the fundamental theorem of calculus.
	 
	Now we turn to the proof of (iv), while the proof of (iii) will be given later together with (v). Let $y\in \Omega$ be fixed. Without loss of generality, we assume $|x-y| \le |z-y|$. Then it is sufficient to show
	\begin{equation}\label{est.Piysp}
	|\Pi_\e(x,y) - \Pi_\e(z,y)| \le C\min \bigg\{ \frac{\delta(y)}{|x-y|^{d}}, \frac{1}{|x-y|^{d-1}} \bigg\}.
	\end{equation}
	Note that Theorem \ref{thm.boundary-estimate} shows that the oscillation of pressure in $D_r$ away from $y$ is bounded by the average of velocity over a larger $D_{4r}$. So if we can choose a family of balls with suitable sizes, covering a connecting path from $x$ to $z$, then we are able to control the maximum oscillation of $\Pi_\e$ between $x$ and $z$.
	
	Let $r = |x-y|$. Consider sets $E_k = D_{2^k r}(y)\setminus D_{2^{k-1}r}(y)$. Observe that for each $k\ge 1$, $E_k$ can be covered by at most $N$ balls $\{B_{kj}: j =1,2,\cdots, N\}$ satisfying $\text{diam}(B_{kj}) \ge c2^k r$ and $\text{dist}(y,8B_{kj}) \ge C2^k r$, where $N$ depends only on $d$. As a result,
	\begin{equation*}
	\Omega\setminus D_r(y) = \bigcup_{k\ge 1} E_k \subset \bigcup_{k,j} B_{kj}.
	\end{equation*}
	Now for each $B_{kj}$, by the oscillation estimate for $p_\e$ in Theorem \ref{thm.boundary-estimate} as well as (\ref{est.ptG}), we have
	\begin{align*}
	\osc{B_{kj} \cap \Omega}{\Pi_\e(\cdot,y)} &\le \frac{C}{\text{diam}(B_{kj})}\bigg( \average_{2B_{kj} \cap \Omega} |G_\e(w,y)|^2 dw\bigg)^{1/2} \\
	&\le C\min\bigg\{\frac{\delta(y)}{(2^kr)^{d}},\, \frac{1}{(2^k r)^{d-1}} \bigg\}.
	\end{align*}
	Summing over all $k$ and $j$, we have
	\begin{align}\label{est.PiOmg-Dr}
	\begin{aligned}
	\osc{\Omega\setminus D_r(y)}{\Pi_\e(\cdot,y)} \le \sum_{k,j} \osc{B_{kj}\cap \Omega}{\Pi_\e(\cdot,y)} & \le \sum_{k\ge 1} \sum_{j=1}^N C\min\bigg\{\frac{\delta(y)}{(2^kr)^{d}},\, \frac{1}{(2^k r)^{d-1}} \bigg\}
	\\
	 & \le CN\min\bigg\{\frac{\delta(y)}{r^{d}},\, \frac{1}{r^{d-1}} \bigg\},
	\end{aligned}
	\end{align}
	which implies (\ref{est.Piysp}) and hence (\ref{est.Piy}).

	Next, we deal with (iii) and (v). Indeed, for any fixed $1\le \ell \le d$, we notice that $(\frac{\partial}{\partial y_\ell}G_\e(\cdot,y), \frac{\partial}{\partial y_\ell}\Pi_\e(\cdot,y))$ is a solution of
	\begin{equation}\label{eq.dyGreen}
	\left\{
	\begin{aligned}
	\cL_\e\bigg(\frac{\partial}{\partial y_\ell}G_\e(\cdot,y) \bigg)+\nabla \frac{\partial}{\partial y_\ell}\Pi_\e(\cdot,y) &= 0 \qquad &\text{ in }& \Omega\setminus\{y\},\\
	\text{div} \bigg(\frac{\partial}{\partial y_\ell} G_\e(\cdot,y) \bigg) &= 0 \qquad & \text{ in }& \Omega,\\
	\frac{\partial}{\partial y_\ell} G_\e(\cdot,y)& = 0 \qquad & \text{ on }&\partial\Omega.
	\end{aligned}
	\right.	
	\end{equation}
	Note that the boundary condition above is due to the second part of (\ref{est.dyG}). (One can justify (\ref{eq.dyGreen}) by first considering the difference $G_\e(\cdot,y+h) - G_\e(\cdot,y)$ and applying a standard regularity argument in PDEs.)
	
	Let $y\in \Omega$ be fixed. Then, by the uniform Lipschitz estimate for the velocity in Theorem \ref{thm.boundary-estimate} and (\ref{est.dyG}), for any $\ell =1,2,\cdots, d$, we have
	\begin{equation*}
	\bigg|\nabla_x\frac{\partial}{\partial y_\ell}G_\e(x,y) \bigg| \le C \bigg( \average_{D_{r/2}(x)} \bigg|\frac{\partial}{\partial y_\ell}G_\e(z,y)\bigg|^2 dz \bigg)^{1/2} \le \frac{C}{|x-y|^d},
	\end{equation*}
	which proves (iii) as desired.
	
	Finally, to see (v), we use the first parts of (\ref{est.dxG}) and (\ref{est.dyG}), and apply  the same argument for (iv) to the system (\ref{eq.dyGreen}) to obtain
	\begin{equation}\label{est.dyPi}
	\osc{\Omega\setminus D_r(y)}{\nabla_y \Pi_\e(\cdot,y)} \le \frac{C}{r^d}, \qquad \text{for any } r>0.
	\end{equation}
	This implies the desired estimate (\ref{est.dxyPi}) and the proof is now  complete.
\end{proof}

\begin{remark}\label{rmk.Pi.Bdry}
	We notice that estimates for $\Pi_\e$ in Theorem \ref{thm.green-function.pointwise} are in the form of difference. By taking advantage of the assumption in the Definition \ref{def.GreenFunc} that $\int_{\Omega} \Pi_\e(z,y) dy = 0$, it can be shown by integrating (\ref{est.Piy}) with respect to $z$ over $\Omega$ that
	\begin{equation}\label{est.Pixy}
	|\Pi_\e(x,y)| \le C\min \bigg\{ \frac{\delta(y) \ln(R/\delta(y)+2)}{|x-y|^{d}}, \frac{1}{|x-y|^{d-1}} \bigg\}.
	\end{equation}
	This is a global pointwise estimate for $\Pi_\e$, though the extra logarithm makes it less useful. Nevertheless, a basic property we can see from (\ref{est.Pixy}) is that $\Pi_\e(x,y) = 0$ for any $y\in \partial\Omega$.
\end{remark}

With the existence and corresponding estimates of the Green's functions, we introduce an integral representation, derived by the integration by parts, for the velocity component of the weak solution of (\ref{def.Dirichlet.Stokes}):

\begin{align}\label{eq.intRepresent}
\begin{aligned}
u^\alpha_\e(x) &= \int_{\Omega} G_\e^{\alpha\beta}(x,y) F^\beta(y) dy - \int_{\Omega} \frac{\partial}{\partial y_j}G_\e^{\alpha\beta}(x,y) h^\beta_j(y) dy - \int_{\Omega} \Pi_\e^{*\alpha}(y,x) g(y) dy \\
& \qquad + \int_{\partial\Omega} \bigg[- n_i a^{\beta\gamma}_{ji}(y/\e)\frac{\partial}{\partial y_j} G_\e^{\alpha\beta}(x,y) + \Pi_\e^{*\alpha}(y,x) n^\gamma \bigg]f^\gamma(y) d\sigma(y).
\end{aligned}
\end{align}
For simplicity, we define
\begin{equation*}
\Upsilon_\e^{\alpha\gamma}(x,y) = - n_i a^{\beta\gamma}_{ji}(y/\e)\frac{\partial}{\partial y_j} G_\e^{\alpha\beta}(x,y),
\end{equation*}
and thus can briefly write (\ref{eq.intRepresent}) as
\begin{align}\label{def.intRep}
\begin{aligned}
u_\e(x) &= \int_{\Omega} G_\e(x,y) F(y) dy - \int_{\Omega} \nabla_y G_\e(x,y) h(y)dy - \int_{\Omega} \Pi_\e^*(y,x) g(y) dy \\
& \qquad + \int_{\partial\Omega} \big[\Upsilon_\e(x,y) + \Pi_\e^*(y,x)\otimes n \big]f(y) d\sigma(y).
\end{aligned}
\end{align}

Generally, there is no simple integral representation formula similar to (\ref{def.intRep}) for the pressure component $p_\e$. Nevertheless, in a particular case, i.e., $h = 0,\, g = 0$ and $f= 0$, we actually have the integral representation for both $u_\e$ and $p_\e$, namely,
\begin{equation}\label{def.intRep.onlyF}
\left\{
\begin{aligned}
u_\e(x) & = \int_{\Omega} G_\e(x,y)F(y) dy, \\
p_\e(x) & = \int_{\Omega} \Pi_\e(x,y) \cdot F(y) dy.
\end{aligned}
\right.
\end{equation}

To end this section and for future applications, we shall mention the fundamental solutions of Stokes system (\ref{def.Dirichlet.Stokes}) corresponding to a special case $\Omega = \R^d$. As usual, we call the pair $(\Gamma_\e, Q_\e)$ 
the fundamental solutions of the Stokes system (\ref{def.Dirichlet.Stokes}), if it satisfies, in the sense of distribution, 
\begin{equation}\label{def.FundSol}
\left\{
\begin{aligned}
\cL_\e(\Gamma_\e(\cdot,y))+\nabla Q_\e(\cdot,y) &= \delta_y I \qquad &\text{ in }& \R^d,\\
\text{div}(\Gamma_\e(\cdot,y)) &= 0 \qquad & \text{ in }& \R^d,\\
\lim_{|x-y|\to \infty} \Gamma_\e(x,y)& = 0.
\end{aligned}
\right.	
\end{equation}
The precise definition for the fundamental solutions for Stokes system with variable coefficients can be found in \cite{ChoiYang17}. Again, if $(\Gamma_\e^* ,Q_\e^*)$ is the the adjoint fundamental solution for the adjoint Stokes system in $\R^d$, then
\begin{equation*}
\Gamma_\e^*(x,y) = \Gamma_\e(y,x)^{T}.
\end{equation*}
Then we have the following estimates for the fundamental solution $(\Gamma_\e, Q_\e)$.
\begin{theorem}\label{thm.FundSol}
	Assume that $A$ satisfies (\ref{cond.ellipticity}), (\ref{cond.periodicity}) and (\ref{cond.holder}). Then the fundamental solutions $(\Gamma_\e ,Q_\e)$ for Stokes system exist and are unique (up to a constant for $Q_\e$). Moreover, there exists some constant $\overline{Q} \in \R^d$ such that for any $x,y\in \R^d, x\neq y$, we have
	\begin{equation}\label{est.Gammae}
	|\Gamma_\e(x,y)| \le \frac{C}{|x-y|^{d-2}},
	\end{equation}
	\begin{equation}\label{est.dGammae}
	|\nabla_x \Gamma_\e(x,y)| + |\nabla_y \Gamma_\e(x,y)| + |Q_\e(x,y) - \overline{Q}| \le \frac{C}{|x-y|^{d-1}},
	\end{equation}
	\begin{equation}\label{est.ddGammae}
	|\nabla_x \nabla_y \Gamma_\e(x,y)| + |\nabla_y Q_\e(x,y)| \le \frac{C}{|x-y|^{d}},
	\end{equation}
	All the constants $C$ above depend only on $d$ and $A$.
\end{theorem}

The proof of Theorem \ref{thm.FundSol} follows from the primary results in \cite{ChoiYang17} and the same argument as in Theorem \ref{thm.green-function.pointwise}. The only point we would like to emphasize is the existence of the constant $\overline{Q}$. To see this, for a given $y\in \R^d$, the estimate (\ref{est.PiOmg-Dr}) gives 
\begin{equation*}
\osc{\R^d\setminus D_r(y)}{Q_\e(\cdot,y)} \le C r^{1-d},
\end{equation*}
for any $r>0$. This implies that $Q_\e(\cdot,y)$ has a limit $\overline{Q}(y)$ such that
\begin{equation}\label{est.QeQbar}
|Q_\e(x,y) - \overline{Q}(y)| \le \frac{C}{|x-y|^{d-1}}.
\end{equation}
Then it remains to show $\overline{Q}(y)$ is independent of $y$. Recall that in \cite[Definition 2.4]{ChoiYang17}, the weak solution of
\begin{equation*}
\left\{
\begin{aligned}
\cL_\e^*(u_\e)+\nabla p_\e &= 0 \qquad &\text{ in }& \R^d,\\
\text{div}(u_\e) &= g \qquad & \text{ in }& \R^d,\\
\end{aligned}
\right.
\end{equation*}
with $g\in C_0^\infty(\R^d)$ and $\int_{\R^d} g = 0$ is given by
\begin{equation}\label{eq.ueGamaQ}
u_\e(x) = - \int_{\R^d} Q_\e(x,y) g(y)dy.
\end{equation}
Now since $u_\e$ belongs to $Y_2^1 = \{f: |f|\in L^{2d/(d-2)}(\R^d) \text{ and } |\nabla f|\in L^2(\R^d)\}$ (see \cite{ChoiYang17}) and is uniformly Lipschitz, we have $\lim_{|x|\to \infty}|u_\e(x)| = 0$. Thus, it follows from (\ref{est.QeQbar}) and (\ref{eq.ueGamaQ}) that
\begin{equation*}
\lim_{|x|\to \infty} \int_{\R^d} Q_\e(x,y) g(y) dy = \int_{\R^d} \overline{Q}(y) g(y) dy = 0,
\end{equation*}
for any $g\in C_0^\infty(\R^d)$ with $\int_{\R^d} g = 0$. This implies that $\overline{Q}$ must be a constant.

\section{Asymptotic Expansion of $G_\e$}

\subsection{Local $L^\infty$ Estimate}
First, we will use (\ref{def.intRep}) to prove an $L^p$ estimate for the velocity $u_\e$. For a bounded $C^{1,\eta}$ domain $\Omega$, we define the usual non-tangential cone for a point $\hat{x} \in \partial\Omega$ by
\begin{equation*}
	\mathcal{C}(\hat{x}) = \{x\in \Omega: |x-\hat{x}| < 2 \delta(x) \}.
\end{equation*}
For a function $F$ defined in $\Omega$, we denote by $(F)^*$ the non-tangential maximal function on $\partial\Omega$, i.e., $(F)^*(\hat{x}) = \sup_{x\in \mathcal{C}(\hat{x})} |F(x)|$. For $f\in L^1(\partial\Omega)$, denote $\mathcal{M}_{\partial\Omega}(f)$ by the Hardy-Littlewood maximal function, i.e.,
\begin{equation*}
	\mathcal{M}_{\partial\Omega}(f)(\hat{x}) = \sup_{r>0} \average_{\partial\Omega \cap B(\hat{x},r)} |f(y)|d\sigma(y).
\end{equation*}

\begin{theorem}\label{thm.Lpqr}
	Let $\Omega$ be a bounded $C^{1,\eta}$ domain and $A$ satisfy (\ref{cond.ellipticity}), (\ref{cond.periodicity}) and (\ref{cond.holder}). Let $(u_\e,p_\e)$ be the weak solution of (\ref{def.Dirichlet.Stokes}) with data in appropriate spaces. Then for any $1 \le p\le \infty$,
	\begin{equation}\label{est.ueLp}
		\norm{u_\e}_{L^p(\Omega)} \le C\big(\norm{f}_{L^q(\partial\Omega)} + \norm{g}_{L^r(\Omega)} + \norm{h}_{L^r(\Omega)} + \norm{F}_{L^s(\Omega)} \big),
	\end{equation}
	where $q,r$ and $s$ satisfy
	\begin{equation}\label{cond.pq}
		\left\{
		\begin{aligned}
			&q \ge \frac{(d-1)p}{d}, \quad \text{and} \quad 1< q\le \infty, \\
			&\frac{1}{r} - \frac{1}{p} < \frac{1}{d}, \quad \text{and} \quad 1\le r\le \infty,\\
			& \frac{1}{s} - \frac{1}{p} < \frac{2}{d},  \quad \text{and} \quad 1\le s \le \infty.
		\end{aligned}
		\right.
	\end{equation}
\end{theorem}
\begin{proof}
	By the uniqueness, the component $u_\e$ of the weak solution can be given by the integral representation formula (\ref{def.intRep}). In view of the assumption $\int_{\Omega}\Pi_\e^*(z,x)dz = 0$, we can further write
	\begin{align*}
		u_\e(x) & = \int_{\Omega} G_\e(x,y) F(y) dy + \int_{\Omega} \nabla_y G_\e(x,y) h(y) dy  \\
		& \qquad- \int_{\Omega} \bigg[ \Pi_\e^*(y,x) - \average_\Omega \Pi_\e^*(z,x) dz \bigg] g(y) dy + \int_{\partial\Omega} \Upsilon_\e(x,y)f(y)d\sigma(y) \\
		& \qquad + \int_{\partial\Omega} \bigg\{ \bigg[ \Pi_\e^*(y,x) - \average_\Omega \Pi_\e^*(z,x) dz\bigg] \otimes n \bigg\} f(y) d\sigma(y) \\
		& = u_\e^{(1)}(x) + u_\e^{(2)}(x) + u_\e^{(3)}(x) + u_\e^{(4)}(x) + u_\e^{(5)}(x),
	\end{align*}
	where $u_\e^{(k)}, k=1,2,3,4,5$, denote the five integrals in proper order in the last identity.
	
	We first estimate $u_\e^{(1)}(x)$. By (\ref{est.ptG}), one has
	\begin{equation*}
		|u_\e^{(1)}(x)| \le C\int_{\Omega} \frac{|F(y)|}{|x-y|^{d-2}} dy.
	\end{equation*}
	It follows from \cite[Lemma 7.12]{GilbargTrudinger83} that
	\begin{equation}\label{est.ue1}
		\norm{u_\e^{(1)}}_{L^p} \le C\norm{F}_{L^s(\Omega)}, \quad \text{with } \frac{1}{s} - \frac{1}{p} < \frac{2}{d}. 
	\end{equation}
	
	The estimate of $u_\e^{(2)}$ is similar to $u_\e^{(1)}$ by using (\ref{est.dyG}). Thus, employing \cite[Lemma 7.12]{GilbargTrudinger83} again, we obtain
	\begin{equation}\label{est.ue2}
		\norm{u_\e^{(2)}}_{L^p(\Omega)} \le C \norm{g}_{L^r(\Omega)}, \qquad \text{with } \frac{1}{r} - \frac{1}{p} < \frac{1}{d}.
	\end{equation}
	
	Next, to estimate $u_\e^{(3)}$, note that (\ref{est.Piy}) implies
	\begin{align*}
		|u_\e^{(3)}(x)| &= \bigg| \int_{\Omega}\average_\Omega \big[ \Pi_\e^*(y,x) -  \Pi_\e^*(z,x) \big] dz g(y) dy \bigg| \\
		& \le C\int_{\Omega}\average_\Omega \bigg( \frac{1}{|x-y|^{d-1}} + \frac{1}{|x-z|^{d-1}} \bigg) |g(y)| dzdy \\
		& \le C \int_{\Omega}\average_\Omega \frac{|g(y)|}{|x-y|^{d-1}}  dzdy +  C \int_{\Omega}\average_\Omega \frac{|g(y)|}{|x-z|^{d-1}}  dzdy.
	\end{align*}
	Clearly, the second term above is bounded by $C\norm{g}_{L^1(\Omega)}$ and the first term can be handle by \cite[Lemma 7.12]{GilbargTrudinger83} analogously.

	We see that  
	\begin{equation}\label{est.ue3}
		\norm{u_\e^{(3)}}_{L^p(\Omega)} \le C \norm{g}_{L^r(\Omega)}, \qquad \text{with } \frac{1}{r} - \frac{1}{p} < \frac{1}{d}.
	\end{equation}
	
	To estimate $u_\e^{(4)}$, we recall that (\ref{est.dyG}) gives $|\Upsilon_\e(x,y)| \le C\delta(x)|x-y|^{-d}$.
	Let $\hat{x} \in \partial\Omega$ such that $x\in \mathcal{C}(\hat{x})$. Then,
	\begin{equation*}
		|u_\e^{(4)}(x)| \le C \int_{\partial\Omega} \frac{\delta(x)}{|x-y|^d} |f(y)|d\sigma(y) \le C\mathcal{M}_{\partial\Omega}(f)(\hat{x}).
	\end{equation*}
	Therefore, using the $L^q$ boundedness of the Hardy-Littlewood maximal function, we derive that $\norm{(u_\e^{(4)})^*}_{L^q(\partial\Omega)} \le C \norm{f}_{L^q(\partial\Omega)}$ for any $1<q\le \infty$.
	By a general result concerning the non-tangential maximal function (see \cite[Remark 9.3]{KenigLinShen1301} or \cite[Lemma 3.3]{WeiZhang14}), we have
	\begin{equation}\label{est.ue4}
		\norm{u_\e^{(4)}}_{L^p(\Omega)} \le C\norm{f}_{L^q(\partial\Omega)}, \quad \text{with } q \ge \frac{(d-1)p}{d}.
	\end{equation}
	
	It is remaining to estimate $u_\e^{(5)}$. To this end, we write
	\begin{align*}
		|u_\e^{(5)}(x)| = \int_{\partial\Omega}\average_\Omega | \Pi_\e^*(y,x) -  \Pi_\e^*(z,x) ||f(y)| dz  d\sigma(y).
	\end{align*}
	In view of (\ref{est.Piy}), we decompose the set $\partial\Omega \times \Omega$, according to a fixed $x$, into $E_x$ and $ E_x'$, where
	\begin{equation*}
		E_x = \{(y,z)\in \partial\Omega \times \Omega: |x-y|>|x-z| \}  \quad \text{and} \quad E_x' = \partial\Omega \times \Omega\setminus E_x.
	\end{equation*}
	Now it follows from (\ref{est.Piy}) that
	\begin{align*}
		|u_\e^{(5)}(x)| & \le C \iint_{E_x} \frac{|f(y)|}{|x-z|^{d-1}} dzd\sigma(y) + C\iint_{E_x'} \frac{\delta(x) |f(y)|}{|x-y|^d} d\sigma(y) dz\\
		& \le C\int_{\partial\Omega} |f(y)|d\sigma(y) + C\mathcal{M}_{\partial\Omega}(f)(\hat{x}) \\
		& \le C\mathcal{M}_{\partial\Omega}(f)(\hat{x}),
	\end{align*}
	where $\hat{x} \in \partial\Omega$ is the point such that $x\in \mathcal{C}(\hat{x})$. As before, we then obtain
	\begin{equation}\label{est.ue5}
		\norm{u_\e^{(5)}}_{L^p(\Omega)} \le C\norm{(u_\e^{(5)})^*}_{L^q(\partial\Omega)} \le  C\norm{f}_{L^q(\partial\Omega)}.
	\end{equation}
	
	Finally, (\ref{est.ueLp}) follows from (\ref{est.ue1}), (\ref{est.ue2}), (\ref{est.ue3}), (\ref{est.ue4}) and (\ref{est.ue5}) as desired.
\end{proof}

\begin{remark}\label{rmk.Max}
	If $F = 0,\, h =0$ and $g = 0$, the proof of Theorem \ref{thm.Lpqr} gives rise to the estimate of non-tangential maximal function for Dirichlet problem with $L^p$ data, i.e.,
	\begin{equation*}
		\norm{(u_\e)^*}_{L^q(\partial\Omega)} \le  C\norm{f}_{L^q(\partial\Omega)},
	\end{equation*}
	for $1<q\le \infty$. In particular, the case $p = \infty$ gives the Miranda-Agmon maximum principle, namely, $\norm{u_\e}_{L^\infty(\Omega)} \le C\norm{f}_{L^\infty(\partial\Omega)}$.
\end{remark}

The following local $L^\infty$ estimate is a key step to show (\ref{est.mainG}).

\begin{lemma}\label{lem.Linfty-estimate}
Let $\Omega$ be a $C^{1,\eta}$ domain and $(u_\varepsilon,p_\varepsilon)$ satisfy
\begin{equation}\label{def.general-system-e-4r}
\left\{
\begin{aligned}
\cL_\e(u_\e)+\nabla p_\e &= F+\rdiv{(h)} &\qquad \rtxt{ in }D_{4r},\\
\rdiv(u_\e) &=g &\qquad\rtxt{ in }D_{4r},\\
u_\e& =f &\qquad\rtxt{ on }\Delta_{4r}.
\end{aligned}
\right.	
\end{equation}
with $F\in L^s(\Omega;\R^d),\, h\in L^p(\Omega;\R^{d\times d}), \, g\in L^p(\Omega)$ and $f\in L^\infty(\Omega;\R^d)$ for some $s>d/2$ and $p>d$. Then
\begin{equation}
\begin{aligned}
\|u_\varepsilon\|_{L^\infty(D_r)}\le C\bigg\{\average_{D_{4r}} |u_\varepsilon|+&r^2\bigg(\average_{D_{4r}}|F|^s\bigg)^{1/s}+ r\bigg(\average_{D_{4r}} |h|^p\bigg)^{1/p} \\
&\qquad+r\bigg(\average_{D_{4r}} |g|^p\bigg)^{1/p}+\|f\|_{L^\infty(\Delta_{4r})}\bigg\},
\end{aligned}
\end{equation}
where $C$ depends only on $d,s,p,A$ and $\Omega$.
\end{lemma}

\begin{proof}
By a rescaling argument, we may assume $r=1$. Let $\widetilde{D}$ be a $C^{1,\eta}$ domain such that $D_2\subset \widetilde{D}\subset D_4$. Consider $u_\varepsilon=u_\varepsilon^{(1)}+u_\varepsilon^{(2)}$, $p_\varepsilon=p_\varepsilon^{(1)}+p_\varepsilon^{(2)}$, where $(u_\varepsilon^{(1)},p_\varepsilon^{(1)})$, $(u_\varepsilon^{(2)},p_\varepsilon^{(2)})$ satisfy the following systems
\begin{equation*}
\left\{
\begin{aligned}
\cL_\e(u_\e^{(1)})+\nabla p_\e^{(1)} &= 0 & \text{ in }\widetilde{D},\\
\rdiv(u_\e^{(1)}) &=\widetilde{C} &\text{ in }\widetilde{D},\\
u_\e^{(1)}& =0 &\text{ on }\partial\widetilde{D}\cap\partial \Omega,\\
u_\e^{(1)}& =u_\varepsilon &\text{ on }\partial\widetilde{D}\backslash\partial \Omega,
\end{aligned}
\right.	
\end{equation*}
and
\begin{equation*}
\left\{
\begin{aligned}
\cL_\e(u_\e^{(2)})+\nabla p_\e^{(2)} &= F+\text{div}(h) & \text{ in }\widetilde{D},\\
\rdiv(u_\e^{(2)}) &=g-\widetilde{C} &\text{ in }\widetilde{D},\\
u_\e^{(2)}& =f &\text{ on }\partial\widetilde{D}\cap\partial \Omega,\\
u_\e^{(2)}& =0 &\text{ on }\partial\widetilde{D}\backslash\partial \Omega,
\end{aligned}
\right.	
\end{equation*}
where the constant $\widetilde{C}$, defined by
$$
\widetilde{C}=\average_{\widetilde{D}} g -\frac{1}{|\widetilde{D}|}\int_{\partial\widetilde{D}\cap \partial \Omega} f\cdot n\, d\sigma,
$$
is applied to adjust the compatibility condition for both systems.

By Theorem \ref{thm.Lpqr}, we see that $u_\e^{(2)}$ is bounded by
\begin{equation}\label{ineq.ue2-Linfty}
\|u_\varepsilon^{(2)}\|_{L^\infty(\widetilde{D}))}\le C\left\{ \|f\|_{L^\infty(\widetilde{\Delta})}+\|F\|_{L^s(\widetilde{D}))}+\|g-\widetilde{C}\|_{L^p(\widetilde{D}))}+\|h\|_{L^p(\widetilde{D})}\right\},
\end{equation}
To handle $u_\e^{(1)}$, for any $t>0$ and $x\in \widetilde{D}$ such that $D_{2t}(x) \subset \widetilde{D}$, we apply the uniform H\"older estimate (see \cite{GuShen15}) to obtain,
\begin{equation*}
\|u_\varepsilon^{(1)}\|_{L^\infty(D_t(x))}\le  C\bigg\{|\widetilde{C}|+\bigg(\average_{D_{2t}(x)}|u_\varepsilon^{(1)}|^2\bigg)^{1/2}\bigg\}.
\end{equation*}
By a convexity argument (see \cite[pp. 1004-1005]{FabesStroock84}), we have
\begin{equation*}
\begin{aligned}
\|u_\varepsilon^{(1)}\|_{L^\infty(D_1)}&\le C\left\{ |\widetilde{C}|+\average_{D_2}\big|u_\varepsilon^{(1)}\big|\right\}\\
&\le C\left\{|\widetilde{C}|+\average_{D_2}|u_\varepsilon|+\|u_\varepsilon^{(2)}\|_{L^\infty(\widetilde{D})}\right\}.\\
\end{aligned}
\end{equation*}
This, together with (\ref{ineq.ue2-Linfty}), leads to
\begin{equation*}
\|u_\varepsilon\|_{L^\infty(D_1)}\le C\left\{\average_{D_{4}} |u_\varepsilon|+\|F\|_{L^s(D_{4})}+\|h\|_{L^p(D_{4})}+\|g\|_{L^p(D_{4})}+\|f\|_{L^\infty(\Delta_{4})}\right\}.
\end{equation*}
The proof is complete.
\end{proof}

\begin{lemma}\label{lem.ue-u0.infty}
Assume $\Omega$ is a $C^{1,\eta}$ domain. Let $(u_\varepsilon,p_\varepsilon)\in H^1(D_{4r};\mathbb{R}^d)\times L^2(D_{4r})$ satisfy (\ref{def.general-system-e-4r}), and $(u_0,p_0)\in W^{2,p}(D_{4r};\mathbb{R}^d)\times W^{1,p}(D_{4r})$ satisfy the corresponding homogenized system, for some $p>d$. Then
\begin{equation}\label{ineq.ue-u0.infty}
\|u_\varepsilon-u_0\|_{L^\infty(D_r)} \le C\bigg\{\average_{D_{4r}} |u_\varepsilon-u_0|+\varepsilon r\bigg(\average_{D_{4r}} |\nabla^2 u_0|^p\bigg)^{1/p}+\varepsilon\|\nabla u_0\|_{L^\infty(D_{4r})}\bigg\}.
\end{equation}
\end{lemma}
\begin{proof}
To utilize Lemma \ref{lem.w-system}, we choose $(V_{\varepsilon,j}^{\beta}(x), T_{\varepsilon,j}^\beta(x))=(\varepsilon\chi_j^\beta(x/\varepsilon)+P_j^\beta(x),\pi_j^\beta(x/\varepsilon))$, and define
\begin{equation}\label{def.wete}
w_\varepsilon=u_\varepsilon-u_0-\varepsilon \chi^\varepsilon\nabla u_0, \quad  \tau_\varepsilon=p_\varepsilon-p_0+\pi^\varepsilon\nabla u_0+\varepsilon q^\varepsilon\nabla^2 u_0.
\end{equation}
Following by Lemma \ref{lem.w-system}, $(w_\varepsilon,\tau_\varepsilon)$ satisfies
\begin{equation}\label{eq.wete}
\left\{
\begin{aligned}
\mathcal{L}_\varepsilon(w_\varepsilon)+\nabla \tau_\varepsilon&=\varepsilon\text{div}\big([\phi^\varepsilon+A^\varepsilon\chi^\varepsilon]\nabla^2u_0\big)+\varepsilon\text{div}\big(q^\varepsilon\nabla^2u_0\big) &\text{ in }D_{4r},\\
\text{div}(w_\varepsilon)&=-\varepsilon\chi^\varepsilon\text{div}(\nabla u_0) &\text{ in }D_{4r},\\
w_\varepsilon&=-\varepsilon\chi^\varepsilon\nabla u_0 &\text{ on }\Delta_{4r}.
\end{aligned}
\right.
\end{equation}
Then by using Lemma \ref{lem.Linfty-estimate}, we have
\begin{equation*}
\|w_\varepsilon\|_{L^\infty(D_r)}\le C\bigg\{\average_{D_{4r}} |w_\varepsilon|+\varepsilon r\bigg(\average_{D_{4r}} |\nabla^2 u_0|^p\bigg)^{1/p}+\varepsilon\|\nabla u_0\|_{L^\infty(\Delta_{4r})}\bigg\},
\end{equation*}
which implies (\ref{ineq.ue-u0.infty}) by a triangle inequality.
\end{proof}

\subsection{Proof of Theorem \ref{thm.main} (i)}
Now with the help of Lemma \ref{lem.ue-u0.infty} and a duality argument, we are able to prove (\ref{est.mainG}).

\begin{proof}[Proof of Theorem \ref{thm.main}, Part (i)]
	We fix $x_0,\, y_0 \in \Omega$ and set $r = |x_0 - y_0|/8$. Let $F\in C_0^\infty(D_r(y_0);\R^d)$ and define
	\begin{equation}\label{eq.ueu0F}
		u_\e(x) = \int_{\Omega}G_\e(x,y)F(y) dy, \quad \text{and} \quad u_0(x) = \int_{\Omega}G_0(x,y)F(y) dy.
	\end{equation}
	Then, in view of (\ref{def.intRep}) (or (\ref{def.intRep.onlyF})), $(u_\varepsilon,p_\varepsilon)$ and $(u_0,p_0)$ satisfy
	\begin{equation*}
		\left\{
		\begin{aligned}
			\cL_\e(u_\e)+\nabla p_\e &= \cL_0 (u_0)+ \nabla p_0 = F &\qquad \text{ in }\Omega,\\
			\tdiv(u_\e) &= \tdiv(u_0) = 0 &\qquad\text{ in }\Omega,\\
			u_\e & = u_0 = 0 &\qquad\text{ on }\partial\Omega.
		\end{aligned}
		\right.	
	\end{equation*}
	It follows from the $W^{2,p}$ estimate of Stokes systems with constant coefficients in $C^{1,1}$ domains \cite{GiaquintaModica82} that for any $1<p<\infty$
	\begin{equation}\label{est.ddu0Lp}
		\norm{\nabla^2 u_0}_{L^p(\Omega)} \le C\norm{F}_{L^p(D_r(y_0))}.
	\end{equation}
	Also, by (\ref{eq.ueu0F}) and (\ref{est.dxG}), we have
	\begin{equation}\label{est.du0inf}
		\norm{\nabla u_0}_{L^\infty(\Omega)} \le C r \bigg( \average_{D_r(y_0)} |F|^p\bigg)^{1/p},
	\end{equation}
	for $p>d$.

	Let $(w_\e, \tau_\e)$ be the same as (\ref{def.wete}). Then, $(w_\e, \tau_\e)$ satisfies the same system (\ref{eq.wete}) in the domain $\Omega$. Thus, it follows from Lemma \ref{lem.ue-u0.infty} that
	\begin{align}\label{est.weinf}
		\begin{aligned}
			\|w_\varepsilon\|_{L^\infty(D_r(x_0))} & \le C\bigg\{\average_{D_{4r}(x_0)} |w_\varepsilon|+\varepsilon r\bigg(\average_{D_{4r}(x_0)} |\nabla^2 u_0|^p\bigg)^{1/p}+\varepsilon\|\nabla u_0\|_{L^\infty(D_{4r}(x_0))}\bigg\} \\
			& \le C\bigg\{\average_{D_{4r}(x_0)} |w_\varepsilon|+\varepsilon r \bigg( \average_{D_r(y_0)} |F|^p\bigg)^{1/p} \bigg\},
		\end{aligned}
	\end{align}
	where we have used (\ref{est.ddu0Lp}) and (\ref{est.du0inf}) in the last inequality.
	
	Now we decompose $(w_\e, \tau_\e) = (w_\e^{(1)}, \tau_\e^{(1)}) + (w_\e^{(2)}, \tau_\e^{(2)})$, where
	\begin{equation*}
		\left\{
		\begin{aligned}
			\mathcal{L}_\varepsilon(w_\varepsilon^{(1)})+\nabla \tau_\varepsilon^{(1)}&=\varepsilon\text{div}\big([\phi^\varepsilon+A^\varepsilon\chi^\varepsilon]\nabla^2u_0\big)+\varepsilon\text{div}\big(q^\varepsilon\nabla^2u_0\big) &\text{ in }\Omega,\\
			\text{div}(w_\varepsilon^{(1)})&=-\varepsilon\chi^\varepsilon\text{div}(\nabla u_0) - \widetilde{C} &\text{ in }\Omega,\\
			w_\varepsilon^{(1)}&=0 &\text{ on }\partial\Omega,
		\end{aligned}
		\right.
	\end{equation*}
	and
	\begin{equation*}
		\left\{
		\begin{aligned}
			\mathcal{L}_\varepsilon(w_\varepsilon^{(2)})+\nabla \tau_\varepsilon^{(2)}&=0 &\text{ in }\Omega,\\
			\text{div}(w_\varepsilon^{(2)})&= \widetilde{C} &\text{ in }\Omega,\\
			w_\varepsilon^{(2)}&= -\e \chi^\e\nabla u_0 &\text{ on }\partial\Omega,
		\end{aligned}
		\right.
	\end{equation*}
	and $\widetilde{C} = -\average_{\Omega} \e \chi^\e \tdiv(\nabla u_0) = - \frac{1}{|\Omega|}\int_{\partial\Omega} \e \chi^\e \nabla u_0 \cdot n d\sigma $ is a constant to adjust the compatibility condition.
	
	For $w_\e^{(1)}$, the energy estimate (\ref{ineq.energy}) and (\ref{est.ddu0Lp}) provide
	\begin{equation*}
		\norm{w_\e^{(1)}}_{H^1(\Omega)} \le C\e\norm{\nabla^2 u_0}_{L^2(\Omega)} \le C\e \norm{F}_{L^2(D_r(y_0))}.
	\end{equation*}
	By the H\"{o}lder's inequality and Sobolev embedding theorem, we have
	\begin{equation}\label{est.we1H1}
		\norm{w_\e^{(1)}}_{L^2(D_r(x_0))} \le Cr \norm{w_\e^{(1)}}_{L^q(\Omega)} \le C\e r\norm{F}_{L^2(D_r(y_0))},
	\end{equation}
	where $q = 2d/(d-2)$. For $w_\e^{(2)}$, we use Theorem \ref{thm.Lpqr} to conclude
	\begin{equation}\label{est.we2inf}
		\norm{w_\e^{(2)}}_{L^\infty(\Omega)} \le C\e \norm{\nabla u_0}_{L^\infty(\Omega)}.
	\end{equation}
	Combining (\ref{est.we1H1}) and (\ref{est.we2inf}), we have
	\begin{equation*}
		\begin{aligned}
			\bigg( \average_{D_r(x_0)} |w_\e|^2 \bigg)^{1/2} &\le C\e r \bigg( \average_{D_r(y_0)} |F|^2 \bigg)^{1/2} + C\e \norm{\nabla u_0}_{L^\infty(\Omega)} \\
			& \le C\e r \bigg( \average_{D_r(y_0)} |F|^p\bigg)^{1/p}.
		\end{aligned}
	\end{equation*}
	This, together with (\ref{est.weinf}) and (\ref{est.du0inf}), leads to
	\begin{equation*}
		|u_\e(x_0) - u_0(x_0)| \le C\e r \bigg( \average_{D_r(y_0)} |F|^p\bigg)^{1/p},
	\end{equation*}
	where $p>d$. Now, in view of (\ref{eq.ueu0F}), it follows by duality that
	\begin{equation}\label{est.GeG0p}
		\bigg(\average_{D_r(y_0)} |G_\e(x_0,y) - G_0(x_0,y)|^{p'} dy\bigg)^{1/p'} \le C\e r^{1-d},
	\end{equation}
	where $p' = p/(p-1)$.
	
	Finally, recall that $G_\e(y,x)^T = G_\e^*(x,y)$. Then, 
	\begin{equation*}
		\left\{
		\begin{aligned}
			\cL_\e^*(G_\e(x_0,\cdot)^T)+\nabla \Pi_\e^*(\cdot,x_0) &= \cL_0^* (G_0(x_0,\cdot)^T)+ \nabla \Pi_0^*(\cdot,x_0) = 0 &\qquad \text{ in }D_r(y_0),\\
			\tdiv(G_\e(x_0,\cdot)^T) &= \tdiv(G_0(x_0,\cdot)^T) = 0 &\qquad\text{ in }D_r(y_0),\\
			G_\e(x_0,\cdot)^T & = G_0(x_0,\cdot)^T = 0 &\qquad\text{ on } \Delta_r(y_0).
		\end{aligned}
		\right.	
	\end{equation*}
	We may apply Lemma \ref{lem.ue-u0.infty} again to conclude that
	\begin{equation*}
		\begin{aligned}
			|G_\e(x_0,y_0) - G_0(x_0,y_0)| &\le C\bigg\{  \average_{D_r(y_0)} |G_\e(x_0,y) - G_0(x_0,y)| dy \\
			& \qquad + \e r \bigg( \average_{D_r(y_0)} |\nabla^2_y G_0(x_0,y)|^p dy\bigg)^{1/p} \\
			&\qquad + \e \norm{\nabla G_0(x_0,\cdot)}_{L^\infty(D_r(y_0))}  \bigg\} \\
			& \le C\e r^{1-d},
		\end{aligned}
	\end{equation*}
	where we have used (\ref{est.GeG0p}) and the following $W^{2,p}$ estimate (see \cite{GiaquintaModica82})
	\begin{equation*}
		\bigg( \average_{D_r(y_0)} |\nabla^2 G_0(x_0,y)|^p dy\bigg)^{1/p} \le Cr^{-2} \norm{G_0(x_0,\cdot)}_{L^\infty(D_{2r}(y_0))} \le Cr^{-d}.
	\end{equation*}
	 This ends the proof.
\end{proof}

\begin{theorem}\label{thm.convergence-rates.Lq}
Suppose that $A$ satisfies (\ref{cond.ellipticity}), (\ref{cond.periodicity}) and \ref{cond.holder}), and $\Omega$ is a bounded $C^{1,\eta}$ domain with $R_0=\text{\rm diam}(\Omega)$. Let $F\in L^2(\Omega;\mathbb{R}^d)$ and $(u_\varepsilon,p_\varepsilon)\in H^1_0(\Omega;\mathbb{R}^d)\times L^2_0(\Omega)$ be a weak solution of 
\begin{equation*}
\left\{
\begin{aligned}
\mathcal{L}_\varepsilon(u_\varepsilon)+\nabla p_\varepsilon &=F &\quad\text{ in }\Omega,\\
\text{\rm div}(u_\varepsilon) & =0 &\quad\text{ in }\Omega,\\
u_\varepsilon &=0 &\quad\text{ on }\partial\Omega.
\end{aligned}
\right.
\end{equation*}
Then if (i) $1<p<d$ and $\frac{1}{q}=\frac{1}{p}-\frac{1}{d}$, or (ii) $p>d$ and $q=\infty$, we have
\begin{equation}\label{ineq.convergence-rates.Lq}
\|u_\varepsilon-u_0\|_{L^q(\Omega)}\le C\varepsilon\|F\|_{L^p(\Omega)}.
\end{equation}
Moreover,
\begin{equation}\label{ineq.convergence-rates.Linfty}
\|u_\varepsilon-u_0\|_{L^\infty(\Omega)}\le C\varepsilon(\ln[\varepsilon^{-1}R_0+2])^{1-\frac{1}{d}}\|F\|_{L^d(\Omega)}.
\end{equation}
\end{theorem}

\begin{proof}
By part (i) of Theorem \ref{thm.main} and the integral representation (\ref{def.intRep.onlyF}), we know that
$$
|u_\varepsilon(x)-u_0(x)|\le C\varepsilon\int_\Omega \frac{|F(y)|}{|x-y|^{d-1}}dy, \quad\text{ for any }x\in \Omega,
$$
which proves (\ref{ineq.convergence-rates.Lq}) for $1<p<d$ and $\frac{1}{q}=\frac{1}{p}-\frac{1}{d}$ by the estimates of fractional integrals \cite[Lemma 7.12]{GilbargTrudinger83}. If $p>d$ and $q=\infty$, (\ref{ineq.convergence-rates.Lq}) holds true by the H\"older's inequality. To see (\ref{ineq.convergence-rates.Linfty}), note that (\ref{est.ptG}) and (\ref{est.mainG}) imply
$$
\begin{aligned}
|u_\varepsilon(x)-u_0(x)|&\le C\int_{D_\varepsilon(x)} \frac{|F(y)|}{|x-y|^{d-2}}dy + C\varepsilon\int_{\Omega\backslash D_\varepsilon(x)} \frac{|F(y)|}{|x-y|^{d-1}}dy\\
&\le C\varepsilon\|F\|_{L^d(\Omega)}+C\varepsilon(\ln[\varepsilon^{-1}R_0+2])^{1-\frac{1}{d}}\|F\|_{L^d(\Omega)}\\
&\le C\varepsilon(\ln[\varepsilon^{-1}R_0+2])^{1-\frac{1}{d}}\|F\|_{L^d(\Omega)},
\end{aligned}
$$
where we have used the H\"older's inequality in the second inequality. The proof is finished.
\end{proof}

\section{Asymptotic Expansions of $\nabla_x G_\varepsilon$ and $\Pi_\varepsilon$}
In this section, we are going to prove part (ii) of Theorem \ref{thm.main}. The key is to prove the adjustable uniform estimates (contained in Theorem \ref{thm.divergence.modifiedC1} and \ref{thm.modified-boundary-estimate}).

\subsection{Adjustable Uniform Estimates}
We first provide an adjustable Lipschitz estimate for the divergence equation $\text{div}(u)=\psi$, which will be useful in deriving a better rate for the asymptotic expansion of $\Pi_\varepsilon$. We encourage the reader to view the following theorem as a sneak peek of our new idea, namely, using the Green's functions to diminish the influence of the H\"older semi-norm of the data as small as possible.

\begin{theorem}\label{thm.divergence.modifiedC1}
Let $\Omega$ be a bounded $C^{1,\eta}$ domain and $R_0=\text{\rm diam}(\Omega)$. Given any $\psi\in C^{0,\eta}(\Omega)$ with $\int_\Omega \psi=0$, there exists a $u\in C^{1,\eta}(\Omega;\mathbb{R}^d)$ satisfying  $\text{\rm div}(u)=\psi$ in $\Omega$ and $u=0$ on $\partial\Omega$. Moreover, for any $0<t\le R_0$,
\begin{equation}\label{est.divu.Du}
\|\nabla u\|_{L^\infty(\Omega)}\le C\ln[t^{-1}R_0+2]\|\psi\|_{L^\infty(\Omega)}+Ct^\eta[\psi]_{C^{0,\eta}(\Omega)},
\end{equation}
where $C$ depends only on $\eta$, $d$ and $\Omega$.\footnote{Even in this underdetermined case, estimate (\ref{est.divu.Du}) is optimal in the sense that there exists bounded $\psi$ such that $\norm{\nabla u}_{L^\infty(\Omega)} \le C\norm{\psi}_{L^\infty(\Omega)}$ is not true for any solution of $\text{div}(u) = \psi$; see \cite{BourgainBrezis03}.}
\end{theorem}

Since $\int_\Omega \psi=0$, there exists a unique weak solution $(u,p)$ of the following Stokes system
\begin{equation}\label{def.Laplace-Stokes.homogeneous}
	\left\{
	\begin{aligned}
		\Delta u+\nabla p &=0 &\quad\text{ in }\Omega,\\
		\text{\rm div}(u) &=\psi &\quad\text{ in }\Omega,\\
		u&=0 &\quad\text{ on }\partial\Omega.\\
	\end{aligned}
	\right.
\end{equation}
Thus $u$ satisfies the divergence equation in Theorem \ref{thm.divergence.modifiedC1}. It is now sufficient to show (\ref{est.divu.Du}). To this end, we need the following lemma.

\begin{lemma}\label{lem.Pi-difference}
	Let $\Omega$ be a $C^{1,\eta}$ domain and $(G^\Delta, \Pi^\Delta)$ be the Green's function of system (\ref{def.Laplace-Stokes.homogeneous}). Then if $|x-y|>2|x-z|$,
	\begin{equation}\label{est.2TermDiff}
	\big|\Pi^\Delta(y,x) - \Pi^\Delta(y,z)\big| \le \frac{C|x-z|}{|\Omega|}\ln \bigg(\frac{R_0}{|x-z|}\bigg) + \frac{C|x-z|}{|x-y|^d}.
	\end{equation}
\end{lemma}
\begin{proof}
	Fix $x,z\in \Omega$ and set $r = |x-z|$. We first show that if $y,w \in \Omega\setminus B_{2r}(x)$, then
	\begin{equation}\label{est.4TermDiff}
	\begin{aligned}
	\Big|[\Pi^\Delta(y,x) - \Pi^\Delta(y,z)] - [\Pi^\Delta(w,x) - \Pi^\Delta(w,z)]\Big| \le C \bigg\{ \frac{|x-z|}{|x-y|^d}  + \frac{|x-z|}{|x-w|^d} \bigg\}.
	\end{aligned}
	\end{equation}
	Observe that to see (\ref{est.4TermDiff}), it suffices to show
	\begin{equation}\label{est.oscPiDelta}
	\osc{\Omega\setminus B_t(x)}{\Pi^\Delta(\cdot,x) - \Pi^\Delta(\cdot,z)} \le \frac{C|x-z|}{t^d}, \qquad \text{for any } t\ge 2r.
	\end{equation}
	Here we use a familiar argument as in Theorem \ref{thm.green-function.pointwise}. Note that $\Pi^\Delta(\cdot,x) - \Pi^\Delta(\cdot,z)$ and $G^\Delta(\cdot,x) - G^\Delta(\cdot,z)$ satisfy
	\begin{equation}\label{eq.LaplaceStokes.Green}
	\left\{
	\begin{aligned}
	-\Delta \big(G^\Delta (\cdot,x) - G^\Delta (\cdot,z) \big)+\nabla \big( \Pi^\Delta(\cdot,x) - \Pi^\Delta(\cdot,z) \big) &= 0 \qquad &\text{ in }& \Omega\setminus\{x,z\},\\
	\text{div} \big(G^\Delta (\cdot,x) - G^\Delta (\cdot,z) \big) &= 0 \qquad & \text{ in }& \Omega,\\
	\big(G^\Delta (\cdot,x) - G^\Delta (\cdot,z) \big) & = 0 \qquad & \text{ on }&\partial\Omega.
	\end{aligned}
	\right.	
	\end{equation}
	Let $B$ be a ball such that $\text{dist}(B_t(x),2B) \simeq 2^k t \simeq \text{diam}(B)$ for some $k\ge 0$. Then by applying Theorem \ref{thm.boundary-estimate}, we have
	\begin{equation*}
	\begin{aligned}
	\osc{B\cap \Omega}{\Pi^\Delta(\cdot,x) - \Pi^\Delta(\cdot,z)} 
	&\le \frac{C}{2^k t} \bigg( \average_{2B \cap \Omega} |G^\Delta(v,x) - G^\Delta(v,z)|^2 dw\bigg)^{1/2} \\
	& \le \frac{C|x-z|}{2^k t} \sup_{v\in 2B,\; \zeta\in B_{t}(x)} |\nabla_\zeta G^\Delta(v,\zeta)| \\
	& \le \frac{C|x-z|}{(2^k t)^d},
	\end{aligned}
	\end{equation*}
	where we have used (\ref{est.dyG}) in the last inequality. Following by the same covering argument in Theorem \ref{thm.green-function.pointwise}, we obtain the desired estimate (\ref{est.oscPiDelta}).
	
	Finally, we claim that (\ref{est.4TermDiff}) implies (\ref{est.2TermDiff}). Actually, by our assumption that $\int_{\Omega} \Pi^\Delta(w,x)\, dw = 0 $ for any $x\in \Omega$, we have
	\begin{equation*}
	\begin{aligned}
	\big|\Pi^\Delta(y,x) - \Pi^\Delta(y,z)\big| & \le \average_{\Omega} \Big|[\Pi^\Delta(y,x) - \Pi^\Delta(y,z)] - [\Pi^\Delta(w,x) - \Pi^\Delta(w,z)]\Big| dw \\
	& \le \frac{C}{|\Omega|} \int_{\Omega \setminus B_{2r}(x)} \bigg\{ \frac{|x-z|}{|x-y|^d}  + \frac{|x-z|}{|x-w|^d} \bigg\}dw \\
	& \qquad \qquad + \frac{C}{|\Omega|} \int_{\Omega \cap B_{2r}(x)} \bigg\{ \frac{1}{|x-w|^{d-1}} + \frac{1}{|z-w|^{d-1}} \bigg\}dw \\
	& \le \frac{C|x-z|}{|\Omega|}\ln \bigg(\frac{R_0}{|x-z|}\bigg) + \frac{C|x-z|}{|x-y|^d}.
	\end{aligned}
	\end{equation*}
	where we have used (\ref{est.4TermDiff}) and (\ref{est.Piy}) in the second inequality.
\end{proof}

\begin{proof}[Proof of Theorem \ref{thm.divergence.modifiedC1}]
	
	First of all, by the Schauder's estimate for (\ref{def.Laplace-Stokes.homogeneous}), we have
	\begin{equation}\label{est.divEq.Holder}
	[\nabla u]_{C^{0,\eta}(\Omega)} \le C[\psi]_{C^{0,\eta}(\Omega)}.
	\end{equation}
For any fixed $0<t\le R_0$, we discuss the following two cases.

Case 1: Assume $ x\in \{\delta(x)\ge 4t\}$. In view of
\begin{equation}\label{ineq.laplace.dLinfty.split}
\begin{aligned}
|\nabla u(x)|&\le \bigg|\nabla u(x)-\average_{B_t(x)}\nabla u\bigg|+\bigg(\average_{B_t(x)}|\nabla u|^2\bigg)^{1/2}\\
& \le Ct^\eta[\nabla u]_{C^{0,\eta}(B_t(x))}+\bigg(\average_{B_t(x)}|\nabla u|^2\bigg)^{1/2}\\
&\le Ct^\eta [\psi]_{C^{0,\eta}(\Omega)}+\bigg(\average_{B_t(x)}|\nabla u|^2\bigg)^{1/2},
\end{aligned}
\end{equation}
it remains to estimate the second term on right-hand side of the above inequality.

Because $u-q$ is also a solution of (\ref{def.Laplace-Stokes.homogeneous}) in $B_{4t}(x)$ for any $q\in \R^d$, by the Cacciopoli's inequality (\ref{ineq.Cacciopoli}), we obtain
\begin{equation}\label{ineq.laplace.d2u.L2average}
\bigg(\average_{B_t(x)}|\nabla u(z)|^2dz\bigg)^{1/2} \le \frac{C}{t}\bigg(\average_{B_{2t}(x)}|u(z)-q|^2 dz\bigg)^{1/2}+C\bigg(\average_{B_{2t}(x)}|\psi|^2\bigg)^{1/2}.
\end{equation}
Following by the integral representation (\ref{def.intRep}), we see that
\begin{equation*}
u(x)=-\int_\Omega \Pi^\Delta(y,x)\psi(y)dy.
\end{equation*}
If we choose $q=u(x)$ in (\ref{ineq.laplace.d2u.L2average}), then for any $z\in B_{2t}(x)$,
\begin{equation}\label{ineq.laplace-green.difference}
\begin{aligned}
|u(z)-u(x)| & \le \|\psi\|_{L^\infty(\Omega)}\int_\Omega |\Pi^\Delta(y,z)-\Pi^\Delta(y,x)|\, dy\\
& \le \|\psi\|_{L^\infty(\Omega)}\int_{B_{4t}(x)} |\Pi^\Delta(y,z)-\Pi^\Delta(y,x)|dy \\
& \qquad\qquad + \|\psi\|_{L^\infty(\Omega)}\int_{\Omega\backslash B_{4t}(x)} |\Pi^\Delta(y,z)-\Pi^\Delta(y,x)|dy \\
&  \le C\|\psi\|_{L^\infty(\Omega)}\int_{B_{4t}(x)} \frac{1}{|z-y|^{d-1}}+\frac{1}{|x-y|^{d-1}}\, dy \\
& \qquad + C\|\psi\|_{L^\infty(\Omega)} \int_{\Omega\backslash B_{4t}(x)}\left[\frac{|x-z|}{|\Omega|}\ln \bigg(\frac{R_0}{|x-z|}\bigg) + \frac{|x-z|}{|x-y|^d}\right] dy \\
&\le  Ct \|\psi\|_{L^\infty(\Omega)}  + Ct\ln[t^{-1} R_0] \|\psi\|_{L^\infty(\Omega)} ,
\end{aligned}
\end{equation}
where we have used (\ref{est.Pixy}) and Lemma \ref{lem.Pi-difference} in the third inequality. Combining this with (\ref{ineq.laplace.dLinfty.split}) and (\ref{ineq.laplace.d2u.L2average}), we prove that for any $x\in\{\delta(x)\ge 4t\}$,
\begin{equation}\label{ineq.Laplace.modifiedC11.interior}
|\nabla u(x)|\le C\ln[t^{-1}R_0+2]\|\psi\|_{L^\infty(\Omega)}+Ct^\eta[\psi]_{C^{0,\eta}(\Omega)}.
\end{equation}

Case 2: For any $x\in\{\delta(x)<4t\}$, there must exist some $z\in D_{8t}(x)\cap \{ \delta(x)\ge 4t \}$. Therefore by a triangle inequality 
\begin{equation*}
\begin{aligned}
|\nabla u(x)|&\le |\nabla u(x)-\nabla u(z)|+|\nabla u(z)|\\
&\le Ct^\eta[\nabla u]_{C^{0,\eta}(D_{8t}(x))}+|\nabla u(z)|\\
&\le C\ln[t^{-1}R_0+2]\|\psi\|_{L^\infty(\Omega)}+Ct^\eta[\psi]_{C^{0,\eta}(\Omega)},
\end{aligned}
\end{equation*}
where we have used (\ref{est.divEq.Holder}) and (\ref{ineq.Laplace.modifiedC11.interior}) in the last inequality. The proof is now finished.
\end{proof}

Now we proceed to prove Theorem \ref{thm.modified-boundary-estimate}.

\begin{proof}[Proof of Theorem \ref{thm.modified-boundary-estimate}]
Part (i): Adjustable Lipschitz estimate for $u_\varepsilon$.

We choose a $C^{1,\eta}$ domain $\widetilde{D}$ such that $D_{4r}\subset \widetilde{D}\subset D_{5r}$, and decompose $(u_\varepsilon,p_\varepsilon)=(u_\varepsilon^{(1)}+u_\varepsilon^{(2)},p_\varepsilon^{(1)}+p_\varepsilon^{(2)})$ such that $(u_\varepsilon^{(1)},p_\varepsilon^{(1)})$, $(u_\varepsilon^{(2)},p_\varepsilon^{(2)})$ satisfy the following systems:
\begin{equation*}
\left\{
\begin{aligned}
\cL_\e(u_\e^{(1)})+\nabla p_\e^{(1)} &= 0 & \text{ in }\widetilde{D},\\
\rdiv(u_\e^{(1)}) &=0 &\text{ in }\widetilde{D},\\
u_\e^{(1)}& =u_\varepsilon &\text{ on }\partial\widetilde{D},\\
\end{aligned}
\right.	
\quad\text{and}\quad
\left\{
\begin{aligned}
\cL_\e(u_\e^{(2)})+\nabla p_\e^{(2)} &= F+ \text{div}(h) & \text{ in }\widetilde{D},\\
\rdiv(u_\e^{(2)}) &=0 &\text{ in }\widetilde{D},\\
u_\e^{(2)}& =0 &\text{ on }\partial\widetilde{D}.\\
\end{aligned}
\right.	
\end{equation*}
It follows from Theorem \ref{thm.boundary-estimate} and a convexity argument that
\begin{equation}\label{ineq.ue1Lip}
\begin{aligned}
\norm{\nabla u_\e^{(1)}}_{L^\infty(D_r)} + \osc{D_r}{p_\varepsilon^{(1)}}\le \frac{C}{r} \average_{D_{4r} } |u_\e^{(1)}| \le C \bigg\{ \frac{1}{r} \average_{D_{4r} } |u_\e| +\|\nabla u_\varepsilon^{(2)}\|_{L^\infty(D_{4r})} \bigg\},
\end{aligned}
\end{equation}
where we have used $\|u_\varepsilon^{(2)}\|_{L^\infty(D_{4r})}\le Cr\|\nabla u_\varepsilon^{(2)}\|_{L^\infty(D_{4r})}$ in the case that $D_{4r}$ meet the boundary $\partial\Omega$. For the interior case, namely, $D_{4r} = B_{4r}$ does not meet $\partial\Omega$, the first inequality of (\ref{ineq.ue1Lip}) holds true for $u_\varepsilon^{(1)}-q$, where $q\in \R^d$ is any constant. In particular, if we choose $q=\average_{D_{4r}} u_\varepsilon^{(2)}$, then the second inequality follows by $\|u_\varepsilon^{(2)}-q\|_{L^\infty(D_{4r})}\le Cr\|\nabla u_\varepsilon^{(2)}\|_{L^\infty(D_{4r})}$.

Let $(\widetilde{G}_\varepsilon,\widetilde{\Pi}_\varepsilon)$ denote the Green's functions of the Stokes system in $\widetilde{D}$. To estimate $\nabla u_\varepsilon^{(2)}$ on $D_{4r}$, in view of the integral representation (\ref{def.intRep})
\begin{align}
\begin{aligned}\label{eq.ue2.Fdivh}
u_\e^{(2)}(x) &=  \int_{\widetilde{D}} \widetilde{G}_\varepsilon(x,y)F(y)\,dy- \int_{\widetilde{D}} \nabla_y \widetilde{G}_\e(x,y) h(y)\,dy \\
&=\int_{\widetilde{D}} \widetilde{G}_\varepsilon(x,y)F(y)\,dy-\int_{\widetilde{D}} \nabla_y \widetilde{G}_\e(x,y) \big[h(y)-h(x)\big]dy,\\
\end{aligned}
\end{align}
it remains to estimate
\begin{equation}\label{ineq.ue2Lip}
\begin{aligned}
|\nabla u_\varepsilon^{(2)}(x)| &\le\int_{\widetilde{D}} |\nabla_x \widetilde{G}_\varepsilon(x,y)||F(y)|\,dy + \int_{\widetilde{D}} |\nabla_x \nabla_y \widetilde{G}_\varepsilon(x,y)||h(y)-h(x)|\,dy\\
&= I_1+I_2.
\end{aligned}
\end{equation}

By Theorem \ref{thm.green-function.pointwise}, it is easy to see
$$
\begin{aligned}
|\nabla_x \widetilde{G}_\varepsilon(x,y)|&\le C\min\bigg\{\frac{\text{dist}(y,\partial \widetilde{D})}{|x-y|^d}, \frac{1}{|x-y|^{d-1}}\bigg\}\\
&\le C\min\left\{\frac{\delta(y)}{|x-y|^d}, \frac{1}{|x-y|^{d-1}}\right\},
\end{aligned}
$$
It follows that for any $0<t\le r$, there exists $N$ with $r \approx 2^N t$ so that
\begin{equation}\label{ineq.ue2.F}
\begin{aligned}
I_1 &\le C\int_{\widetilde{D}\backslash B(x,t)} \frac{|F(y)|\delta(y)}{|x-y|^d}\,dy+C\|F\|_{L^p(D_{5r})}\bigg(\int_{\widetilde{D}\cap B(x,t)} \frac{dy}{|x-y|^{(d-1)p'}}\bigg)^{1/p'}\\
&\le C\sum_{i=0}^N\int_{\widetilde{\Omega} \cap B(x,2^{i+1}t)\backslash B(x,2^it)}\frac{|F(y)|\delta(y)}{|x-y|^d}\,dy+C\|F\|_{L^p(D_{5r})}\bigg(\int_{\widetilde{D}\cap B(x,t)} \frac{dy}{|x-y|^{(d-1)p'}}\bigg)^{1/p'}\\
&\le C\ln[t^{-1}r+2]|\mathcal{M}_{D_{5r},t}(F\delta)(x)|+Ct^{1-d/p}\|F\|_{L^p(D_{5r})}.
\end{aligned}
\end{equation}
Also, recall from Theorem \ref{thm.green-function.pointwise} that
$$
|\nabla_x \nabla_y \widetilde{G}_\varepsilon(x,y)|\le \frac{C}{|x-y|^d},
$$
which yields that for any $0<t \le r$,
\begin{equation}\label{ineq.ue2.h}
\begin{aligned}
I_2&\le C\|h\|_{L^\infty(D_{5r})}\int_{\widetilde{D}\backslash B(x,t)} \frac{dy}{|x-y|^d}+C[h]_{C^{0,\eta}(D_{5r})}\int_{\widetilde{D}\cap B(x,t)} \frac{dy}{|x-y|^{d-\eta}}\\
&\le C\ln\big[t^{-1}r+2\big]\|h\|_{L^\infty(D_{5r})}+Ct^{\eta}[h]_{C^{0,\eta}(D_{5r})}.
\end{aligned}
\end{equation}

Substituting (\ref{ineq.ue2.F}) and (\ref{ineq.ue2.h}) into (\ref{ineq.ue2Lip}), and combining with (\ref{ineq.ue1Lip}), we obtain the adjustable Lipschitz estimate for $u_\varepsilon$, i.e., for any $0<t \le r$,
\begin{equation*}\label{ineq.modified-Lipschitz}
\begin{aligned}
\norm{\nabla u_\e}_{L^\infty(D_r)} & \le C \bigg\{ \frac{1}{r} \average_{D_{5r} } |u_\e| +\ln[t^{-1}r+2]\|\mathcal{M}_{D_{5r},t}(F\delta)\|_{L^\infty(D_{5r})} \\  
	&\qquad +t\bigg(\average_{D_{5r}}|F|^p\bigg)^{1/p}+\ln[t^{-1}r+2]\|h\|_{L^\infty(D_{5r})}+t^{\eta}[h]_{C^{0,\eta}(D_{5r})}\bigg\}.
\end{aligned}
\end{equation*}

Part (ii): Adjustable oscillation estimate for $p_\varepsilon$.

In view of (\ref{ineq.ue1Lip}), it now remains to estimate the oscillation of $p_\varepsilon^{(2)}$. By (\ref{def.intRep.onlyF}), $p_\varepsilon^{(2)}$ can be represented by
\begin{equation*}
\begin{aligned}
p_\varepsilon^{(2)}(x)&=\int_{\widetilde{D}} \widetilde{\Pi}_\varepsilon(x,y)\big[F(y)+ \text{div}(h)(y)\big]\,dy\\
&=\int_{\widetilde{D}} \widetilde{\Pi}_\varepsilon(x,y)F(y)\,dy-\int_{\widetilde{D}}\nabla_y \widetilde{\Pi}_\varepsilon(x,y)h(y)\, dy.
\end{aligned}
\end{equation*}
The boundary integral term vanishes since $\widetilde{\Pi}_\varepsilon(x,y) = 0$ for $y\in \partial\widetilde{D}$; see Remark \ref{rmk.Pi.Bdry}. We now consider the pressure difference
\begin{equation}\label{ineq.pe2diff}
\begin{aligned}
p_\varepsilon^{(2)}(x)-p_\varepsilon^{(2)}(z)&= \int_{\widetilde{D}} [ \widetilde{\Pi}_\varepsilon(x,y)-\widetilde{\Pi}_\varepsilon(z,y)] F(y)dy \\
&\qquad + \int_{\widetilde{D}} [\nabla_y \widetilde{\Pi}_\varepsilon(x,y)-\nabla_y \widetilde{\Pi}_\varepsilon(z,y)] h(y)dy\\
& = J_1 + J_2.
\end{aligned}
\end{equation}
We first estimate $J_1$. By (\ref{est.Piy}), we have,
\begin{equation*}
\begin{aligned}
|\widetilde{\Pi}_\e(x,y) - \widetilde{\Pi}_\e(z,y)| &\le C\min \bigg\{ \frac{\text{dist}(y,\partial\widetilde{D})}{|x-y|^{d}} + \frac{\text{dist}(y,\partial\widetilde{D})}{|z-y|^d}, \frac{1}{|x-y|^{d-1}} + \frac{1}{|z-y|^{d-1}} \bigg\}\\
&\le C\min \bigg\{ \frac{\delta(y)}{|x-y|^{d}} + \frac{\delta(y)}{|z-y|^d}, \frac{1}{|x-y|^{d-1}} + \frac{1}{|z-y|^{d-1}} \bigg\}.
\end{aligned}
\end{equation*}
Then, it is natural to consider the integral over $\widetilde{D}\cap\{|x-y|\le |z-y|\}$ and $\widetilde{D}\cap\{|x-y|> |z-y|\}$ separately. By the similar argument as in (\ref{ineq.ue2.F}), we have
\begin{equation*}
\begin{aligned}
&\int_{\widetilde{D}\cap\{|x-y|\le |z-y|\}}|\widetilde{\Pi}_\varepsilon(x,y)-\widetilde{\Pi}_\varepsilon(z,y)||F(y)|dy\\
&\quad\le C\int_{\{\widetilde{D}\cap\{|x-y|\le |z-y|\}\}\backslash B_t(x)} \frac{|F(y)|\delta(y)}{|x-y|^d}\,dy\\
&\qquad\qquad\qquad+C\|F\|_{L^p(D_{4r})}\bigg(\int_{\{\widetilde{D}\cap\{|x-y|\le |z-y|\}\}\cap B_t(x)} \frac{dy}{|x-y|^{(d-1)p'}}\bigg)^{1/p'}\\
&\quad\le C\ln[t^{-1}r+2]|\mathcal{M}_{D_{5r},t}(F\delta)(x)|+Ct^{1-d/p}\|F\|_{L^p(D_{5r})},
\end{aligned}
\end{equation*}
and obviously the same estimate holds for the integral over $\widetilde{D}\cap\{|x-y|> |z-y|\}$. Hence, we arrive at
\begin{equation}\label{ineq.pe2.F}
\begin{aligned}
|J_1| \le C\ln[t^{-1}r+2]|\mathcal{M}_{D_{5r},t}(F\delta)(x)|+Ct\bigg(\average_{D_{5r}}|F|^p\bigg)^{1/p}.
\end{aligned}
\end{equation}

To estimate $J_2$, we first consider the case that $x$ and $z$ are far enough from each other, i.e., there exists some constant $c_1>0$ depending only on $d$ and $\Omega$, such that $|x-z|\ge c_1 r$. Then we construct the following auxiliary function $\zeta^{x,z}_h(y)$ for the given $x,z$ and $h(y)$,
\begin{equation*}
\zeta^{x,z}_h(y)=h(x)\frac{|z-y|^\eta}{|x-z|^\eta}+h(z)\frac{|x-y|^\eta}{|x-z|^\eta}.
\end{equation*}
It is easy to observe that $\zeta^{x,z}_h(x)=h(x)$, $\zeta^{x,z}_h(z)=h(z)$ and
\begin{align}
&\|\zeta_h^{x,z}\|_{L^\infty(D_{5r})}\le C\|h\|_{L^\infty(D_{5r})}, \label{prop.zeta.Linfty}\\
&\|\nabla \zeta^{x,z}_h(y)\|_{L^\infty(D_{5r})}\le Cr^{-\eta}\|h\|_{L^\infty(D_{5r})}[|z-y|^{\eta-1}+|x-y|^{\eta-1}], \label{prop.zeta.differential}\\
&|\zeta^{x,z}_h(y)-h(y)|\le C\min\{|x-y|^\eta,|z-y|^\eta\}\left\{[h]_{C^{0,\eta}(D_{5r})}+r^{-\eta}\|h\|_{L^\infty(D_{5r})}\right\}. \label{prop.zeta-h}
\end{align}
Inserting $\zeta_h^{x,z}(y)$ into $J_2$ and by the integration by parts, we have
\begin{equation}\label{ineq.pe2.h.zeta}
\begin{aligned}
J_2
&=\int_{\widetilde{D}}[\nabla_y \widetilde{\Pi}_\varepsilon(x,y)-\nabla_y \widetilde{\Pi}_\varepsilon(z,y)][h(y)-\zeta_h^{x,z}(y)]\,dy\\
&\qquad -\int_{\widetilde{D}}[\widetilde{\Pi}_\varepsilon(x,y)-\widetilde{\Pi}_\varepsilon(z,y)]\text{div}(\zeta^{x,z}_h(y))\,dy. 
\end{aligned}
\end{equation}
Write
\begin{equation*}
\begin{aligned}
|J_2| &\le  \int_{\widetilde{D}}|\nabla_y \widetilde{\Pi}_\varepsilon(x,y)-\nabla_y \widetilde{\Pi}_\varepsilon(z,y)||h(y)-\zeta_h^{x,z}(y)|\,dy\\
&\quad+\int_{\widetilde{D}}|\widetilde{\Pi}_\varepsilon(x,y)-\widetilde{\Pi}_\varepsilon(z,y)||\nabla\zeta^{x,z}_h(y)|\,dy\\
&=K_1+K_2.
\end{aligned}
\end{equation*}

To estimate $K_1$, by using (\ref{prop.zeta-h}), (\ref{prop.zeta.Linfty}) and (\ref{est.dxyPi}), we obtain
\begin{equation*}
\begin{aligned}
&\int_{\widetilde{D}\cap\{|x-y|\le |z-y|\}} |\nabla_y \widetilde{\Pi}_\varepsilon(x,y)-\nabla_y \widetilde{\Pi}_\varepsilon(z,y)| |h(y)-\zeta_h^{x,z}(y)|\,dy\\
&\quad\le C\big\{\|\zeta^{x,z}_h\|_{L^\infty(D_{5r})}+\|h\|_{L^\infty(D_{5r})}\big\} \int_{\{\widetilde{D}\cap\{|x-y|\le |z-y|\}\}\cap B_t(x)}\frac{1}{|x-y|^d}\,dy\\
&\qquad+C\big\{[h]_{C^{0,\eta}(D_{5r})}+r^{-\eta}\|h\|_{L^\infty(D_{5r})}\big\}\int_{\{\widetilde{D}\cap\{|x-y|\le |z-y|\}\}\backslash B_t(x)}\frac{1}{|x-y|^{d-\eta}}\,dy\\
&\quad \le C\ln[t^{-1}r+2]\|h\|_{L^\infty(D_{5r})}+Ct^{\eta}[h]_{C^{0,\eta}(D_{5r})}.
\end{aligned}
\end{equation*}
The same argument also gives the estimate for the integral over $\widetilde{D}\cap\{|x-y|>|z-y|\}$. These imply that
\begin{equation*}
\begin{aligned}
K_1\le C\ln[t^{-1}r+2]\|h\|_{L^\infty(D_{5r})}+Ct^{\eta}[h]_{C^{0,\eta}(D_{5r})}.
\end{aligned}
\end{equation*}
On the other hand, by (\ref{prop.zeta.differential}) and (\ref{est.Piy}), we have
\begin{equation*}
\begin{aligned}
K_2&\le C r^{-\eta}\|h\|_{L^\infty(D_{5r})}\int_{\widetilde{D}}\bigg(\frac{1}{|z-y|^{d-1}}+\frac{1}{|x-y|^{d-1}}\bigg)[|z-y|^{\eta-1}+|x-y|^{\eta-1}]\,dy\\
&\le C \|h\|_{L^\infty(D_{5r})}.
\end{aligned}
\end{equation*}
It follows that
\begin{equation}\label{ineq.pe2.h}
|J_2| \le C\ln[t^{-1}r+2]\|h\|_{L^\infty(D_{5r})}+Ct^{\eta}[h]_{C^{0,\eta}(D_{5r})}.
\end{equation}

Combining (\ref{ineq.pe2diff}), (\ref{ineq.pe2.F}) and (\ref{ineq.pe2.h}), we have proved that if $|x-z|\ge c_1 r$, 
\begin{equation*}
\begin{aligned}
|p_\varepsilon^{(2)}(x)-p_\varepsilon^{(2)}(z)|&\le C\big\{\ln[t^{-1}r+2]\|\mathcal{M}_{D_{5r},t}(F\delta)\|_{L^\infty(D_{5r})}+t^{1-d/p}\|F\|_{L^p(D_{5r})}\\
&\qquad+\ln[t^{-1}r+2]\|h\|_{L^\infty(D_{5r})}+t^{\eta}[h]_{C^{0,\eta}(D_{5r})}\big\}.
\end{aligned}
\end{equation*}
For the case $|x-z|<c_1 r$, one can always find a $z_1\in \widetilde{D}$, such that $|x-z_1|\ge c_2 r$ and $|z_1-z|\ge c_2 r$. Then, the above inequality still holds true by a triangle inequality $|p_\varepsilon^{(2)}(x)-p_\varepsilon^{(2)}(z)|\le |p_\varepsilon^{(2)}(x)-p_\varepsilon^{(2)}(z_1)|+|p_\varepsilon^{(2)}(z_1)-p_\varepsilon^{(2)}(z)|$. As a consequence,
\begin{equation}\label{ineq.pe2.osc.complete}
\begin{aligned}
\osc{\widetilde{D}}{p_\varepsilon^{(2)}} &\le C\Big\{\ln[t^{-1}r+2]\|\mathcal{M}_{D_{5r},t}(F\delta)\|_{L^\infty(D_{5r})}+t^{1-d/p}\|F\|_{L^p(D_{5r})} \\  & \qquad+\ln[t^{-1}r+2]\|h\|_{L^\infty(D_{5r})}+t^{\eta}[h]_{C^{0,\eta}(D_{5r})}\Big\}.
\end{aligned}
\end{equation}

Finally, combining (\ref{ineq.ue1Lip}) with (\ref{ineq.pe2.osc.complete}), we have proved the adjustable oscillation estimate for $p_\varepsilon$,
\begin{equation*}
\begin{aligned}
\osc{D_{r}}{p_\varepsilon} & \le C \bigg\{ \frac{1}{r} \average_{D_{5r} } |u_\e| +\ln[t^{-1}r+2]\|\mathcal{M}_{D_{5r},t}(F\delta)\|_{L^\infty(D_{5r})} \\  
	&\quad +t\bigg(\average_{D_{5r}}|F|^p\bigg)^{1/p}+\ln[t^{-1}r+2]\|h\|_{L^\infty(D_{5r})}+t^{\eta}[h]_{C^{0,\eta}(D_{5r})}\bigg\}.
\end{aligned}
\end{equation*}
The proof is now complete.
\end{proof}

\begin{remark}\label{rmk.recover}
	We claim that Theorem \ref{thm.modified-boundary-estimate} recovers Theorem \ref{thm.boundary-estimate} if we set $t = r$. In fact, we only need to estimate $\|\mathcal{M}_{D_{5r},r}(F\delta)\|_{L^\infty(D_{5r})}$. Note that $\delta(x) = \text{dist}(x,\partial D_{5r}) \le Cr$. Then for any $x\in D_{5r}$,
	\begin{equation*}
	\mathcal{M}_{D_{5r},r}(F\delta)(x) \le C r \sup_{s>r}\average_{B(x,r)\cap D_{5r}} |F| \le C r \bigg(\average_{D_{5r}}|F|^p\bigg)^{1/p}.
	\end{equation*}
	The claim then follows readily. Therefore, Theorem \ref{thm.modified-boundary-estimate} can be viewed as an improved version of Theorem \ref{thm.boundary-estimate} with an adjustable parameter $t$.
\end{remark}

As we have mentioned in the Introduction, similar adjustable estimate may be obtained for non-trivial boundary data $f$. To demonstrate this, we will show, of independent interest, an analog for elliptic system with non-trivial boundary data.

\begin{theorem}\label{thm.elliptic.adj}
	Assume $\Omega$ is a bounded $C^{1,\eta}$ domain and $R_0 = \rtxt{diam}(\Omega)$. Let $h\in C^{0,\eta}(\Omega;\R^{d\times d})$, $ f\in C^{1,\eta}(\partial\Omega;\R^d)$ and $u_\e$ be the weak solution of
	\begin{equation}\label{eq.elliptic}
		\left\{
		\begin{aligned}
			\mathcal{L}_\varepsilon(u_\varepsilon) &= \rtxt{div}(h) &\qquad \rtxt{ in }\Omega,\\
			u_\varepsilon& =f &\qquad\rtxt{ on }\partial\Omega.
		\end{aligned}
		\right.	
	\end{equation}
	Then, for any $0<t\le R_0$,
	\begin{equation*}
		\norm{\nabla u_\e}_{L^\infty(\Omega)} \le C\ln[t^{-1}R_0+2] \big(\norm{h}_{L^\infty(\Omega)} + \norm{\nabla f}_{L^\infty(\Omega)} \big) + Ct^\eta \big([h]_{C^{0,\eta}(\Omega)} + [\nabla f]_{C^{0,\eta}(\Omega)} \big),
	\end{equation*}
	where $C$ depends only on $d,\eta, A$ and $\Omega$.
\end{theorem}
\begin{proof}
	We temporarily let $G_\e^\dagger (x,y)$ and $P_\e^\dagger(x,y)$ be the Green's function and Poisson kernel of $\cL_\e$ in $\Omega$, respectively. By the representation formula, we have
	\begin{equation*}
		\nabla u_\e(x) = \int_{\Omega} \nabla_x G_\e^\dagger(x,y) \text{div}(h)(y) dy + \int_{\partial\Omega}\nabla_x P_\e^\dagger(x,y) f(y)d\sigma(y)
	\end{equation*}
	The first term above can be handled analogously as the second term of (\ref{eq.ue2.Fdivh}). To deal with the second term, we extend $f$ from $\partial\Omega$ to $\R^d$ with both $\norm{\nabla f}_{L^\infty}$ and $[\nabla f]_{C^{0,\eta}}$ being preserved, and denote the extended function still by $f$. Then, using the $C^{1,\eta}$ continuity, we have for any $x\in \Omega,y\in \partial\Omega$,
	\begin{equation*}
		|f(y) - f(x) - (y-x)\cdot \nabla f(x)| \le \min \big\{ |x-y|^{1+\eta} [\nabla f]_{C^{0,\eta}}, 2|x-y|\norm{f}_{L^\infty} \big\}.
	\end{equation*}
	Therefore, for any fixed $t>0$
	\begin{align}
		\begin{aligned}\label{ineq.Pe.plus}
			& \bigg| \int_{\partial\Omega}\nabla_x P_\e^\dagger (x,y) [f(y) - f(x) - (y-x)\cdot \nabla f(x)] d\sigma(y) \bigg| \\
			& \qquad \le C \int_{\partial\Omega \cap B(x,t)} |P_\e^\dagger(x,y)||x-y|^{1+\eta} [\nabla f]_{C^{0,\eta}} d\sigma(y) \\
			& \qquad \qquad + C \int_{\partial\Omega \setminus B(x,t)} |P_\e^\dagger(x,y)| |x-y|\norm{\nabla f}_{L^\infty} d\sigma(y) \\
			& \qquad \le C t^\eta [\nabla f]_{C^{0,\eta}} + C\ln[R_0/t + 2] \norm{f}_{L^\infty}.
		\end{aligned}
	\end{align}
	Now observe that $\int_{\partial\Omega} \nabla_x P_\e^\dagger(x,y) d\sigma(y) = 0$ and 
	\begin{equation*}
		\int_{\partial\Omega} \nabla_x P_\e^\dagger(x,y) y\cdot \nabla f(x) d\sigma(y) = \nabla \Phi_\e^+(x) \cdot \nabla f(x),
	\end{equation*}
	where $\Phi_\e^\dagger(x)$ is the Dirichlet corrector for elliptic operator $\cL_\e$ (i.e., $\cL_\e \Phi_\e^\dagger = 0$) subject to $\Phi_\e^\dagger(x) = x$ on $\partial\Omega$. Thus, (\ref{ineq.Pe.plus}) implies that
	\begin{equation}\label{ineq.ue-Phif}
		|\nabla u_\e(x) - \nabla \Phi_\e^\dagger(x)\cdot \nabla f(x)| \le C t^\eta [\nabla f]_{C^{0,\eta}} + C\ln[R_0/t + 2] \norm{\nabla f}_{L^\infty}.
	\end{equation}
	Finally, note that $|\nabla \Phi_\e^\dagger(x)| \le C$ (by Theorem \ref{thm.boundary-estimate}), where $C$ depends only on $A$ and $\Omega$. This, together with (\ref{ineq.ue-Phif}), leads to the desired estimate.
\end{proof}

\subsection{Proof of Theorem \ref{thm.main} (ii)}
We need some estimates for the Dirichlet correctors.
\begin{lemma}\label{lem.Dirichlet-corrector.estimate}
	Let $\Omega$ be a bounded $C^{1,\eta}$ domain and $(\Phi_{\e,j}^\beta, \Lambda_{\e,j}^\beta)$ be the matrix of Dirichlet correctors defined in (\ref{def.Dirichlet-corrector}). Assume $\Lambda_{\e,j}^\beta(x_0) = \pi_j^\beta(x_0/\e)$ for some fixed point $x_0\in \Omega$ with $\delta(x_0)>r_0>0$ ($r_0$ will be figured out in the proof). Then
	\begin{equation}\label{est.PhiLip}
	|\nabla \Phi_{\e,j}^\beta(x)| \le C,\qquad |\Phi_{\e,j}^\beta(x) - P_j^\beta(x)| \le C\e,
	\end{equation}
	and
	\begin{equation}\label{est.DPhi}
	|\nabla \big\{ \Phi_{\e,j}^\beta(x) - P_j^\beta(x) - \e \chi_j^\beta(x/\e) \big\}| + | \Lambda_{\e,j}^\beta(x)- \pi_j^\beta(x/\e) | \le C\min\bigg\{1, \frac{\e}{\delta(x)}\bigg\},
	\end{equation}
	where $\delta(x) = \rtxt{dist}(x,\partial\Omega)$ and $C$ depends only on $d,A$ and $\Omega$.
\end{lemma}
\begin{proof}
	The first part of (\ref{est.PhiLip}) directly follows from the Lipschitz estimate in Theorem \ref{thm.boundary-estimate}. To show the second part of (\ref{est.PhiLip}), we consider $u_\e(x) = \Phi_{\e,j}^\beta(x) - P_j^\beta(x) - \e \chi_j^\beta(x/\e)$ and $p_\e(x)  = \Lambda_{\e,j}^\beta(x) - \pi_j^\beta(x/\e)$ and observe that $(u_\e,p_\e)$ satisfies
	\begin{equation*}
	\left\{
	\begin{aligned}
	\cL_\e(u_\e)+\nabla p_\e &= 0 &\qquad \text{ in }\Omega,\\
	\tdiv(u_\e) &=0 &\qquad\text{ in }\Omega,\\
	u_\e & = -\e\chi_j^\beta(x/\e) &\qquad\text{ on }\partial\Omega.
	\end{aligned}
	\right.	
	\end{equation*}
	By the Miranda-Agmon maximum principle (Remark \ref{rmk.Max}), one obtains $\norm{u_\e}_{L^\infty(\Omega)} \le C\e$, which implies the second part of (\ref{est.PhiLip}). 
	
	To see (\ref{est.DPhi}), we use the interior estimate in Theorem \ref{thm.boundary-estimate},
	\begin{equation}\label{est.Gam-Pi}
	|\nabla u_\e(x)| + \osc{B_{\delta(x)/8}(x)}{p_\e} \le \frac{C}{\delta(x)} \bigg(\average_{B(x,\delta(x)/2)} |u_\e|^2 \bigg)^{1/2} \le \frac{C\e}{\delta(x)}, 
	\end{equation}
	where we have used (\ref{est.PhiLip}) in the last step. The estimate for $u_\e$ in (\ref{est.Gam-Pi}) gives the first part of (\ref{est.DPhi}). Finally, the estimate for $p_\e$ follows by a method of Harnack chain. We describe the argument in detail as follows.
	
	First, since $\Omega$ is a bounded $C^{1,\eta}$ domain, the following conditions are satisfied: (i) There exists $r_0>0$ such that $\{x\in \Omega: \delta(x)>r_0\}$ is connected and for any $\hat{x} \in \partial\Omega$, $\mathcal{C}(\hat{x}) \cap \{x\in \Omega: \delta(x)>r_0\} \neq \emptyset$. (ii) For any $x_1,x_2\in \{x\in \Omega: \delta(x)>r_0\}$, there exists at most $M$ balls $\{B_i: i=1,2,\cdots,M\}$ such that $x_1$ is connected to $x_2$ by these balls. Moreover, for each $i$, $8B_i\subset \Omega$ and $\text{diam}(B_i) > r_1$ for some fixed $r_1>0$. (iii) The above parameters $r_0,r_1$ and $M$ depend only on the geometry of $\Omega$.
	
	Now let $x_0\in \Omega$ be a point such that $\delta(x_0) > r_0$ and $p_\e(x_0) = 0$. For any $x\in \Omega$, we construct a chain of balls connecting $x$ with $x_0$ in two separated cases.
	
	Case 1: $\delta(x) \ge r_0$. Based on condition (ii), there are at most $M$ balls $\{B_i: i=1,2,\cdots,M\}$ with minimum radius $r_1$ connecting $x$ to $x_0$. Therefore,
	\begin{equation*}
	|p_\e(x)| = |p_\e(x) - p_\e(x_0)| \le \sum_{i = 1}^M \osc{B_i}{p_\e} \le \frac{CM\e}{8r_1} \le \bigg(\frac{CMr_0}{8r_1 }\bigg) \frac{\e}{\delta(x)},
	\end{equation*}
	where we have used (\ref{est.Gam-Pi}) in the second inequality.
	
	Case 2: $\delta(x) < r_0$. Let $\hat{x}$ be a point on $\partial\Omega$ such that $x\in \mathcal{C}(\hat{x})$. By the condition (i), there exists some $\widetilde{x} \in \mathcal{C}(\hat{x})$ satisfying $\delta(\widetilde{x}) > r_0$. Then we will construct a family of balls connecting $x$ to $\widetilde{x}$. To do this, we consider the set $E_k = \mathcal{C}(\hat{x}) \cap B_{2^k\delta(x)}(\hat{x})\setminus B_{2^{k-1}\delta(x)}(\hat{x})$. Observe that by the definition of the non-tangential cone, for each $k$ such that $E_k \neq \emptyset$, there exist at most $N$ balls, denoted by $\{B_{ki}: i=1,2,\cdots, N\}$, such that $E_k \subset \bigcup\limits_i B_{ki}$ and $\text{diam}(B_{ki}) > c_0 2^k\delta(x) $ and $8B_{ki} \subset \Omega$ for each $B_{ki}$. Now since $\{x,\widetilde{x}\} \subset \mathcal{C}(\hat{x}) \setminus B_{\delta(x)}(\hat{x}) \subset \bigcup\limits_{i,k} B_{ki}$,
	\begin{equation*}
	|p_\e(x) - p_\e(\widetilde{x})| \le \sum_{k,i} \osc{B_{ki}}{p_\e} \le \sum_{k\ge 1} CN \frac{\e}{c_02^k \delta(x)} \le \bigg(\frac{CN}{c_0}\bigg) \frac{\e}{\delta(x)}.
	\end{equation*}
	Finally, since $\delta(\widetilde{x}) \ge r_0$, we apply the result of Case 1 and conclude
	\begin{equation*}
	\begin{aligned}
	|p_\e(x)| &\le |p_\e(x) - p_\e(\widetilde{x})| + |p_\e(\widetilde{x})|  \\
	&\le \bigg(\frac{CN}{c_0}\bigg) \frac{\e}{\delta(x)} + \bigg(\frac{CM r_0}{8r_1}\bigg) \frac{\e}{\delta(\widetilde{x})} \\
	&\le  \bigg(\frac{CN}{c_0} + \frac{CM r_0}{8r_1}\bigg) \frac{\e}{\delta(x)}.
	\end{aligned}
	\end{equation*}
	This proves the second part of (\ref{est.DPhi}).
\end{proof}

Note that $\Lambda_\e$ is uniquely specified in Lemma \ref{lem.Dirichlet-corrector.estimate} by $\Lambda_{\e,j}^\beta(x_0) = \pi_j^\beta(x_0/\e)$ for some interior point $x_0\in \Omega$ with $\delta(x_0)>r_0>0$. It turns out this condition is necessary for deriving the correct asymptotic expansions of $\nabla_x G_\e$ and $\Pi_\e$. Throughout this section, we will always assume that $\Lambda_\e$ is uniquely given by Lemma \ref{lem.Dirichlet-corrector.estimate}.

\begin{proof}[Proof of Theorem \ref{thm.main}, Part (ii)]
Fix $x_0,\, y_0\in \Omega$ and $r=|x_0-y_0|/8$. For any fixed $1\le \gamma\le d$, we let $(u_\varepsilon(x),p_\varepsilon(x))=(G^\gamma_\varepsilon(x,y_0),\Pi^\gamma_\varepsilon(x,y_0))$ and $(u_0(x),p_0(x))=(G^\gamma_0(x,y_0),\Pi^\gamma_0(x,y_0))$, and define
\begin{equation}\label{def.w.Green}
\begin{aligned}
w_\varepsilon=u_\varepsilon-u_0-\big\{\Phi_{\varepsilon,j}^\beta-P_j^\beta\big\}\cdot \frac{\partial u_0^\beta}{\partial x_j},\quad\tau_\varepsilon=p_\varepsilon-p_0-\Lambda_{\varepsilon,j}^\beta\frac{\partial u_0^\beta}{\partial x_j}-\varepsilon q_{ij}^\beta(x/\varepsilon)\frac{\partial^2 u_0^\beta}{\partial x_i\partial x_j}.
\end{aligned}
\end{equation}
In order to apply Theorem \ref{thm.modified-boundary-estimate}, we construct a $C^{2,\eta}$ domain $\widetilde{D}$ such that $D_{5r}(x_0)\subset \widetilde{D}\subset D_{6r}(x_0)$. By Lemma \ref{lem.w-system}, $(w_\varepsilon,\tau_\varepsilon)$ satisfies the following system
\begin{equation*}
\left\{
\begin{aligned}
\cL_\e(w_\e)+\nabla \tau_\e &= F_\varepsilon+\rdiv{(h_\varepsilon)} &\qquad \text{ in }\widetilde{D},\\
\rdiv(w_\e) &=g_\varepsilon &\qquad\text{ in }\widetilde{D},\\
w_\e& =0 &\qquad\text{ on }\widetilde{\Delta},
\end{aligned}
\right.	
\end{equation*}
where $(F_\varepsilon, h_\varepsilon, g_\varepsilon)$ are given by
\begin{equation}\label{def.Fgh.e}
\left\{
\begin{aligned}
F_\varepsilon^\alpha&=a_{ik}^{\alpha\gamma}(x/\varepsilon)\frac{\partial}{\partial x_k}\Big[\Phi_{\varepsilon,j}^{\gamma\beta}-P_j^{\gamma\beta}-\varepsilon\chi_j^{\gamma\beta}(x/\varepsilon)\Big]\frac{\partial^2 u_0^\beta}{\partial x_i \partial x_j}+\Big[\pi_j^\beta(x/\varepsilon)-\Lambda_{\varepsilon,j}^\beta\Big]\frac{\partial^2 u_0^\beta}{\partial x_\alpha\partial x_j},\\
h_\varepsilon^{\alpha i}&= -\Big[\varepsilon \phi_{kij}^{\alpha\beta}(x/\varepsilon)-a_{ik}^{\alpha\gamma}(x/\varepsilon)\big(\Phi_{\varepsilon,j}^{\gamma\beta}-P_j^{\gamma\beta}\big)\Big]\frac{\partial^2 u_0^\beta}{\partial x_k\partial x_j}-\varepsilon q_{ij}^\beta(x/\varepsilon)\frac{\partial^2 u_0^\beta}{\partial x_\alpha\partial x_j},\\
g_\varepsilon&=-\big\{\Phi_{\varepsilon,j}^{\alpha\beta}-P_j^{\alpha\beta}\big\}\frac{\partial^2 u_0^\beta}{\partial x_\alpha \partial x_j}.
\end{aligned}
\right.
\end{equation}
Using Lemma \ref{lem.Dirichlet-corrector.estimate}, we have
\begin{equation}\label{ineq.Fhg}
\left\{
\begin{aligned}
&\norm{F_\e\delta }_{L^\infty(D_{6r})}+ \|h_\varepsilon\|_{L^\infty(D_{6r})}+\|g_\varepsilon\|_{L^\infty(D_{6r})} \le C\varepsilon\|\nabla^2 u_0\|_{L^\infty(D_{6r})}, \\
&[h_\varepsilon]_{C^{0,\eta}(D_{6r})}+[g_\varepsilon]_{C^{0,\eta}(D_{6r})}\le C\varepsilon^{1-\eta} [\nabla^2 u_0]_{C^{0,\eta}(D_{6r})} + C \varepsilon\|\nabla^2 u_0\|_{L^\infty(D_{6r})}.
\end{aligned}
\right.
\end{equation}

For the given $g_\varepsilon \in C^{0,\eta}(\widetilde{D})$ above, by Theorem \ref{thm.divergence.modifiedC1}, there exists a $W\in C^{1,\eta}(\widetilde{D};\mathbb{R}^d)$ such that $\text{div}(W)=g_\varepsilon$ in $\widetilde{D}$, with $W=0$ on $\partial\widetilde{D}$, and
\begin{equation}\label{ineq.W}
\left\{
\begin{aligned}
\|\nabla W\|_{L^\infty(\widetilde{D})}&\le C\ln[t^{-1}r+2]\|g_\varepsilon\|_{L^\infty(\widetilde{D})}+Ct^\eta[g_\varepsilon]_{C^{0,\eta}(\widetilde{D})},\\
[\nabla W]_{C^{0,\eta}(\widetilde{D})}& \le Cr^{-\eta}\|g_\varepsilon\|_{L^\infty(\widetilde{D})}+C[g_\varepsilon]_{C^{0,\eta}(\widetilde{D})}.
\end{aligned}
\right.
\end{equation}
Observe that $(w_\varepsilon- W,\tau_\varepsilon)$ satisfies
\begin{equation*}
\left\{
\begin{aligned}
\cL_\e(w_\varepsilon-W)+\nabla \tau_\varepsilon &= F_\varepsilon+ \text{div}(h_\varepsilon+A^\varepsilon\nabla W) & \text{ in }\widetilde{D},\\
\rdiv(w_\varepsilon- W) &=0 &\text{ in }\widetilde{D},\\
w_\varepsilon- W& =0 &\text{ on }\widetilde{\Delta}.\\
\end{aligned}
\right.	
\end{equation*}
Now we can apply Theorem \ref{thm.modified-boundary-estimate} to the above system to obtain
\begin{equation*}
\begin{aligned}
&\norm{\nabla (w_\varepsilon- W)}_{L^\infty(D_r)} +\osc{D_{r}}{\tau_\varepsilon}\\
& \le C \bigg\{ \frac{1}{r} \average_{D_{5r} } |w_\varepsilon-W| +\ln[\e^{-1}r+2]\|\mathcal{M}_{D_{5r},\e}(F_\varepsilon\delta)\|_{L^\infty(D_{5r})} +\e \bigg(\average_{D_{5r}}|F_\varepsilon|^p\bigg)^{1/p}\\  
&\qquad +\ln[\e^{-1}r+2]\|h_\varepsilon+A^\varepsilon\nabla W\|_{L^\infty(D_{5r})}+\e^{\eta}[h_\varepsilon+A^\varepsilon\nabla W]_{C^{0,\eta}(D_{5r})}\bigg\},
\end{aligned}
\end{equation*}
where we have assigned $t=\varepsilon$ for our situation. By (\ref{ineq.Fhg}) and Lemma \ref{lem.Dirichlet-corrector.estimate}, it is not hard to see that $\|\mathcal{M}_{D_{5r},\varepsilon}(F_\varepsilon\delta)\|_{L^\infty(D_{5r})}\le C\varepsilon\|\nabla^2 u_0\|_{L^\infty(D_{6r})}$ and $\|F_\varepsilon\|_{L^\infty(D_{5r})}\le C\|\nabla^2 u_0\|_{L^\infty(D_{5r})}$. It follows that
\begin{equation}\label{ineq.expansion.dGe.temp}
\begin{aligned}
&\norm{\nabla w_\varepsilon}_{L^\infty(D_r)} +\osc{D_{r}}{\tau_\varepsilon}\\
&\quad \le C \bigg\{ \frac{1}{r} \average_{D_{5r} } |w_\varepsilon| +\varepsilon\ln[\varepsilon^{-1}r+2]\|\nabla^2 u_0\|_{L^\infty(D_{5r})}+\varepsilon^{1+\eta}[\nabla^2 u_0]_{C^{0,\eta}(D_{5r})}\\  
&\quad\qquad +\|\nabla W\|_{L^\infty(\widetilde{D})}+\ln[\varepsilon^{-1}r+2]\|A^\varepsilon\nabla W\|_{L^\infty(\widetilde{D})}+\varepsilon^{\eta}[A^\varepsilon\nabla W]_{C^{0,\eta}(\widetilde{D})}\bigg\},\\
&\quad \le C \bigg\{ \frac{1}{r} \average_{D_{5r} } |u_\varepsilon-u_0| +\varepsilon\big(\ln[\varepsilon^{-1}r+2]\big)^2\|\nabla^2 u_0\|_{L^\infty(D_{6r})}+\varepsilon^{1+\eta}[\nabla^2 u_0]_{C^{0,\eta}(D_{6r})}\bigg\},\\
\end{aligned}
\end{equation}
where we have used (\ref{ineq.Fhg}), (\ref{ineq.W}) and the following observation
\begin{equation*}
\left\{
\begin{aligned}
\|A^\varepsilon \nabla W\|_{L^\infty(\widetilde{D})}&\le C\norm{\nabla W}_{L^\infty(\widetilde{D})},\\
[A^\varepsilon \nabla W]_{C^{0,\eta}(\widetilde{D})}&\le C\varepsilon^{-\eta} \norm{\nabla W}_{L^\infty(\widetilde{D})} + C[\nabla W]_{C^{0,\eta}(\widetilde{D})}.
\end{aligned}
\right.
\end{equation*}

Finally, by the already proved estimate (\ref{est.mainG}), we have $\|u_\varepsilon-u_0\|_{L^\infty(D_{6r})}\le C\varepsilon  r^{1-d}$. Also, the $C^{2,\eta}$ estimate for $G_0$ gives $\|\nabla^2 u_0\|_{_{L^\infty(D_{6r})}}\le Cr^{-d}$, and $[\nabla^2 u_0]_{C^{0,\eta}(D_{6r})}\le Cr^{-d-\eta}$. Thus, (\ref{ineq.expansion.dGe.temp}) implies that
\begin{equation}\label{ineq.expansion.dGe.temp2}
\begin{aligned}
&\Big\| \frac{\partial u_\varepsilon^\alpha}{\partial x_i}-\frac{\partial}{\partial x_i}\{\Phi_{\varepsilon,j}^{\alpha\beta}\}\cdot\frac{\partial u_0^\beta}{\partial x_j}\Big\|_{L^\infty(D_r)} +\osc{D_{r}}{p_\varepsilon-p_0-\Lambda_{\varepsilon,j}^\beta\frac{\partial u_0^\beta}{\partial x_j}}
\le C\varepsilon r^{-d}\big(\ln[\varepsilon^{-1}r+2]\big)^2.
\end{aligned}
\end{equation}
Note that (\ref{est.mainDG}) is an immediate consequence of (\ref{ineq.expansion.dGe.temp2}), while (\ref{est.mainPi}) follows from the same argument as in the proof of Theorem \ref{thm.green-function.pointwise}, by a covering argument.
\end{proof}

\begin{remark}
Let $(G^*_\varepsilon(x,y),\Pi^*_\varepsilon(x,y))$ be the Green's functions for the adjoint problem of (\ref{def.Dirichlet.Stokes}). By Theorem \ref{thm.main}, and the fact that $G^*_\varepsilon(x,y)=G_\varepsilon(y,x)^T$, we have the following asymptotic expansion,
\begin{equation}\label{ineq.expansion.dGe*}
\begin{aligned}
\left|\frac{\partial}{\partial y_i}\{G_\varepsilon^{\beta\alpha}(x,y)\}-\frac{\partial}{\partial y_i}\{\Phi_{\varepsilon,j}^{*\alpha\gamma}(y)\}\cdot\frac{\partial}{\partial y_j}\{G_0^{\beta\gamma}(x,y)\}\right|\le \frac{C\varepsilon(\ln[\varepsilon^{-1}|x-y|+2])^2}{|x-y|^d},
\end{aligned}
\end{equation}
and
\begin{equation}\label{ineq.expansion.Pie*}
\begin{aligned}
&\bigg|\Big[\Pi^{*\beta}_\varepsilon(x,y)-\Pi^{*\beta}_0(x,y)-\Lambda_{\varepsilon,j}^{*\gamma}(x)\frac{\partial}{\partial x_j}\{G^{\beta\gamma}_0(y,x)\}\Big]\\
&\qquad\qquad\qquad-\Big[\Pi^{*\beta}_\varepsilon(z,y)-\Pi^{*\beta}_0(z,y)-\Lambda_{\varepsilon,j}^{*\gamma}(z)\frac{\partial}{\partial z_j}\{G^{\beta\gamma}_0(y,z)\}\Big]\bigg|\\
&\qquad\le \frac{C\varepsilon(\ln[\varepsilon^{-1}|x-y|+2])^2}{|x-y|^d}+\frac{C\varepsilon(\ln[\varepsilon^{-1}|z-y|+2])^2}{|z-y|^d},
\end{aligned}
\end{equation}
for any $x,\,y\in \Omega$ and $x\neq y$, where $C$ depends only on $d$, $\eta$, $A$ and $\Omega$.
\end{remark}

\begin{corollary}\label{coro.PiExp}
Let $\Omega$ be a bounded $C^{2,\eta}$ domain and $A$ satisfy (\ref{cond.ellipticity}), (\ref{cond.periodicity}) and (\ref{cond.holder}). Then
\begin{equation*}
\begin{aligned}
\bigg|\Pi_\varepsilon(x,y)-\Pi_0(x,y)-\Big[\Lambda_{\varepsilon,j}^\beta(x)\frac{\partial}{\partial x_j}\{G^{\beta}_0(x,y)\}&-\average_{\Omega}\Lambda_{\varepsilon,j}^\beta(z)\frac{\partial}{\partial z_j}\{G^{\beta}_0(z,y)\} dz\Big]\bigg|\\
&\le \frac{C\varepsilon(\ln[\varepsilon^{-1}|x-y|+2])^3}{|x-y|^d},
\end{aligned}
\end{equation*}
for any $x,\,y\in \Omega$ and $x\neq y$, where $C$ depends only on $d$, $\eta$, $A$ and $\Omega$.
\end{corollary}

\begin{proof}
This follows from (\ref{est.mainPi}) and the assumption $\int_{\Omega} \Pi_\e(x,y)\,dx = 0$, as well as Theorem \ref{thm.green-function.pointwise}.
\end{proof}

\begin{theorem}\label{thm.convergence-rates.W1p}
Let $\Omega$ be a bounded $C^{2,\eta}$ domain with $R_0 = \rtxt{diam}(\Omega)$ and $A$ satisfy (\ref{cond.ellipticity}), (\ref{cond.periodicity}) and (\ref{cond.holder}). Let $1<p<\infty$, $F\in L^p(\Omega;\mathbb{R}^d)$ and $(u_\varepsilon,p_\varepsilon)\in W^{1,p}_0(\Omega;\mathbb{R}^d)\times L^p_0(\Omega)$ be the weak solution of 
\begin{equation*}
\left\{
\begin{aligned}
\mathcal{L}_\varepsilon(u_\varepsilon)+\nabla p_\varepsilon &=F &\quad\text{ in }\Omega,\\
\text{\rm div}(u_\varepsilon) & =0 &\quad\text{ in }\Omega,\\
u_\varepsilon &=0 &\quad\text{ on }\partial\Omega.
\end{aligned}
\right.
\end{equation*}
Then 
\begin{equation}\label{ineq.convergence-rates.W1p}
\begin{aligned}
\Big\|u_\varepsilon-u_0-\big\{\Phi_{\varepsilon,j}^\beta-P_j^\beta\big\}\cdot\frac{\partial u_0^\beta}{\partial x_j}\Big\|_{W^{1,p}_0(\Omega)}&+\Big\|p_\varepsilon-p_0-\Lambda_{\varepsilon,j}^\beta\frac{\partial u_0^\beta}{\partial x_j}\Big\|_{L^p_0(\Omega)}\\
&\qquad\le C\varepsilon(\ln[\varepsilon^{-1}R_0+2])^{8|\frac{1}{2}-\frac{1}{p}|}\|F\|_{L^p(\Omega)},
\end{aligned}
\end{equation}
where $C$ depends only on $d,\eta,p, A$ and $\Omega$.
\end{theorem}

\begin{proof}
Based on Lemma \ref{lem.Dirichlet-corrector.estimate} and the fact that $\|\nabla^2 u_0\|_{L^p(\Omega)}\le C\|F\|_{L^p(\Omega)}$ for $1< p < \infty$, it is sufficient to show for $1\le p\le \infty$ that
\begin{equation}\label{ineq.convergence-rates.dLp}
\begin{aligned}
\Big\|\frac{\partial u_\varepsilon}{\partial x_i}-\frac{\partial}{\partial x_i}\Phi_{\varepsilon,j}^\beta \frac{\partial u_0^\beta}{\partial x_j}\Big\|_{L^{p}(\Omega)}&+\Big\|p_\varepsilon-p_0-\Big[\Lambda_{\varepsilon,j}^\beta\frac{\partial u_0^\beta}{\partial x_j}-\average_\Omega \Lambda_{\varepsilon,j}^\beta(z)\frac{\partial u_0^\beta}{\partial z_j}dz \Big]\Big\|_{L^p(\Omega)}\\
&\qquad\le C\varepsilon(\ln[\varepsilon^{-1}R_0+2])^{8|\frac{1}{2}-\frac{1}{p}|}\|F\|_{L^p(\Omega)}.
\end{aligned}
\end{equation}
In fact, following by the integral representation (\ref{def.intRep.onlyF}) and Corollary \ref{coro.PiExp}, we obtain that
\begin{equation*}
\begin{aligned}
&\Big|\frac{\partial u_\varepsilon}{\partial x_i}-\frac{\partial}{\partial x_i} \Phi_{\varepsilon,j}^\beta(x) \frac{\partial u_0^\beta}{\partial x_j}\Big|+\Big|p_\varepsilon(x)-p_0(x)-\Big[\Lambda_{\varepsilon,j}^\beta(x)\frac{\partial u_0^\beta}{\partial x_j}-\average_\Omega \Lambda_{\varepsilon,j}^\beta(z)\frac{\partial u_0^\beta(z)}{\partial z_j}dz \Big]\Big|\\
&\le C\int_{D_\varepsilon(x)} \frac{|F(y)|}{|x-y|^{d-1}} dy + C\varepsilon\int_{\Omega \backslash D_\varepsilon(x)} \frac{(\ln[\varepsilon^{-1}|x-y|+2])^3|F(y)|}{|x-y|^{d}} dy.
\end{aligned}
\end{equation*}
This provides that left-hand side of (\ref{ineq.convergence-rates.dLp}) is bounded by 
$C\varepsilon(\ln[\varepsilon^{-1}R_0+2])^4\|F\|_{L^p(\Omega)}$ for the case $p=1$ and $\infty$. Thus, by the Riesz interpolation theorem, it remains to prove (\ref{ineq.convergence-rates.dLp}) for the case $p=2$. In order to do so, we let $(w_\varepsilon,\tau_\varepsilon)$ be defined the same as in (\ref{def.w.Green}).
By Lemma \ref{lem.w-system} and the energy estimate (\ref{ineq.energy}), we have
\begin{equation}\label{ineq.energy.wtau}
\begin{aligned}
\|\nabla w_\varepsilon\|_{L^2(\Omega)}+\|\tau_\varepsilon\|_{L^2_0(\Omega)}&\le C\|\nabla^2 u_0\|_{L^2(\Omega)}+C\varepsilon\|F_\varepsilon\|_{H^{-1}(\Omega)}\\
\end{aligned}
\end{equation}
where $F_\varepsilon$ is defined by (\ref{def.Fgh.e}). 

For any $\varphi\in H^1_0(\Omega)$, by Lemma \ref{lem.Dirichlet-corrector.estimate}, we have
$$
\begin{aligned}
\left|\int_\Omega F_\varepsilon(x) \varphi(x)dx\right|&\le C\e \int_\Omega \delta(x)^{-1}|\nabla^2 u_0(x)||\varphi(x)|dx\\
&\le C\e \|\delta(x)^{-1}\varphi\|_{L^2(\Omega)}\|\nabla^2 u_0\|_{L^2(\Omega)}\\
&\le C\e \|\nabla \varphi\|_{L^2(\Omega)}\|\nabla^2 u_0\|_{L^2(\Omega)},\\
\end{aligned}
$$
where we have used the Hardy's inequality in the second inequality. This implies that $\|F_\varepsilon\|_{H^{-1}(\Omega)}\le C\e \|\nabla^2 u_0\|_{L^2(\Omega)}$. Thus (\ref{ineq.energy.wtau}) leads to
\begin{equation*}
\begin{aligned}
\|\nabla w_\varepsilon\|_{L^2(\Omega)}+\|\tau_\varepsilon\|_{L^2_0(\Omega)}\le C\varepsilon\|\nabla^2 u_0\|_{L^2(\Omega)},
\end{aligned}
\end{equation*}
and hence the case $p=2$ now follows. This completes the proof.
\end{proof}

\section{Asymptotic Expansions of $\nabla_x \nabla_y G_\e, \nabla_y \Pi_\e$ and $(\Gamma_\e,Q_\e)$}
In this section, we will show the asymptotic expansions of $\nabla_x\nabla_y G_\varepsilon$ and $\nabla_y \Pi_\varepsilon$ in a bounded $C^{3,\eta}$ domain. We also state the corresponding results for the fundamental solutions without a concrete proof.

\begin{theorem}\label{thm.asymp-expansion.ddGe}
Let $\Omega$ be a bounded $C^{3,\eta}$ domain and $A$ satisfy (\ref{cond.ellipticity}), (\ref{cond.periodicity}) and (\ref{cond.holder}). Then

(i). Asymptotic expansion for $\nabla_x\nabla_y G_\varepsilon(x,y)$,
\begin{equation}\label{ineq.expansion.ddGe}
\begin{aligned}
\bigg|\frac{\partial^2}{\partial x_i\partial y_j}\{G_\varepsilon^{\alpha\beta}(x,y)\}-\frac{\partial}{\partial x_i}\{\Phi_{\varepsilon,k}^{\alpha\gamma}(x)\}\cdot&\frac{\partial^2}{\partial x_k\partial y_\ell}\{G_0^{\gamma\sigma}(x,y)\}\cdot\frac{\partial}{\partial y_j}\{\Phi_{\varepsilon,\ell}^{*\beta\sigma}(y)\}\bigg|\\
&\le \frac{C\varepsilon(\ln[\varepsilon^{-1}|x-y|+2])^2}{|x-y|^{d+1}},
\end{aligned}
\end{equation}

(ii). Asymptotic expansion for $\nabla_y \Pi_\varepsilon(x,y)$,
\begin{equation}\label{ineq.expansion.dPie}
\begin{aligned}
&\Bigg|\bigg[\frac{\partial}{\partial y_j}\{\Pi^\beta_\varepsilon(x,y)\}-\frac{\partial}{\partial y_j}\{\Phi_{\varepsilon,\ell}^{*\beta\sigma}(y)\}\frac{\partial}{\partial y_\ell}\{\Pi^\sigma_0(x,y)\}\\
&\qquad\qquad\qquad-\Lambda_{\varepsilon,k}^\gamma(x)\frac{\partial^2}{\partial x_k\partial y_\ell}\{G_0^{\gamma\sigma}(x,y)\}\cdot\frac{\partial}{\partial y_j}\{\Phi_{\varepsilon,\ell}^{*\beta\sigma}(y)\}\bigg]\\
&\quad-\bigg[\frac{\partial}{\partial y_j}\{\Pi^\beta_\varepsilon(z,y)\}-\frac{\partial}{\partial y_j}\{\Phi_{\varepsilon,\ell}^{*\beta\sigma}(y)\}\frac{\partial}{\partial y_\ell}\{\Pi^\sigma_0(z,y)\}\\
&\quad\qquad\qquad\qquad-\Lambda_{\varepsilon,k}^\gamma(z)\frac{\partial^2}{\partial z_k\partial y_\ell}\{G_0^{\gamma\sigma}(z,y)\}\cdot\frac{\partial}{\partial y_j}\{\Phi_{\varepsilon,\ell}^{*\beta\sigma}(y)\}\bigg]\Bigg|\\
&\le \frac{C\varepsilon(\ln[\varepsilon^{-1}|x-y|+2])^2}{|x-y|^{d+1}}+\frac{C\varepsilon(\ln[\varepsilon^{-1}|z-y|+2])^2}{|z-y|^{d+1}},
\end{aligned}
\end{equation}
for any $x,\,y\in \Omega$ and $x\neq y$, where $C$ depends only on $d$, $\eta$, $A$ and $\Omega$.
\end{theorem}

\begin{proof}
Fix $x_0, \, y_0\in \Omega$ and $r=|x_0-y_0|/8$. Fix $1\le j,\beta\le d$ and let $(u_\varepsilon(x),p_\varepsilon(x))=\big(\frac{\partial }{\partial y_j}G_\varepsilon^\beta(x,y_0),\frac{\partial }{\partial y_j}\Pi_\varepsilon^{\beta}(x,y_0)\big)$ and
\begin{equation*}
\begin{aligned}
u_0^\alpha(x)&=\frac{\partial}{\partial y_j}\{\Phi_{\varepsilon,\ell}^{*\beta\sigma}\}(y_0)\cdot\frac{\partial }{\partial y_\ell}G_0^{\alpha\sigma}(x,y_0),\\
p_0(x)&=\frac{\partial}{\partial y_j}\{\Phi_{\varepsilon,\ell}^{*\beta\sigma}\}(y_0)\frac{\partial }{\partial y_\ell}\Pi_0^{\sigma}(x,y_0),
\end{aligned}
\end{equation*}
in $D_{6r}=D_{6r}(x_0)$. Therefore, $(u_\varepsilon, p_\varepsilon)$ and $(u_0,p_0)$ satisfy the Stokes systems (\ref{def.Dirichlet.Stokes}) and (\ref{def.Dirichlet.Stokes0}), respectively, with all data vanishing in $D_{6r}$. Let $(w_\varepsilon,\tau_\varepsilon)$ be defined the same as (\ref{def.w.Green}). By a similar argument for (\ref{ineq.expansion.dGe.temp}), it is not hard to see that
\begin{equation}\label{ineq.expansion.ddGe.temp}
\begin{aligned}
&\norm{\nabla w_\varepsilon}_{L^\infty(D_r)} +\osc{D_{r}}{\tau_\varepsilon}\\
& \le C \bigg\{ \frac{1}{r} \average_{D_{6r} } |u_\varepsilon-u_0| +\varepsilon\big(\ln[\varepsilon^{-1}r+2]\big)^2\|\nabla^2 u_0\|_{L^\infty(D_{6r})}+\varepsilon^{1+\eta}[\nabla^2 u_0]_{C^{0,\eta}(D_{6r})}\bigg\},\\
\end{aligned}
\end{equation}
By (\ref{ineq.expansion.dGe*}), we obtain that
\begin{equation*}
\|u_\varepsilon-u_0\|_{L^\infty(D_{6r})}\le C\varepsilon r^{-d}\big(\ln[\varepsilon^{-1}r+2]\big)^2.
\end{equation*}
By the $C^{3,\eta}$ estimate of $G_0$ in a bounded $C^{3,\eta}$ domain, we have $\|\nabla^2 u_0\|_{L^\infty(D_{6r})}\le Cr^{-d-1}$, and $[\nabla^2 u_0]_{C^{0,\eta}(D_{6r})}\le Cr^{-d-1-\eta}$. Then (\ref{ineq.expansion.ddGe.temp}) implies that
\begin{equation}\label{ineq.expansion.ddGe.temp2}
\begin{aligned}
&\Big\| \frac{\partial u_\varepsilon^\alpha}{\partial x_i}-\frac{\partial}{\partial x_i}\{\Phi_{\varepsilon,j}^{\alpha\beta}\}\cdot\frac{\partial u_0^\beta}{\partial x_j}\Big\|_{L^\infty(D_r)} +\osc{D_{r}}{p_\varepsilon-p_0-\Lambda_{\varepsilon,j}^\beta\frac{\partial u_0^\beta}{\partial x_j}}
\le  \frac{C\varepsilon\big(\ln[\varepsilon^{-1}r+2]\big)^2}{r^{d+1}}
.
\end{aligned}
\end{equation}
Now (\ref{ineq.expansion.ddGe}) is a direct consequence of (\ref{ineq.expansion.ddGe.temp2}), while (\ref{ineq.expansion.dPie}) follows by a covering argument used in Theorem \ref{thm.green-function.pointwise}.
\end{proof}

For asymptotic expansions of the fundamental solution $(\Gamma_\e,Q_\e)$, since we deal with $\Omega = \R^d$, the Dirichlet correctors $(\Phi_{\e,j}^\beta,\Lambda_{\e,j}^\beta)$ in the expansion can be simply replaced by the usual correctors $(\chi_j^\beta,\pi_j^\beta)$. Precisely, we have

\begin{theorem}\label{thm.Fund.Exp}
	Let $A$ satisfy (\ref{cond.ellipticity}), (\ref{cond.periodicity}) and (\ref{cond.holder}). Let $(\Gamma_\e,Q_\e)$ and $(\Gamma_0,Q_0)$ be the fundamental solutions of (\ref{def.FundSol}) and the corresponding homogenized system, respectively. Then
	\begin{enumerate}
		\item[(i)] For any $x,y\in \R^d$,
		\begin{equation}
			|\Gamma_\e(x,y) - \Gamma_0(x,y)| \le \frac{C\e}{|x-y|^{d-1}}.
		\end{equation}
		
		\item[(ii)] For any $x,y\in \R^d$,
		\begin{equation}
			\begin{aligned}
				\bigg|\frac{\partial}{\partial x_i}\{\Gamma_\varepsilon^{\alpha\beta}(x,y)\}-\frac{\partial}{\partial x_i}\{ P_j^{\alpha\gamma}(x) + \e \chi_j^{\alpha\gamma}(x/\e) \}\cdot\frac{\partial}{\partial x_j}\{\Gamma_0^{\gamma\beta}(x,y)\}\bigg| \qquad \quad \\
				\le \frac{C\varepsilon(\ln[\varepsilon^{-1}|x-y|+2])^2}{|x-y|^d},
			\end{aligned}
		\end{equation}
		and for any $x,z,y\in \R^d$,
		\begin{equation}
			\begin{aligned}
				&\bigg|\Big[Q^\beta_\varepsilon(x,y)-Q^\beta_0(x,y)-\pi_{j}^\gamma(x/\e)\frac{\partial}{\partial x_j}\{\Gamma^{\gamma\beta}_0(x,y)\}\Big]\\
				&\qquad\qquad-\Big[Q^\beta_\varepsilon(z,y)-Q^\beta_0(z,y)-\pi_{j}^\gamma(z/\e)\frac{\partial}{\partial z_j}\{\Gamma^{\gamma\beta}_0(z,y)\}\Big]\bigg|\\
				&\qquad\le \frac{C\varepsilon(\ln[\varepsilon^{-1}|x-y|+2])^2}{|x-y|^d}+\frac{C\varepsilon(\ln[\varepsilon^{-1}|z-y|+2])^2}{|z-y|^d}.
			\end{aligned}
		\end{equation}
		
		\item[(iii)] For any $x,y\in \R^d$,
		\begin{equation}
		\begin{aligned}
		&\bigg|\frac{\partial^2}{\partial x_i\partial y_j}\{\Gamma_\varepsilon^{\alpha\beta}(x,y)\}\\
		&\quad -\frac{\partial}{\partial x_i}\{ P_k^{\alpha\gamma}(x) + \e\chi_k^{\alpha\gamma}(x/\e) \}\cdot\frac{\partial^2}{\partial x_k\partial y_\ell}\{\Gamma^{\gamma\sigma}(x,y)\}\cdot\frac{\partial}{\partial y_j}\{P_\ell^{\beta\sigma}(y) + \e\chi_\ell^{*\beta\sigma}(y/\e) \}\bigg|\\
		&\qquad \qquad \le \frac{C\varepsilon(\ln[\varepsilon^{-1}|x-y|+2])^2}{|x-y|^{d+1}},
		\end{aligned}
		\end{equation}
		and
		\begin{equation}
		\begin{aligned}
		&\Bigg|\bigg[\frac{\partial}{\partial y_j}\{Q^\beta_\varepsilon(x,y)\}-\frac{\partial}{\partial y_j}\{P_\ell^{\beta\sigma}(y) + \e\chi_\ell^{*\beta\sigma}(y/\e) \}\frac{\partial}{\partial y_\ell}\{Q^\sigma_0(x,y)\}\\
		&\qquad\qquad\qquad-\pi_{k}^\gamma(x/\e)\frac{\partial^2}{\partial x_k\partial y_\ell}\{\Gamma_0^{\gamma\sigma}(x,y)\}\cdot\frac{\partial}{\partial y_j}\{P_\ell^{\beta\sigma}(y) + \e\chi_\ell^{*\beta\sigma}(y/\e) \}\bigg]\\
		&\quad-\bigg[\frac{\partial}{\partial y_j}\{Q^\beta_\varepsilon(z,y)\}-\frac{\partial}{\partial y_j}\{P_\ell^{\beta\sigma}(y) + \e\chi_\ell^{*\beta\sigma}(y/\e)\}\frac{\partial}{\partial y_\ell}\{Q^\sigma_0(z,y)\}\\
		&\quad\qquad\qquad\qquad-\pi_{k}^\gamma(z/\e)\frac{\partial^2}{\partial z_k\partial y_\ell}\{\Gamma_0^{\gamma\sigma}(z,y)\}\cdot\frac{\partial}{\partial y_j}\{P_\ell^{\beta\sigma}(y) + \e\chi_\ell^{*\beta\sigma}(y/\e)\}\bigg]\Bigg|\\
		&\le \frac{C\varepsilon(\ln[\varepsilon^{-1}|x-y|+2])^2}{|x-y|^{d+1}}+\frac{C\varepsilon(\ln[\varepsilon^{-1}|z-y|+2])^2}{|z-y|^{d+1}}.
		\end{aligned}
		\end{equation}
	\end{enumerate}
	The constant $C$ depends only on $d$ and $A$.
\end{theorem}

\section*{Acknowledgement}

Both authors would like to thank Professor Zhongwei Shen for insightful conversations and discussions regarding this work. The second author is supported in part by National Science Foundation grant DMS-1600520.


\appendix
\section{Proof of Theorem \ref{thm.convergence-rates}}
Theorem \ref{thm.convergence-rates} is a straightforward corollary of \cite[Theorem 3.1]{GuXu17} and an argument of interpolation. We provide a proof here for readers' convenience. Since we assume $A$ is H\"older continuous, we have $\|\chi\|_{L^\infty(Y)}\le C$. Then \cite[Theorem 3.1]{GuXu17} reads as follows.
\begin{theorem}\label{thm.C1Rate}
	Suppose $\Omega$ is a bounded $C^1$ domain and $A$ satisfies (\ref{cond.ellipticity}), (\ref{cond.periodicity}) and (\ref{cond.holder}). Assume $F \in L^p(\Omega;\R^d)$, $g\in W^{1,p}(\Omega)$ for some $p>d$ and $f\in C^{0,1}(\partial\Omega;\R^d)$ satisfying the compatibility (\ref{cond.compatibility}). Let $(u_\e,p_\e)$ and $(u_0,p_0)$ be the solutions of (\ref{def.Dirichlet.Stokes}) and (\ref{def.Dirichlet.Stokes0}), respectively. Then for any $\sigma \in (0,1/2)$,
	\begin{align}\label{est.C1rate}
	\begin{aligned}
	\norm{u_\e - u_0}_{L^2(\Omega)} &+ \norm{p_\e - p_0 - \pi^\e \nabla u_0}_{L^2_0(\Omega)} \\
	&\qquad \le C\e^\sigma \Big(\norm{F}_{L^p(\Omega)} + \norm{g}_{W^{1,p}(\Omega)} + \norm{f}_{C^{0,1}(\partial\Omega)} \Big),
	\end{aligned}
	\end{align}
	where $C$ depends only on $d,p, \sigma, A$ and $\Omega$.
\end{theorem}

\begin{proof}[Proof of Theorem \ref{thm.convergence-rates}]
	Since $\Omega$ is $C^1$, $g$ and $h$ can be extended to the whole space $\R^d$ by $\widetilde{g}$ and $\widetilde{h}$, respectively, such that $\widetilde{g} = g,\  \widetilde{h} = g$ and $\norm{\widetilde{g}}_{C^{0,\eta}(\R^d)} \le C \norm{g}_{C^{0,\eta}(\Omega)}, \ \norm{\widetilde{h}}_{C^{0,\eta}(\R^d)} \le C \norm{h}_{C^{0,\eta}(\Omega)}$. Also, $f$ can be extended to a function $\tilde{f} \in H^{3/2}(\R^d;\R^d)$ so that $\widetilde{f} = f$ on $\partial\Omega$ and $\norm{\widetilde{f}}_{H^{3/2}(\R^d)} \le C\norm{f}_{H^1(\partial\Omega)}$. Let $\varphi \in C_0^\infty(\R^d)$ be a smooth cut-off function with $\int_{\R^d} \varphi = 1$. Let $\varphi_r(x) = r^{-d}\varphi(x/r)$, where $r\in (0,1)$ is to be determined. Define $\widetilde{h}_r = \varphi_r* \widetilde{h},\  \widetilde{g}_r = \varphi_r* \widetilde{g}$ and $\widetilde{f}_r = \varphi_r* \widetilde{f} $. Then it is not hard to verify that
	\begin{equation}\label{est.hr-h}
	\norm{\widetilde{h}_r - h}_{L^2(\Omega)} \le Cr^\eta \norm{h}_{C^{0,\eta}(\Omega)}, \quad \norm{\widetilde{g}_r - g}_{L^2(\Omega)} \le Cr^\eta \norm{g}_{C^{0,\eta}(\Omega)},
	\end{equation}
	and
	\begin{equation}\label{est.hr}
	\norm{\widetilde{h}_r}_{W^{1,p}(\Omega)} \le C r^{\eta-1} \norm{h}_{C^{0,\eta}(\Omega)}, \quad \norm{\widetilde{g}_r}_{W^{1,p}(\Omega)} \le C r^{\eta-1} \norm{g}_{C^{0,\eta}(\Omega)}.
	\end{equation}
	Also, we have
	\begin{equation}\label{est.fr}
	\norm{\widetilde{f}_r - f}_{H^{1/2}(\partial\Omega)} \le Cr^{1/2} \norm{f}_{H^1(\partial\Omega)}, \quad \norm{\widetilde{f}_r}_{C^{0,1}(\partial\Omega)} \le Cr^{(1-d)/2} \norm{f}_{H^1(\partial\Omega)}.
	\end{equation}
	Set
	\begin{equation*}
	C_r = \frac{1}{|\Omega|} \bigg( \int_{\Omega} \widetilde{g}_r - \int_{\partial\Omega} \widetilde{f}_r\cdot n\bigg) = \frac{1}{|\Omega|} \bigg( \int_{\Omega} \widetilde{g}_r - \int_{\Omega} g - \int_{\partial\Omega} \widetilde{f}_r\cdot n + \int_{\partial\Omega} f\cdot n \bigg).
	\end{equation*}
	Observe that (\ref{est.hr-h}) and (\ref{est.fr}) imply
	\begin{equation}\label{est.Cr}
	|C_r| \le Cr^{\eta} \Big(\norm{g}_{C^{0,\eta}(\Omega)} + \norm{f}_{H^1(\partial\Omega)} \Big).
	\end{equation}
	
	Now let $(u_{\e,r}, p_{\e,r})$ and $(u_{0,r}, p_{0,r})$ be solutions of
	\begin{equation*}
	\left\{
	\begin{aligned}
	\cL_\e(u_{\e,r})+\nabla p_{\e,r} &= F + \tdiv(\widetilde{h}_r) &\qquad \text{ in }\Omega,\\
	\tdiv(u_{\e,r}) &=\widetilde{g}_r - C_r &\qquad\text{ in }\Omega,\\
	u_{\e,r}& =\tilde{f}_r &\qquad\text{ on }\partial\Omega,
	\end{aligned}
	\right.	
	\end{equation*}
	and
	\begin{equation*}
	\left\{
	\begin{aligned}
	\cL_0(u_{0,r})+\nabla p_{0,r} &= F + \tdiv(\widetilde{h}_r) &\qquad \text{ in }\Omega,\\
	\tdiv(u_{0,r}) &=\widetilde{g}_r - C_r &\qquad\text{ in }\Omega,\\
	u_{0,r}& =\widetilde{f}_r &\qquad\text{ on }\partial\Omega.
	\end{aligned}
	\right.	
	\end{equation*}
	Then it follows from Theorem \ref{thm.C1Rate} that
	\begin{align}\label{est.ur-u0}
	\begin{aligned}
	&\norm{u_{\e,r} - u_{0,r}}_{L^2(\Omega)}  + \norm{p_{\e,r} - p_{0,r} - \pi^\e \nabla u_{0,r}}_{L^2_0(\Omega)} \\
	& \qquad \le C\e^\sigma \Big(\norm{F}_{L^p(\Omega)} + \norm{\widetilde{h}_r}_{W^{1,p}(\Omega)} +  \norm{\widetilde{g}_r - C_r}_{W^{1,p}(\Omega)} + \norm{f}_{H^1(\partial\Omega)} \Big) \\
	& \qquad \le C\e^\sigma r^{(1-d)/2} \Big(\norm{F}_{L^p(\Omega)} + \norm{h}_{C^{0,\eta}(\Omega)} + \norm{g}_{C^{0,\eta}(\Omega)} + \norm{f}_{H^1(\partial\Omega)} \Big).
	\end{aligned}
	\end{align}
	
	On the other hand, it is obvious to see $u_{\e,r} - u_\e$ satisfies
	\begin{equation*}
	\left\{
	\begin{aligned}
	\cL_\e(u_{\e,r} - u_\e)+\nabla (p_{\e,r} - p_\e) &= \tdiv(\widetilde{h}_r - h) &\qquad \text{ in }\Omega,\\
	\tdiv(u_{\e,r} - u_\e) &=\widetilde{g}_r - g - C_r &\qquad\text{ in }\Omega,\\
	u_{\e,r}-u_\e& = \widetilde{f}_r - f &\qquad\text{ on }\partial\Omega,
	\end{aligned}
	\right.	
	\end{equation*}
	Then applying the energy estimate (\ref{ineq.energy}), we have
	\begin{align}\label{est.uer-ue}
	\begin{aligned}
	&\norm{u_{\e,r} - u_\e}_{H^1(\Omega)} + \norm{p_{\e,r} - p_\e}_{L^2_0(\Omega)} \\
	&\qquad \le C\Big( \norm{\widetilde{h}_r - h}_{L^2(\Omega)} + C\norm{\widetilde{g}_r - g - C_r}_{L^2(\Omega)} + \norm{\widetilde{f}_r - f}_{H^{1/2}(\partial\Omega)} \Big) \\
	&\qquad \le Cr^\eta \Big( \norm{h}_{C^{0,\eta}(\Omega)} + \norm{g}_{C^{0,\eta}(\Omega)} + \norm{f}_{H^1(\partial\Omega)} \Big).
	\end{aligned}
	\end{align}
	Similarly, we also have
	\begin{equation}\label{est.u0r-u0}
	\norm{u_{0,r} - u_0}_{H^1(\Omega)} + \norm{p_{0,r} - p_0}_{L^2_0(\Omega)} \le Cr^\eta \Big( \norm{h}_{C^{0,\eta}(\Omega)} + \norm{g}_{C^{0,\eta}(\Omega)}\Big).
	\end{equation}
	Combining this with (\ref{est.ur-u0}), we obtain
	\begin{align*}
	&\norm{u_{\e} - u_{0}}_{L^2(\Omega)}  + \norm{p_{\e} - p_{0} - \pi^\e \nabla u_{0}}_{\dot{L}^2(\Omega)} \\
	& \qquad \le C\big(\e^\sigma r^{(1-d)/2} + r^\eta \big)  \Big(\norm{F}_{L^p(\Omega)} + \norm{h}_{C^{0,\eta}(\Omega)} + \norm{g}_{C^{0,\eta}(\Omega)} + \norm{f}_{H^1(\partial\Omega)} \Big) \\
	& \qquad \le C\e^\gamma  \Big(\norm{F}_{L^p(\Omega)} + \norm{h}_{C^{0,\eta}(\Omega)} + \norm{g}_{C^{0,\eta}(\Omega)} + \norm{f}_{H^1(\partial\Omega)} \Big),
	\end{align*}
	where we have chosen $r = \e^{2\sigma/d}$ in the last inequality and hence $\gamma = 2\sigma\eta/d$. Since $\sigma \in (0,1/2)$ is arbitrary, $\gamma$ can be specified arbitrarily close to $\eta/d$.
\end{proof}
\bibliographystyle{amsplain}
\bibliography{Lib}

\providecommand{\bysame}{\leavevmode\hbox to3em{\hrulefill}\thinspace}
\providecommand{\MR}{\relax\ifhmode\unskip\space\fi MR }
\providecommand{\MRhref}[2]{%
  \href{http://www.ams.org/mathscinet-getitem?mr=#1}{#2}
}
\providecommand{\href}[2]{#2}
\begin{thebibliography}{10}

\bibitem{AllaireGhoshVanninathan17}
G.~Allaire, T.~Ghosh, and M.~Vanninathan, \emph{Homogenization of {S}tokes
  system using {B}loch waves}, Netw. Heterog. Media \textbf{12} (2017), no.~4,
  525--550. \MR{3714981}

\bibitem{ArmstrongKuusiMourratPrange17}
S.~Armstrong, T.~Kuusi, J.~Mourrat, and C.~Prange, \emph{Quantitative
  {A}nalysis of {B}oundary {L}ayers in {P}eriodic {H}omogenization}, Arch.
  Ration. Mech. Anal. \textbf{226} (2017), no.~2, 695--741. \MR{3687879}

\bibitem{AL8701}
M.~Avellaneda and F.~Lin, \emph{{Compactness methods in the theory of
  homogenization}}, Comm. Pure Appl. Math. \textbf{40} (1987), no.~6, 803--847.

\bibitem{AL89}
\bysame, \emph{{Compactness methods in the theory of homogenization II:
  Equations in non-divergence form}}, Comm. Pure Appl. Math. \textbf{42}
  (1989), no.~2, 139--172.

\bibitem{AL8902}
\bysame, \emph{Homogenization of {P}oisson's kernel and applications to
  boundary control}, J. Math. Pures Appl. (9) \textbf{68} (1989), no.~1, 1--29.
  \MR{985952}

\bibitem{AL91}
\bysame, \emph{{$L^p$ bounds on singular integrals in homogenization}}, Comm.
  Pure Appl. Math. \textbf{44} (1991), no.~8-9, 897--910.

\bibitem{Lions}
A.~Bensoussan, J.-L. Lions, and G.~Papanicolaou, \emph{Asymptotic analysis for
  periodic structures}, AMS Chelsea Publishing, Providence, RI, 2011, Corrected
  reprint of the 1978 original [MR0503330]. \MR{2839402}

\bibitem{BourgainBrezis03}
J.~Bourgain and H.~Brezis, \emph{On the equation {${\rm div}\, Y=f$} and
  application to control of phases}, J. Amer. Math. Soc. \textbf{16} (2003),
  no.~2, 393--426. \MR{1949165}

\bibitem{Briane03}
M.~Briane, \emph{Homogenization of the {S}tokes equations with high-contrast
  viscosity}, J. Math. Pures Appl. (9) \textbf{82} (2003), no.~7, 843--876.
  \MR{2005297}

\bibitem{ChoiLee15}
J.~Choi and K.~Lee, \emph{{The Green function for the Stokes system with
  measurable coefficients}}, arXiv:1503.07290v4 (2015).

\bibitem{ChoiYang17}
J.~Choi and M.~Yang, \emph{{Fundamental solutions for stationary Stokes systems
  with measurable coefficients}}, arXiv:1705.02736v1 (2017).

\bibitem{FabesStroock84}
E.~B. Fabes and D.~W. Stroock, \emph{The {$L^p$}-integrability of {G}reen's
  functions and fundamental solutions for elliptic and parabolic equations},
  Duke Math. J. \textbf{51} (1984), no.~4, 997--1016. \MR{771392}

\bibitem{GengShen1702}
J.~Geng and Z.~Shen, \emph{{Asymptotic Expansions of Fundamental Solutions in
  Parabolic Homogenization}}, arXiv:1711.10638 (2017).

\bibitem{Gerard-VaretMasmoudi12}
D.~G\'erard-Varet and N.~Masmoudi, \emph{Homogenization and boundary layers},
  Acta Math. \textbf{209} (2012), no.~1, 133--178. \MR{2979511}

\bibitem{GiaquintaModica82}
M.~Giaquinta and G.~Modica, \emph{{Nonlinear systems of the type of the
  stationary Navier-Stokes system}}, J. Reine Angew. Math. \textbf{330} (1982),
  173--214.

\bibitem{GilbargTrudinger83}
D.~Gilbarg and N.~S. Trudinger, \emph{Elliptic partial differential equations
  of second order}, second ed., Grundlehren der Mathematischen Wissenschaften
  [Fundamental Principles of Mathematical Sciences], vol. 224, Springer-Verlag,
  Berlin, 1983. \MR{737190}

\bibitem{Gu16}
S.~Gu, \emph{{Convergence rates in homogenization of Stokes systems}}, J.
  Differential Equations \textbf{260} (2016), no.~7, 5796--5815.

\bibitem{GuShen15}
S.~Gu and Z.~Shen, \emph{{Homogenization of Stokes systems and uniform
  regularity estimates}}, SIAM J. Math. Anal. \textbf{47} (2015), no.~5,
  4025--4057.

\bibitem{GuXu17}
S.~Gu and Q.~Xu, \emph{Optimal {B}oundary {E}stimates for {S}tokes {S}ystems in
  {H}omogenization {T}heory}, SIAM J. Math. Anal. \textbf{49} (2017), no.~5,
  3831--3853. \MR{3706915}

\bibitem{Gurtin}
M.~E. Gurtin, E.~Fried, and L.~Anand, \emph{The mechanics and thermodynamics of
  continua}, Cambridge University Press, Cambridge, 2010. \MR{2884384}

\bibitem{JikovKozlovOleinik94}
V.~V. Jikov, S.~M. Kozlov, and O.~A. Ole{\u\i}nik, \emph{Homogenization of
  differential operators and integral functionals}, Springer-Verlag, Berlin,
  1994. \MR{1329546}

\bibitem{KenigLinShen1301}
C.~E. Kenig, F.~Lin, and Z.~Shen, \emph{{Homogenization of elliptic systems
  with Neumann boundary conditions}}, J. Amer. Math. Soc. \textbf{26} (2013),
  no.~4, 901--937.

\bibitem{KenigLinShen14}
\bysame, \emph{Periodic homogenization of {G}reen and {N}eumann functions},
  Comm. Pure Appl. Math. \textbf{67} (2014), no.~8, 1219--1262. \MR{3225629}

\bibitem{KenigShen1101}
C.~E. Kenig and Z.~Shen, \emph{{Layer potential methods for elliptic
  homogenization problems}}, Comm. Pure Appl. Math. \textbf{64} (2011), no.~1,
  1--44.

\bibitem{NiuShenXu17}
W.~Niu, Z.~Shen, and Y.~Xu, \emph{{Convergence Rates and Interior Estimates in
  Homogenization of Higher Order Elliptic Systems}}, arXiv:1706.02072 (2017).

\bibitem{Sevostjanova82}
E.~V. Sevost'janova, \emph{{An asymptotic expansion of the solution of a second
  order elliptic equation with periodic rapidly oscillating coefficients}},
  Mathematics of the USSR-Sbornik \textbf{43} (1982), no.~2, 181.

\bibitem{Shen17}
Z.~Shen, \emph{Boundary estimates in elliptic homogenization}, Anal. PDE
  \textbf{10} (2017), no.~3, 653--694. \MR{3641883}

\bibitem{ShenLec17}
\bysame, \emph{{Lectures on Periodic Homogenization of Elliptic Systems}},
  arXiv:1710.11257 (2017).

\bibitem{ShenZhuge1702}
Z.~Shen and J.~Zhuge, \emph{{Boundary Layers in Periodic Homogenization of
  Neumann Problems}}, Comm. Pure Appl. Math., to appear (2017).

\bibitem{terElstRobinsonSikora01}
A.~F.~M. ter Elst, D.~W. Robinson, and A.~Sikora, \emph{On second-order
  periodic elliptic operators in divergence form}, Math. Z. \textbf{238}
  (2001), no.~3, 569--637. \MR{1869699}

\bibitem{WeiZhang14}
W.~Wei and Z.~Zhang, \emph{{$L^p$} resolvent estimates for constant coefficient
  elliptic systems on {L}ipschitz domains}, J. Funct. Anal. \textbf{267}
  (2014), no.~9, 3262--3293. \MR{3261110}

\bibitem{Xu17}
Q.~Xu, \emph{{Convergence rates and $W^{1,p}$ estimates in homogenization
  theory of Stokes systems in Lipschitz domains}}, J. Differential Equations
  \textbf{263} (2017), no.~1, 398--450. \MR{3631311}

\bibitem{Zhuge16}
J.~Zhuge, \emph{{Homogenization and boundary layers in domains of finite
  type}}, arXiv:1612.05383v2 (2016).

\end{thebibliography}

\medskip

\begin{flushleft}
Shu Gu\\
Department of Mathematics,\\
The Florida State University,\\
Tallahassee, FL 32306,
USA.\\
E-mail: \email{gu@math.fsu.edu}
\end{flushleft}

\begin{flushleft}
Jinping Zhuge\\
Department of Mathematics,\\
University of Kentucky,\\
Lexington, KY 40506,
USA.\\
E-mail: \email{jinping.zhuge@uky.edu}
\end{flushleft}

\medskip

\noindent \today
\end{document}